\documentclass[10pt]{article}

\usepackage{url}
\usepackage{mathtools}
\usepackage{amssymb}
\usepackage{amsthm}
\usepackage{empheq}
\usepackage{latexsym}
\usepackage{enumitem}
\usepackage{eurosym}
\usepackage{dsfont}
\usepackage{yfonts}
\usepackage{appendix}
\usepackage{color} 
\usepackage{aliascnt}
\usepackage[unicode]{hyperref}
\usepackage{frcursive}
\usepackage[utf8]{inputenc}
\usepackage[T1]{fontenc}
\usepackage{geometry}
\usepackage{multirow}
\usepackage[colorinlistoftodos]{todonotes}
\usepackage{lmodern}
\usepackage{anyfontsize}
\usepackage{stmaryrd}
\usepackage{natbib}
\usepackage{cleveref}

\usepackage{amsbsy}

\usepackage{comment}

\setcitestyle{numbers,open={[},close={]}}

\definecolor{red}{rgb}{0.7,0.15,0.15}
\definecolor{green}{rgb}{0,0.5,0}
\definecolor{blue}{rgb}{0,0,0.7}
\hypersetup{colorlinks, linkcolor={red},citecolor={green}, urlcolor={blue}}
			
\makeatletter \@addtoreset{equation}{section}

\newtheorem{theorem}{Theorem}[section]
\crefname{theorem}{theorem}{theorems}
\Crefname{theorem}{Theorem}{Theorems}

\newaliascnt{assumption}{theorem}
\newtheorem{assumption}[assumption]{Assumption}
\aliascntresetthe{assumption}
\crefname{assumption}{assumption}{assumptions}
\Crefname{assumption}{Assumption}{Assumptions}

\newaliascnt{corollary}{theorem}
\newtheorem{corollary}[corollary]{Corollary}
\aliascntresetthe{corollary}
\crefname{corollary}{corollary}{corollaries}
\Crefname{corollary}{Corollary}{Corollaries}

\newaliascnt{example}{theorem}

\aliascntresetthe{example}
\crefname{example}{example}{examples}
\Crefname{example}{Example}{Examples}

\newaliascnt{exercise}{theorem}

\aliascntresetthe{exercise}
\crefname{exercise}{exercise}{exercises}
\Crefname{exercise}{Exercise}{Exercises}

\newaliascnt{lemma}{theorem}
\newtheorem{lemma}[lemma]{Lemma}
\aliascntresetthe{lemma}
\crefname{lemma}{lemma}{lemmas}
\Crefname{lemma}{Lemma}{Lemmas}

\newaliascnt{proposition}{theorem}
\newtheorem{proposition}[proposition]{Proposition}
\aliascntresetthe{proposition}
\crefname{proposition}{proposition}{propositions}
\Crefname{proposition}{Proposition}{Propositions}

\newaliascnt{condition}{theorem}

\aliascntresetthe{condition}
\crefname{condition}{condition}{conditions}
\Crefname{condition}{Condition}{Conditions}

\newaliascnt{definition}{theorem}
\newtheorem{definition}[definition]{Definition}
\aliascntresetthe{definition}
\crefname{definition}{definition}{definitions}
\Crefname{definition}{Definition}{Definitions}

\newaliascnt{remark}{theorem}
\newtheorem{remark}[remark]{Remark}
\aliascntresetthe{remark}
\crefname{remark}{remark}{remarks}
\Crefname{remark}{Remark}{Remarks}

\def\no{\noindent}
\def\beq{\begin{eqnarray}}
\def\eeq{\end{eqnarray}}
\def\be*{\begin{eqnarray*}}
\def\ee*{\end{eqnarray*}}

\def \E{\mathbb{E}}
\def \F{\mathbb{F}}
\def \G{\mathbb{G}}
\def \H{\mathbb{H}}

\def \M{\mathbb{M}}
\def \N{\mathbb{N}}

\def \P{\mathbb{P}}

\def \R{\mathbb{R}}
\def \S{\mathbb{S}}

\def \U{\mathbb{U}}

\def \X{\mathbb{X}}

\def \Pr{\mathrm{P}}

\def\Ac{{\cal A}}

\def\Cc{{\cal C}}

\def\Fc{{\cal F}}
\def\Gc{{\cal G}}
\def\Hc{{\cal H}}

\def\Jc{{\cal J}}

\def\Lc{{\cal L}}

\def\Pc{{\cal P}}

\def\Rc{{\cal R}}
\def\Sc{{\cal S}}

\def\Uc{{\cal U}}
\def\Vc{{\cal V}}
\def\Wc{{\cal W}}

\def\Pb{{\overline \P}}

\def\x{\times}

\def\Om{\Omega}

\def\om{\omega}

\def\Gch{\widehat{\Gc}}

\def\0{\mathbf{0}}

\def \mub{\overline{\mu}}

\def \muh{\widehat{\mu}}

\def \mut{\widetilde{\mu}}
\def \nub{\overline{\nu}}

\def\normeL2#1{\left\|{#1}\right\|_{L^2}}

\def\Vh{\widehat V}

\def \Lim{\displaystyle\lim}

\def \alphab {\boldsymbol{\alpha}}

\def \Xbb{\mathbf{X}}

\def \Ybb{\mathbf{Y}}

\def \Qr{\mathrm{Q}}

\def \1{\mathds{1}}

\def \alphab {\boldsymbol{\alpha}}
\def \betab {\boldsymbol{\beta}}

\def \xbb {\boldsymbol{x}}

\def \Xbb{\mathbf{X}}

\def \Ybb{\mathbf{Y}}

\def\Er{{\rm E}}
\def\Gr{{\rm G}}
\def\Ir{{\rm I}}

\def\Rr{{\rm R}}
\def\Kr{{\rm K}}

\def \Qr{\mathrm{Q}}

\def \Rr{\mathrm{R}}

\DeclareUnicodeCharacter{014D}{\=o}
 \title{ Non--exchangeable mean field games with moderate interactions and common noise
 
 }

\author{
 Mao Fabrice Djete\footnote{\'Ecole Polytechnique Paris, Centre de Math\'ematiques Appliqu\'ees, mao-fabrice.djete@polytechnique.edu. This work benefits from the financial support of the Chairs {\it Financial Risk} and {\it Finance and Sustainable Development}} 
    }
             \date{\today}

\begin{document}

\maketitle
 
\begin{abstract}
We study mean field games for large non--exchangeable populations with moderate local interactions and common noise. The finite--player system is driven by two complementary interaction mechanisms : a graphon--type structure, which encodes heterogeneous large--scale interactions between agents, and a rescaled local kernel, which produces a density-dependent interaction term in the limit. The limiting model is a non--exchangeable mean field game in which the representative player is indexed by a label \(u\in[0,1]\), interacts through a graphon--weighted local density, and is affected by a graphon--induced environment law.

We introduce a relaxed formulation of the limiting mean field game, adapted to the presence of common noise, and prove existence under general continuity and non--degeneracy assumptions. Under additional convexity assumptions, relaxed equilibria can be realized in strict form. In the deterministic case without common noise, we obtain deterministic equilibria and provide a probabilistic characterization of strict equilibria through a nonlinear Feynman--Kac representation.

We then establish the asymptotic connection with the finite--player game. We prove that every limit point of approximate closed--loop Nash equilibria is a relaxed solution of the limiting mean field game, and that the corresponding averaged equilibrium payoffs converge. Conversely, every relaxed mean field game equilibrium can be approximated by Markovian approximate Nash equilibria of the finite--player systems. These results give a complete asymptotic characterization of equilibrium behavior for non--exchangeable games with moderate interactions and common noise.
\end{abstract}

\vspace{3mm}

\vspace{3mm}
\no{\bf MSC2010.} 60K35, 60H30, 91A13, 91A23, 91B30.

\section{Introduction}\label{sec:intro}

Classical mean field game (MFG) theory provides a powerful framework for analyzing strategic interactions in large populations of agents. In the standard formulation, each player interacts with the rest of the population only through the empirical distribution of states (and, in some variants, controls), which leads in the large population limit to a McKean--Vlasov type equilibrium model. This symmetric dependence implies that players are exchangeable in law and is one of the fundamental structural features underlying the classical theory. It is precisely this exchangeability that makes the limit tractable and leads to elegant descriptions in terms of coupled stochastic equations, forward--backward systems, or partial differential equations on the space of measures. We refer to the seminal works of \citeauthor*{lasry2007mean}~\cite{lasry2007mean}, \citeauthor*{huang2006large}~\cite{huang2006large,huang2007nash}, and to the monograph of \citeauthor*{carmona2018probabilisticI} \cite{carmona2018probabilisticI,carmona2018probabilisticII}.

\medskip

However, many systems of practical and theoretical interest fall outside the exchangeable paradigm. In numerous applications, agents are heterogeneous and interact through structured networks, spatially localized couplings, or asymmetric environments. Examples include social or communication networks, interacting financial institutions with heterogeneous exposures, multi--class population models, and large decentralized systems in which the influence of an agent depends not only on the empirical distribution of the population, but also on its position inside a non-homogeneous interaction architecture. In such situations, the classical mean field approximation based solely on a symmetric empirical measure can fail to capture the relevant interaction patterns.

\medskip

Motivated by these considerations, a growing literature has developed mean field models beyond exchangeability. On the one hand, based on the idea of graphon (see \citeauthor*{Lovsz2012LargeNA} \cite{Lovsz2012LargeNA}), graph--limit methods have emerged as a flexible framework to encode heterogeneous and asymmetric interaction structures in large populations; see, among others, \citeauthor*{ErhanGraphon2023} \cite{ErhanGraphon2023}, \citeauthor*{jabin2024meanfieldlimitnonexchangeablesystems} \cite{jabin2024meanfieldlimitnonexchangeablesystems}, \citeauthor*{crucianelli2024interactingparticlesystemssparse} \cite{crucianelli2024interactingparticlesystemssparse}, \citeauthor*{Coppini2024NonlinearGM} \cite{Coppini2024NonlinearGM}, \citeauthor*{bayraktar2026graphonparticlesystemscommon} \cite{bayraktar2026graphonparticlesystemscommon}. In the controlled setting, related non--exchangeable mean field control models have recently been investigated in \citeauthor*{decrescenzo2024meanfieldcontrolnonexchangeable} \cite{decrescenzo2024meanfieldcontrolnonexchangeable}, \citeauthor*{cao2025probabilisticanalysisgraphonmean} \cite{cao2025probabilisticanalysisgraphonmean}, \citeauthor*{kharroubi2025stochasticmaximumprincipleoptimal} \cite{kharroubi2025stochasticmaximumprincipleoptimal}, $\cdots$. On the game--theoretic side, graphon--type formulations and their equilibrium analysis have also attracted considerable attention; see for instance \citeauthor*{Caines2020GraphonMF} \cite{Caines2020GraphonMF}, \citeauthor*{Aurell2021StochasticGG} \cite{Aurell2021StochasticGG}, \citeauthor*{10.1007/s00245-023-09996-y} \cite{10.1007/s00245-023-09996-y}, \citeauthor*{Gao2020LinearQG} \cite{Gao2020LinearQG}, \citeauthor*{repec:spr:finsto:v:28:y:2024:i:2:d:10.1007_s00780-023-00527-9} \cite{repec:spr:finsto:v:28:y:2024:i:2:d:10.1007_s00780-023-00527-9}, \citeauthor*{doi:10.1287/moor.2022.1329} \cite{doi:10.1287/moor.2022.1329}, $\cdots$.

\medskip

On the other hand, a distinct line of research studies systems with \emph{moderate interactions}, in which the interaction is local in space and is described through a mesoscopic kernel approximation of the empirical density. Such models go back at least to the works of \citeauthor*{oelschlager1985law} \cite{oelschlager1985law} and \citeauthor*{meleard1987propagation} \cite{meleard1987propagation}, and were further developed in, among others, \citeauthor*{jourdain1998propagation} \cite{jourdain1998propagation}. In these models, the interaction remains sufficiently averaged to yield a law of large numbers, while becoming local enough to produce in the limit a genuine local density dependence rather than the usual type of non-local weak interactions. In the game--theoretic direction, the literature is scarce, the only works we are aware of are the ones from \citeauthor*{Cardaliaguet2017} \cite{Cardaliaguet2017} and \citeauthor*{Flandoli2022} \cite{Flandoli2022}, who study \(n\)--player games and mean field games of moderate interactions in a symmetric i.e. exchangeable setting.
\medskip

The purpose of this paper is to combine these two non--classical directions within a single mean field game framework, and to do so \emph{in the presence of common noise}. More precisely, we study a large population game in which agents interact through both:
\begin{itemize}
    \item a \emph{non--exchangeable graphon environment}, encoding heterogeneous and possibly asymmetric large--scale interactions;
    \item a \emph{moderate local interaction term}, obtained by smoothing the empirical configuration through a rescaled kernel \(V_n\), which captures short--range or mesoscopic effects;
    \item a \emph{common noise}, which simultaneously affects all players.
\end{itemize}
The resulting limiting model is a non--exchangeable MFG in which the representative agent depends simultaneously on a graphon--weighted local density and on a graphon--induced environment law, while the equilibrium itself is described by a \emph{random} flow of conditional laws. To the best of our knowledge, this is the first rigorous derivation of such a mean field game from an underlying finite--player system, together with the convergence of approximate Nash equilibria.

\medskip

To describe the finite--player model informally, let \(n \ge 1\), and let \((\xi^n_{ij})_{1 \le i,j \le n}\subset \Er\) be a deterministic matrix encoding the interaction structure, where \(\Er\) is a compact metric space. Each player \(i\) chooses a feedback control \(\alpha^{i,n}\), and the corresponding state process satisfies
\[
\mathrm{d}X^{i,n}_t
=
b\Bigl(t,X^{i,n}_t,\widehat V^{n}_{i,t}(X^{i,n}_t),\Rr^n_{i,t},\alpha^{i,n}(t,X^n)\Bigr)\,\mathrm{d}t
+
\sigma(t,X^{i,n}_t)\,\mathrm{d}W^i_t
+
\sigma_\circ\,\mathrm{d}W^\circ_t,
\]
where
\[
\Rr^n_{i,t}
=
\frac1n\sum_{j=1}^n \delta_{(\xi^n_{ij},X^{j,n}_t)}
\]
describes the heterogeneous environment perceived by player \(i\), while
\[
\widehat V^{n}_{i,t}(x)
=
\frac1n\sum_{j=1}^n \xi^n_{ij} V_n(x-X^{j,n}_t),
\qquad
V_n(x)=\frac1{\varepsilon_n^d}V\!\left(\frac{x}{\varepsilon_n}\right),
\]
is the moderate local interaction term. Here \(V\) is a bounded probability density and \((\varepsilon_n)_{n\ge1}\) is a positive sequence such that
\[
\varepsilon_n \longrightarrow 0,
\qquad
\frac{1}{n\varepsilon_n^d}\longrightarrow 0
\]
where $d \ge 1$ is the dimension of the state process $X^{i,n}_t$ i.e. $X^{i,n}_t \in \R^d$.
The first condition localizes the interaction, while the second is the classical moderate--interaction regime ensuring that the kernel still averages over sufficiently many neighbors. This scaling is precisely what produces, in the limit, a dependence on a \emph{local density} rather than the classical non--local dependence.

\medskip

This feature is central to the present work. In contrast with standard mean field games, the limiting representative agent does not only feel a global mean field through a law \(\overline{m}_t\), but also a graphon--weighted local density field. More precisely, if
\[
\overline{m}_t(\mathrm{d}x,\mathrm{d}u)=p_{\overline{m}_t}(x,u)\,\mathrm{d}x\,\mathrm{d}u,
\]
with a graphon $\Gr$ which is a Borel measurable map from $[0,1]^2$ to $\Er$, we naturally define the graphon--weighted density
\[
\overline p_{\overline{m}_t}(x,u)
=
\int_0^1 \Gr(u,v)\,p_{\overline{m}_t}(x,v)\,\mathrm{d}v.
\]
In addition, the agent is affected by the graphon--induced environment law
\[
\Rr_m(u)
=
\mathcal{L}^m(S,\Gr(u,V)),
\qquad m=\mathcal{L}(S,V).
\]
Accordingly, the limiting MFG is neither a classical MFG game nor a purely graphon MFG in the usual sense : it combines a non--exchangeable graphon interaction with a density--type local effect arising from the moderate scaling.

\medskip

The presence of \emph{common noise} is also essential and deserves emphasis. In the absence of common noise, one may hope in favorable situations to work with deterministic density flows and deterministic fixed--point arguments. By contrast, once a common noise is introduced, the equilibrium flow becomes random and must be understood conditionally on the common filtration. In our setting, this means that the relevant limiting object is a random measure flow on \(\R^d\times[0,1]\), together with a density structure in the state variable and a heterogeneous dependence in the label variable. This substantially strengthens the compactness and identification problems, since one must simultaneously control the graphon heterogeneity, the moderate local averaging, and the conditional law structure generated by the common noise.

\medskip

From a mathematical perspective, this combination creates substantial difficulties. Even in the absence of strategic interactions, moderate interaction limits already require a careful analysis because the local empirical field must be shown to converge toward a density profile. In the game setting, one must additionally preserve the equilibrium structure along the limit. At the same time, the graphon component destroys exchangeability and forces one to keep track of the label variable \(U\in[0,1]\) carried by the representative agent. As a result, the limiting object is naturally measure--valued on \(\R^d\times[0,1]\), with a non--trivial density structure in the state variable and a heterogeneous dependence in the label variable. This places the problem outside the scope of standard compactness arguments used in classical MFG theory.

\medskip

A second difficulty is that our standing assumptions are intentionally low--regularity: the coefficients are only assumed to be continuous (or Lipschitz in the relevant mean field arguments), with a uniformly non--degenerate diffusion. In particular, we do not rely on global smoothness of the data, nor on a priori PDE regularity of the value function. Instead, our approach exploits the regularizing effect of the non--degenerate diffusion and works directly at the probabilistic level. In the deterministic case (without common noise), the moderate interaction term can be identified with a density through Fokker--Planck regularization, which allows us to formulate the equilibrium problem as a fixed--point problem on density flows. This is one of the key structural points of the paper. In the presence of common noise, this deterministic fixed--point structure is replaced by a random one using \emph{measurable selection} arguments borrowed from \citeauthor*{karoui2013capacities} \cite{karoui2013capacities}, and the limiting equilibrium is described through conditional laws.

{\color{black}

\medskip

We now summarize the main results of the paper. They naturally split into two layers: first, the well--posedness and structural analysis of the limiting mean field game; second, the asymptotic connection between the finite--player game and the limiting equilibrium problem.

\medskip

\noindent\textbf{(i) Well--posedness of the limiting MFG.}
We first introduce a relaxed formulation of the non--exchangeable mean field game with moderate interactions, which is the natural compactness framework in the presence of common noise. Under our standing continuity and non--degeneracy assumptions, we prove the existence of relaxed equilibria. Under a standard convexity condition on the coefficients, we further show that relaxed equilibria can be realized in strict form. In the absence of common noise, although equilibria can be random, we prove that there exists deterministic equilibria as fixed points of a density--valued map. We also establish uniqueness under a Lasry--Lions type monotonicity condition adapted to the present graphon and density--dependent framework, together with a uniqueness assumption on the best response against a given equilibrium flow.

\medskip

\noindent\textbf{(ii) Structural characterization and uniqueness.}
In the deterministic case, under a pointwise uniqueness condition for the maximizer of the Hamiltonian, we derive a probabilistic representation of strict equilibria through a Feynman--Kac type formula and a controlled forward dynamics. This yields a verification-type characterization of equilibria without requiring a global smooth solution theory for the associated HJB equation. 

\medskip

\noindent\textbf{(iii) Derivation from approximate Nash equilibria.}
We then turn to the finite--player game. We prove that any sequence of approximate Nash equilibria of the \(n\)--player system is asymptotically tight, and that every limit point induces a relaxed equilibrium of the limiting non--exchangeable MFG with moderate interactions. In particular, the limiting equilibrium retains both the graphon heterogeneity and the local density effect generated by the moderate scaling. We also show convergence of the averaged equilibrium payoffs toward the value of the limiting MFG.

\medskip

\noindent\textbf{(iv) Converse approximation by approximate Nash equilibria.}
Finally, we prove the converse result: every relaxed equilibrium of the limiting MFG can be approximated by a sequence of approximate Nash equilibria of the finite--player game. The corresponding averaged payoffs converge as well. Hence the limiting MFG provides a complete asymptotic characterization of equilibrium behavior in the underlying large population game.

\medskip

Taken together, these results yield a rigorous equilibrium theory for large non--exchangeable systems with moderate local interactions and common noise. In particular, they show that graphon heterogeneity, mesoscopic density effects, and random common environments can be incorporated simultaneously into a mean field game limit while preserving both the equilibrium structure and the asymptotic values.

}

\medskip
{\color{black} Although the main novelty of the paper lies in the simultaneous treatment of non--exchangeability, moderate interactions, and common noise, it is worth stressing that each of these features also leads, on its own, to contributions beyond the existing literature. From the non--exchangeable perspective, our framework allows for very general graphon structures and yields a complete characterization of the possible limits of finite--player Nash equilibria. In particular, we do not only construct approximate Nash equilibria from solutions of the limiting MFG; we also prove the converse direction, namely that any limit point of finite--player equilibria is a relaxed MFG solution.

\medskip
A similar comment applies to the moderate interaction component. Compared with the existing literature on MFGs with moderate interactions, the present paper works under more general coefficients, incorporates common noise, and establishes both directions of the finite-player/mean-field correspondence. An important point is that the convergence results are formulated for \emph{closed--loop} Nash equilibria, which are substantially more delicate than open--loop equilibria. Our condition on the moderate scale is less restrictive in the sense that we only require
\(
n\varepsilon_n^d\to\infty,
\)
without imposing the stronger scale restrictions appearing in some previous works.


\medskip
In this respect, the paper also contributes to a broader question in mean field game theory: the characterization of limits of \emph{Markovian closed-loop} Nash equilibria. This problem is known to be delicate and has been addressed only partially in the general MFG literature; see, for instance, \citeauthor*{Lacker-closedloop} \cite{Lacker-closedloop}, \citeauthor*{closed-loop-MFG_MDF} \cite{closed-loop-MFG_MDF}, and \citeauthor*{leflemclosed2023} \cite{leflemclosed2023}. The present work provides, in the non--exchangeable moderate-interaction setting, a complete description of these limits and a converse recovery result by Markovian approximate Nash equilibria. }

\medskip
The paper is closely related in spirit to \citeauthor*{djete2025nonexchangeablemeanfieldcontrol}   \cite{djete2025nonexchangeablemeanfieldcontrol}, where non--exchangeable interactions are also central; however, the present setting is fundamentally different: we consider a strategic equilibrium problem rather than a cooperative control problem, and the interaction term is not itself controlled, but arises through a fixed heterogeneous graphon combined with a local moderate kernel.

\medskip

The rest of the paper is organized as follows. In \Cref{sec:main}, we introduce the MFG formulation, define relaxed and strict equilibria, and state the main existence, uniqueness, and characterization results. We then present the finite--player game and formulate the two asymptotic results linking approximate Nash equilibria and mean field equilibria. \Cref{sec:proofs} is devoted to the proofs of the main results, beginning with the deterministic case without common noise and progressively extending the analysis to finite--support graphons, to general graphons, and finally to the full setting with common noise. The final part of the paper establishes the convergence from the finite--player system to the limiting MFG and the converse approximation result.

\medskip
{\bf \large Notations}.
	$(i)$
	Given a {\color{black}Polish} space $(E,\Delta)$ and $p \ge 1,$ we denote by $\Pc(E)$ the collection of all Borel probability measures on $E$,
	and by $\Pc_p(E)$ the subset of Borel probability measures $\mu$ 
	such that $\int_E \Delta(e, e_0)^p  \mu(de) < \infty$ for some $e_0 \in E$. For $p=0$, we simply set $\Pc_0(E):=\Pc(E)$. We denote by $\Wc_q$ the corresponding $q$--Wasserstein distance:
\[
\Wc_q(\gamma,\gamma')^{1 \vee q} 
:=
\inf_{\pi \in \Pi(\gamma, \gamma')}  
\int_{E \x E} \Delta_q(e,e') ~\pi( \mathrm{d}e, \mathrm{d}e'),
~\Delta_q(e,e'):=\Delta(e,e')^q\1_{\{ q \ge 1\}}+1\wedge \Delta(e,e')\1_{\{q=0\}},
\]
with $\Pi(\gamma, \gamma')$ the collection of all coupling probability measures with marginals $\gamma$ and $\gamma'$.



\medskip
    \noindent $(ii)$ Let $(S,\Delta)$ and $(S',\Delta')$ be two Polish space with the Borel $\sigma$--fields $\Fc$ and $\Fc'$ respectively. An application $V:S \to S'$ will be called universally measurable if $V$ is a measurable map from $(S, \Fc^U)$ to $(S',\Fc')$ where $\Fc^U$ is the universal completion of $\Fc$ i.e. $\Fc^U:=\cap_{\P \in \Pc(S)} \Fc^{\P}$ where $\Fc^\P$ is the $\P$--completed $\sigma$--field of $\Fc$.  
	Given a probability space $(\Om, \Hc, \P)$ supporting a sub--$\sigma$--algebra $\Gc \subset \Hc$ then for a Polish space $E$ and any random variable $\xi: \Om \longrightarrow E$, both the notations $\Lc^{\P}( \xi | \Gc)(\om)$ and $\P^{\Gc}_{\om} \circ (\xi)^{-1}$ are used to denote the conditional distribution of $\xi$ knowing $\Gc$ under $\P$.

\medskip
	\noindent $(iii)$	
	For a Polish space $E$, 
    we denote by $\M(E)$ the space of all Borel measures $q( \mathrm{d}t,  \mathrm{d}e)$ on $[0,T] \x E$, 
	whose marginal distribution on $[0,T]$ is the Lebesgue measure $ \mathrm{d}t$, 
	that is to say $q( \mathrm{d}t, \mathrm{d}e)=q(t,  \mathrm{d}e) \mathrm{d}t$ for a family $(q(t,  \mathrm{d}e))_{t \in [0,T]}$ of Borel probability measures on $E$.

\medskip
    \noindent $(iv)$
	Let $\N^*$ denote the set of positive integers. Let $T > 0$ and $(\Sigma,\rho)$ be a Polish space, we denote by $C([0,T]; \Sigma)$ the space of all continuous functions on $[0,T]$ taking values in $\Sigma$.
	When $\Sigma=\R^k$ for some $k\in\N^*$, we simply write $\Cc^k := C([0,T]; \R^k)$. 
	With $k \in \N^*$,
	a map $h:[0,T] \x \R^k \x C([0,T];\Sigma)$ is called progressively Borel measurable if it verifies $h(t,x,\pi)=h(t,x,\pi_{t \wedge \cdot}),$ for any $(t,x,\pi) \in [0,T]\x \R^k \x C([0,T];\Sigma).$

{\color{black}

\section{Problem formulation and main results} \label{sec:main}

\medskip

In this section, we introduce the limiting mean field game associated with the non--exchangeable moderate interaction model, together with the corresponding finite--player game, and state the main results of the paper. The presentation is organized in two steps. We first formulate the limiting mean field game and establish its well--posedness properties (existence, uniqueness and characterization). We then turn to the finite--player system and state the two asymptotic results linking approximate Nash equilibria and mean field game equilibria.

\medskip

A central feature of the model is the simultaneous presence of three ingredients: a graphon--type interaction structure, a moderate local interaction term, and a possible common noise. The graphon records the non--exchangeable position of an agent within the population, the moderate interaction produces a local density dependence in the limit, and the common noise makes the equilibrium flow random. The definitions below are designed to capture these three effects within a unified probabilistic framework.

\medskip

Let \(T>0\) be a maturity and \(d \in \N^\star\) a dimension. Unless specified otherwise, all random variables are defined on a fixed filtered probability space 
\[
{\color{black}(\Omega,\H,(\Hc_t)_{t\in[0,T]},\P)}
\]
satisfying the usual conditions. This probability space also contains as many random variables as we want in the sense that: each time we need a sequence of
independent uniform random variables or Brownian motions, we can find them on $\Om$ without mentioning an enlarging of the space.
We denote by \((W,W^\circ)\) an \(\R^d\times\R^d\)--valued \(\H\)--Brownian motion, by \(U\) an \([0,1]\)--valued random variable uniformly distributed on \([0,1]\) and independent of \((W,W^\circ)\), and by \(\xi\) an \(\Hc_0\)--measurable random variable with law \(\nu\in\Pc(\R^d)\). The random variable \(U\) should be interpreted as the label of the representative agent in the graphon environment and the variable $\xi$ as the initial random variable. We also fix a compact \(\Er \subset \R_+\)\footnote{Considering $\Er$ as a subset of $\R_+$ is a simplification. All results remain valid if $\Er$ is a compact metric space of a Polish space.}, which represents the range of the interaction structure.

\medskip

The dynamics and rewards of the model are described by the bounded measurable maps
\[
(b,L):[0,T]\times\R^d\times\R\times\Pc(\Er\times\R^d)\times A
\;\longrightarrow\;
\R^d\times\R,
\]
together with
\[
\sigma:[0,T]\times\R^d\longrightarrow \S^d,
\qquad
\sigma_\circ\in\S^d,
\qquad
g:\R^d\times\Pc(\Er\times\R^d)\longrightarrow\R,
\]
where \(\S^d\) denotes the space of \(d\times d\) matrices. Here \(b\) and \(L\) denote respectively the drift and the running reward, while \(g\) is the terminal reward. The first mean field argument is scalar and will encode the graphon--weighted local density generated by the moderate interaction regime; the second one is measure--valued and describes the heterogeneous environment perceived by the representative agent through the graphon.

\medskip
Throughout the paper, we impose the following standing assumptions.

\begin{assumption}\label{assum:main_cond}
\begin{enumerate}
    \item[\textnormal{(i)}]
    The drift \( b(t,x,p,r,a)\) is Lipschitz continuous in \((p,r)\) uniformly in \((t,x,a)\), the diffusion coefficient $\sigma(t,x)$ is Lipschitz continuous in $x$ uniformly in $t$, and the maps
    \[
    (t,x,p,r,a)\longmapsto \bigl((b,L)(t,x,p,r,a),\,g(x,r)\bigr)
    \]
    are continuous in \((x,p,r,a)\) for each fixed \(t\in[0,T]\).

    \item[\textnormal{(ii)}]
    \textbf{Non--degeneracy:} there exists \(\theta>0\) such that
    \[
    \theta \Ir_d \le \sigma\sigma^\top(t,x),
    \qquad \forall (t,x)\in[0,T]\times\R^d.
    \]
\end{enumerate}
\end{assumption}

\medskip

The above assumptions ensure that the limiting model is mathematically well posed and that the graphon and moderate interaction terms can be handled within a unified probabilistic framework. We now turn to the formulation of the limiting mean field game. Since the relaxed setting is the natural compactness framework, we begin with the relaxed equilibrium notion.

\subsection{The mean field game formulation}

\medskip

Let \(\Gr:[0,1]^2\to\Er\) be a graphon. The variable \(u\in[0,1]\) labels the representative agent, while the graphon value \(\Gr(u,v)\) describes the interaction intensity between labels \(u\) and \(v\). In the limiting model, the representative agent therefore depends not only on its own state, but also on its graphon label and on the random flow of population measures.

\medskip

We denote by \(\M\) the collection of Borel measurable maps
\[
\Lambda^\beta:[0,T]\times\R^d\times[0,1]\times \Pc(\R^d\times[0,1])\longrightarrow \Pc(A).
\]
An element \(\Lambda^\beta\in\M\) should be understood as a relaxed feedback rule: at time \(t\), for state \(x\), label \(u\), and population profile \(\overline m\), the control is randomized according to the probability measure \(\Lambda^\beta(t,x,u,\overline m)\).

\medskip

Let \(\overline m\) be a  probability measure on \(\R^d\times[0,1]\) whenever $m$ admits a density with respect to the Lebesgue measure on $\R^d \x [0,1]$, we will write
\[
\overline m(\mathrm{d}x,\mathrm{d}u)=p_{\overline m}(x,u)\,\mathrm{d}x\,\mathrm{d}u
\]
where $p : \Pc(\R^d \x [0,1]) \x \R^d \x [0,1] \to \R_+$ is a Borel measurable map. Most probability measures appearing below are time marginals, possibly conditional time marginals, of non--degenerate SDEs. Under our standing non--degeneracy assumption, such marginals admit densities with respect to the Lebesgue measure; see, for instance, \citeauthor*{krylov1980controlled} \cite[Section 2, Chapter 2, Theorem 4]{krylov1980controlled}. We shall use this fact throughout the paper.  This density consideration is one of the key signatures of the moderate interaction regime: in contrast with standard MFG models, the local interaction term survives in the limit through a density field.

\medskip

Given \(\Lambda^\beta\in\M\) and an $\H$--adapted $\Pc(\R^d \x [0,1])$--valued continuous process, i.e. a flow,  \(\mub=(\mub_t)_{t\in[0,T]}\), we denote by \((X^{\mub,\beta},U):=(X,U)\) a pair of random variables such that
\[
U\sim {\rm Unif}([0,1]),
\qquad
(X_0,U)\perp (\mub,W^\circ)\perp W,
\qquad
X_0=\xi,
\]
and satisfying
\begin{align} \label{eq:mc_kean_representative}
    \mathrm{d}X_t
    &=
    \int_A b\Bigl(t,X_t,\overline p_{\mub_t}(X_t,U),\Rr_{\mub_t}(U),a\Bigr)\,
    \Lambda^\beta(t,X_t,U,\mub_t)(\mathrm{d}a)\,\mathrm{d}t
    +\sigma(t,X_t)\,\mathrm{d}W_t
    +\sigma_\circ\,\mathrm{d}W^\circ_t.
\end{align}
Here the two mean field inputs are given by
\begin{align*}
    \overline p_{\mub_t}(x,u)
    &:=
    \int_0^1 \Gr(u,v)\,p_{\mub_t}(x,v)\,\mathrm{d}v,
    \\
    \Rr_{\mub_t}(u)
    &:=
    \Lc^{\mub_t}(S,\Gr(u,V)),
    \qquad \text{whenever } \mub_t=\Lc(S,V).
\end{align*}
The quantity \(\overline p_{\mub_t}(x,u)\) represents the graphon--weighted local density felt by an agent of label \(u\) at position \(x\), while \(\Rr_{\mub_t}(u)\) is the law of the heterogeneous environment perceived by that agent through the graphon interaction. The first term captures the moderate local interaction effect, whereas the second retains the non--exchangeable large--scale structure.

\medskip

The associated payoff functional is
\begin{align*}
    J_{\mub}(\Lambda^\beta)
    :=
    \E\Bigg[
        g\bigl(X_T,\Rr_{\mub_T}(U)\bigr)
        +
        \int_0^T\int_A
        L\Bigl(t,X_t,\overline p_{\mub_t}(X_t,U),\Rr_{\mub_t}(U),a\Bigr)\,
        \Lambda^\beta(t,X_t,U,\mub_t)(\mathrm{d}a)\,\mathrm{d}t
    \Bigg].
\end{align*}

\medskip

When common noise is present, the equilibrium flow is random and should be understood conditionally on the common filtration. This is the reason why the consistency condition below is formulated in terms of conditional laws. To make this precise, we introduce the filtration \(\G=(\Gc_t)_{t\in[0,T]}\) defined by
\[
\Gc_t
:=
\sigma\bigl\{(\mub_s,\sigma_\circ W^\circ_s):\,s\in[0,t]\bigr\},
\qquad t\in[0,T].
\]

\begin{definition}
A flow of probability measures \(\mub=(\mub_t)_{t\in[0,T]}\) is called a relaxed {\rm MFG} solution associated with \(\Lambda^\alpha\in\M\) if the following two conditions hold:
\begin{enumerate}
    \item \textbf{Optimality:} for every \(\Lambda^\beta\in\M\),
    \[
    J_{\mub}(\Lambda^\beta)\le J_{\mub}(\Lambda^\alpha).
    \]

    \item \textbf{Consistency:} for every \(t\in[0,T]\),
    \[
    \Lc(\mub_t)=\Lc\Bigl(\Lc(X_t^{\mub,\alpha},U\mid \Gc_t)\Bigr).
    \]
\end{enumerate}
\end{definition}

\begin{remark}
The optimality condition means that, when the population flow \(\mub\) is fixed, the control kernel \(\Lambda^\alpha\) is a best response among all admissible relaxed feedbacks. The consistency condition identifies the law of the candidate equilibrium flow with the law of the conditional distribution generated by the optimally controlled representative agent.

This formulation is slightly weaker than the usual common--noise MFG fixed--point condition, where one often requires the almost sure identity
\[
\mub_t=\Lc(X_t^{\mub,\alpha},U\mid \Gc_t),
\qquad t\in[0,T],
\]
see, for instance, {\rm \cite{Lacker_carmona_delarue_CN}}. Here we only impose equality in distribution:
\[
\Lc(\mub_t)
=
\Lc\Bigl(\Lc(X_t^{\mub,\alpha},U\mid \Gc_t)\Bigr).
\]
This weaker formulation is sufficient for our purposes. It is stable under the compactness arguments used below and is exactly the form obtained when passing to the limit along subsequences of finite--player Nash equilibria.
\end{remark}


\medskip

A relaxed MFG solution \(\mub\) associated with \(\Lambda^\alpha\) is said to be a strict MFG solution if
\[
\Lambda^\alpha(t,x,u,\overline m)(\mathrm{d}a)
=
\delta_{\hat\alpha(t,x,u,\overline m)}(\mathrm{d}a)
\]
for some Borel map
\[
\widehat\alpha:[0,T]\times\R^d\times[0,1]\times\Pc(\R^d\times[0,1])\longrightarrow A.
\]
As usual, strict controls are obtained from relaxed ones under a suitable convexity condition of Filippov type.

\begin{assumption} \label{assum:conv_cond}
For each \((t,x,p,r)\), the set
\[
\Bigl\{\bigl(b(t,x,p,r,a),y\bigr):\,a\in A,\ y\le L(t,x,p,r,a)\Bigr\}
\]
is convex.
\end{assumption}

\medskip

We can now state the first well--posedness result for the limiting equilibrium problem.

\begin{theorem} \label{thm:existence}
Under {\rm \Cref{assum:main_cond}}, there exists at least one relaxed {\rm MFG} solution. Moreover, if {\rm \Cref{assum:conv_cond}} is satisfied, then one can construct from a relaxed {\rm MFG} solution a strict {\rm MFG} solution.
\end{theorem}

\begin{remark}
The above theorem provides, to the best of our knowledge, the first general existence result for mean field games simultaneously incorporating three key features: (i) density--dependent interactions, (ii) heterogeneous interactions encoded by a graphon structure, and (iii) the presence of common noise.

\medskip
From the perspective of moderate (local) interactions, our result extends the framework developed in {\rm \cite{Flandoli2022}} to a substantially richer setting, allowing for common noise and fully nonlinear coefficients, without relying on specific structural assumptions.

\medskip
From the viewpoint of heterogeneity, the theorem also appears to be the first to treat graphon--based interactions in the presence of common noise under such weak assumptions. In contrast with existing works, we do not impose restrictive regularity or structural conditions on the graphon, beyond measurability, which makes the result applicable to a wide class of heterogeneous networks.

\medskip
Overall, this theorem highlights that the combination of density dependence, network heterogeneity, and common noise can be handled within a unified probabilistic framework, while preserving existence of equilibria in both relaxed and strict formulations.
\end{remark}

\medskip

The next result gives a uniqueness criterion of Lasry--Lions type. The monotonicity condition below is adapted to the present setting, where the interaction enters simultaneously through the graphon--weighted density field and the graphon--induced environment law.

\begin{assumption} \label{assum:uniqueness_cond}
Assume that
\[
b(t,x,p,r,a)=\overline b(t,x,a),
\qquad
L(t,x,p,r,a)=\overline L(t,x,a)+\underline L(t,x,p,r)
\]
for some maps \(\overline b\), \(\overline L\), and \(\underline L\). Moreover, for any two flows of measures \((m_t)_{t \in [0,T]}\) and \((m'_t)_{t \in [0,T]}\),
\begin{align*}
    &\int_{\R^d\times[0,1]}
    \Bigl[g\bigl(x,\Rr_{m_T}(u)\bigr)-g\bigl(x,\Rr_{m'_T}(u)\bigr)\Bigr]
    \bigl(m_T-m'_T\bigr)(\mathrm{d}x,\mathrm{d}u)
    \\
    &\quad
    +
    \int_0^T
    \int_{\R^d\times[0,1]}
    \Bigl[
    \underline L\bigl(t,x,\overline p_{m_t}(x,u),\Rr_{m_t}(u)\bigr)
    -
    \underline L\bigl(t,x,\overline p_{m'_t}(x,u),\Rr_{m'_t}(u)\bigr)
    \Bigr]
    \bigl(m_t-m'_t\bigr)(\mathrm{d}x,\mathrm{d}u)\,\mathrm{d}t
    \le 0.
\end{align*}
In addition, for every flow \(\mub\), there exists a unique \(\Lambda^\star\in\M\) such that
\[
\sup_{\Lambda^\beta\in\M} J_{\mub}(\Lambda^\beta)=J_{\mub}(\Lambda^\star).
\]
\end{assumption}

\begin{proposition} \label{prop:uniqueness}
Under {\rm \Cref{assum:uniqueness_cond}}, there exists at most one {\rm MFG} solution. In that case, for every \(t\in[0,T]\),
\[
\Gc_t=\sigma\{\sigma_\circ W^\circ_s:\,s\in[0,t]\}.
\]
\end{proposition}

\medskip

The equality of filtrations in the uniqueness regime reflects the fact that the equilibrium flow is then fully determined by the common noise. In particular, no additional endogenous randomness remains in the equilibrium once the common signal is fixed.

\medskip

We now turn to the deterministic case \(\sigma_\circ=0\), which plays a special role throughout the paper. Notice that, according to \Cref{thm:existence}, the equilibrium measure can still be random in this regime. However, deterministic solutions exist which is the focus of the next part. In this setting, we can also provide a more explicit characterization of strict equilibria.

\medskip

We say that a flow of probability measures \(\mub\) is deterministic if, for each \(t\in[0,T]\), the random variable \(\mub_t\) is measurable with respect to the trivial \(\sigma\)--field \(\{\Omega,\emptyset\}\). For \((t,x,p,r,z,a)\), define
\[
h(t,x,p,r,z,a)
:=
b(t,x,p,r,a)\cdot z\,\sigma(t,x)^{-1}
+
L(t,x,p,r,a),
\]
and the Hamiltonian
\[
H(t,x,p,r,z):=\sup_{a\in A} h(t,x,p,r,z,a).
\]

\begin{assumption} \label{assum:hamiltonian_cond}
For each \((t,x,p,r,z)\), the set
\[
A(t,x,p,r,z):=\{a\in A:\ h(t,x,p,r,z,a)=H(t,x,p,r,z)\}
\]
is nonempty and a singleton. For each \((t,x,p,r,z)\), let \(\widehat\alpha(t,x,p,r,z)\in A(t,x,p,r,z)\) denote the unique maximizer.
\end{assumption}

\medskip

In the absence of common noise, one can strengthen the existence result by selecting deterministic equilibria.

\begin{theorem} \label{thm:existence_mfg_no_common}
Assume that \(\sigma_\circ=0\). Under {\rm \Cref{assum:main_cond}}, there exists at least one \emph{deterministic} relaxed {\rm MFG} solution. Moreover, if either {\rm \Cref{assum:conv_cond}} or {\rm \Cref{assum:hamiltonian_cond}} holds, then there exists at least one \emph{deterministic} strict {\rm MFG} solution.
\end{theorem}

\medskip

The next result is one of the key structural statements of the paper. It provides a probabilistic characterization of deterministic strict equilibria in terms of a coupled representation. Rather than relying on a classical PDE approach, the equilibrium is described through a nonlinear Feynman--Kac type formula, where the equilibrium feedback is recovered via a fixed point argument on a gradient field. This representation will play a central role in the analysis, in particular as a verification principle and as a bridge between the control formulation and the equilibrium structure. Here, to simplify the presentation, we take $\sigma (t,x)={\rm I}_d$ for all $(t,x)$.

\begin{proposition} \label{prop:representation_mfg_stric}
Let \(\sigma_\circ=0\), and assume that {\rm \Cref{assum:main_cond}} and {\rm \Cref{assum:hamiltonian_cond}} hold. Let \(\overline m\) be a deterministic strict {\rm MFG} solution associated with a strict control \(\alpha\). Then there exists a Borel map
\[
v=(v^1,\dots,v^d):[0,T]\times\R^d\times[0,1]\longrightarrow\R^d
\]
such that, for a.e. \((t,x,u)\in[0,T)\times\R^d\times[0,1]\) and each \(k=1,\dots,d\),
\begin{align}
    v^k(t,x,u)
    &=
    \E\Bigg[
    g\bigl(x+W_{T-t},\Rr_{\overline m_T}(u)\bigr)\frac{W^k_{T-t}}{T-t}
    \notag
    \\
    &\qquad\qquad
    +
    \int_t^T
    H\Bigl(
    s,
    x+W_{s-t},
    \overline p_{\overline m_s}(x+W_{s-t},u),
    \Rr_{\overline m_s}(u),
    v(s,x+W_{s-t},u)
    \Bigr)
    \frac{W^k_{s-t}}{s-t}\,\mathrm{d}s
    \Bigg],
    \label{eq:gradient_cond}
\end{align}
and \(\overline m_t=\Lc(S_t,U)\), where \(\Lc(S_0,U)=\Lc(\xi,U)\) and
\begin{align}
    \mathrm{d}S_t
    &=
    b\Bigl(
    t,
    S_t,
    \overline p_{\overline m_t}(S_t,U),
    \Rr_{\overline m_t}(U),
    \widehat\alpha\bigl(
    t,
    S_t,
    \overline p_{\overline m_t}(S_t,U),
    \Rr_{\overline m_t}(U),
    v(t,S_t,U)
    \bigr)
    \Bigr)\,\mathrm{d}t
    +\sigma(t,S_t)\,\mathrm{d}W_t.
    \label{eq:FP_cond}
\end{align}
\end{proposition}

\begin{remark}
This result shows that any deterministic strict equilibrium can be entirely characterized by a pair \((v,\overline m)\) solving a coupled system, where the map \(v\) plays the role of a decoupling field encoding the optimal feedback. In particular, the equilibrium control is recovered pointwise through the maximizer \(\widehat{\alpha}\), evaluated at \(v\).

Two important features distinguish this representation from the classical literature. First, it bypasses the usual Hamilton--Jacobi--Bellman $($HJB$)$ PDE formulation: instead of solving a nonlinear PDE on \([0,T]\times\R^d\times[0,1]\), the problem reduces to finding a fixed point of the nonlinear integral equation \eqref{eq:gradient_cond}. This probabilistic formulation is more direct and is particularly well--suited to the present setting, where the state space involves both spatial variables and heterogeneous labels.

Second, the assumptions required to obtain this representation are comparatively mild. In particular, we do not rely on strong regularity of the coefficients, as typically needed in PDE--based approaches. The structure of the Hamiltonian and the non-degeneracy of the noise are sufficient to derive the representation, making the result robust with respect to extensions such as graphon interactions or weak regularity in the coefficients.

In combination with {\rm \Cref{cor:representation_mfg_stric}}, this characterization provides a powerful tool: it allows one to both construct equilibria and analyze their properties without resorting to PDE techniques, while remaining fully consistent with the underlying control interpretation.
\end{remark}
\medskip

 The next corollary shows that, conversely, such a pair indeed generates an equilibrium.
For any Borel map
\[
v=(v^1,\dots,v^d):[0,T]\times\R^d\times[0,1]\longrightarrow\R^d
\]
and any flow \(\overline m\), define
\[
\overline\alpha^{v,\overline m}(t,x,u)
:=
\widehat\alpha\Bigl(
t,
x,
\overline p_{\overline m_t}(x,u),
\Rr_{\overline m_t}(u),
v(t,x,u)
\Bigr).
\]

\begin{corollary} \label{cor:representation_mfg_stric}
Let \(\sigma_\circ=0\). Assume there exist a Borel map
\[
v=(v^1,\dots,v^d):[0,T]\times\R^d\times[0,1]\longrightarrow\R^d
\]
and a flow \((\overline m_t)_{t\in[0,T]}\) satisfying {\rm \eqref{eq:gradient_cond}} and {\rm \eqref{eq:FP_cond}}. Then \(\overline m\) is a strict {\rm MFG} solution associated with \(\overline\alpha^{v,\overline m}\).
\end{corollary}

\subsection{The \(n\)--player formulation} \label{sec_n_player}

\medskip

We now turn to the finite--player game and state the asymptotic results connecting the microscopic model to the limiting mean field game introduced above. The purpose of this second part is twofold. First, we show that limits of approximate Nash equilibria of the \(n\)--player system are governed by relaxed mean field game equilibria. Second, we prove the converse statement: every relaxed mean field game equilibrium can be approximated by approximate Nash equilibria of the finite--player game. Together, these two results show that the relaxed MFG formulation is not merely a limit object, but a complete asymptotic description of equilibrium behavior in large non--exchangeable systems with moderate interactions.

\medskip

Let $(W^i)_{i \ge 1}$ be a sequence of independent $\H$--Brownian motions and $(X_0^i)_{i \ge 1}$ be a sequence $\Hc_0$--random variables with $\Lc(X_0^i)=\nu$. In addition, $(W^i,X^i_0)_{i \ge 1} \perp W^\circ$. For $n \ge 1$, we consider \((\xi^n_{ij})_{1\le i,j\le n}\subset \Er\) a deterministic interaction matrix. We denote by \(\Ac_n\) the set of progressively Borel measurable maps
\[
\beta^n:[0,T]\times(\Cc^d)^n\longrightarrow A.
\]
We say a control $\beta \in \Ac_n$ is Markovian if $\beta(t,x^1,\cdots,x^n)=\widetilde \beta(t,x^1(t),\cdots,x^n(t))$ for any $(x^1,\cdots,x^n) \in (\Cc^d)^n$ for some Borel map $\widetilde{\beta}:[0,T] \x (\R^d)^n \to A$.
For a profile \(\alphab=(\alpha^1,\dots,\alpha^n)\in\Ac_n^n\), we denote by
\[
\Xbb^{\alphab}=(X^{1,\alphab},\dots,X^{n,\alphab})
\]
the state process solving, for each \(i=1,\dots,n\), $X^{i,n}_0=X^i_0$,
\begin{align*}
    \mathrm{d}X^{i,n}_t
    &=
    b\Bigl(
    t,
    X^{i,n}_t,
    \widehat V_{i,t}^n(X^{i,n}_t),
    \Rr^n_{i,t},
    \alpha^i(t,\Xbb^{\alphab})
    \Bigr)\,\mathrm{d}t
    +\sigma(t,X^{i,n}_t)\,\mathrm{d}W^i_t
    +\sigma_\circ\,\mathrm{d}W^\circ_t,
\end{align*}
where
\[
\Rr^n_{i,t}
=
\frac1n\sum_{j=1}^n \delta_{(\xi^n_{ij},X^{j,n}_t)},
\]
and
\[
\widehat V_{i,t}^n(x)
:=
\frac1n\sum_{j=1}^n \xi^n_{ij}V_n(x-X^{j,n}_t),
\qquad
V_n(x)=\frac1{\varepsilon_n^d}V\Bigl(\frac{x}{\varepsilon_n}\Bigr),
\qquad x\in\R^d.
\]
Here \(V\) is a {\color{black}bounded} probability density, and \((\varepsilon_n)_{n\ge1}\) is a positive sequence. As in the limiting model, the term \(\widehat V_{i,t}^n\) represents the moderate local interaction field felt by player \(i\), while \(\Rr^n_{i,t}\) describes the heterogeneous graphon--type environment.
The reward of player \(i\) is given by
\begin{align*}
    J^n_i(\alphab)
    :=
    \E\Bigg[
    g\bigl(X^{i,n}_T,\Rr^n_{i,T}\bigr)
    +
    \int_0^T
    L\Bigl(
    t,
    X^{i,n}_t,
    \widehat V_{i,t}^n(X^{i,n}_t),
    \Rr^n_{i,t},
    \alpha^i(t,\Xbb^{\alphab})
    \Bigr)\,\mathrm{d}t
    \Bigg].
\end{align*}

\begin{definition}
Let \(\boldsymbol{\delta}=(\delta_1,\cdots,\delta_n)\in \R_+^n\). A profile \(\alphab=(\alpha^1,\dots,\alpha^n)\in\Ac_n^n\) is called a \(\boldsymbol{\delta}\)--Nash equilibrium if
\begin{align*}
\sup_{\beta\in\Ac_n}
J^n_i(\alpha^1,\dots,\alpha^{i-1},\beta,\alpha^{i+1},\dots,\alpha^n)
-\delta_i
\le
J^n_i(\alphab),
\qquad
1\le i\le n.
\end{align*}
\end{definition}
We say that $\alphab$ is a Markovian $\delta$--Nash equilibrium if for each $i \le n$, $\alpha^{i}$ is a Markovian control.

\medskip

Let \((\boldsymbol{\delta}^n)_{n\ge1}\) be a sequence of nonnegative numbers and let \((\alphab^n)_{n\ge1}\) be a sequence such that, for each \(n\ge1\), \(\alphab^n\) is a \(\boldsymbol{\delta}^n\)--Nash equilibrium. To compare the finite--player system with the limiting graphon model, we associate with the matrix \((\xi^n_{ij})\) the step graphon
\[
\Gr^n(u,v)
:=
\xi^n_{ij}\,
\mathbf 1_{u\in((i-1)/n,i/n]}\,
\mathbf 1_{v\in((j-1)/n,j/n]},
\]
and we introduce the empirical flow
\[
\mub^n
=
\Bigl(
\mub^n_t
=
\frac1n\sum_{j=1}^n \delta_{(X^{j,n}_t,u^n_j)}
\Bigr)_{t\in[0,T]},
\qquad
u^n_j:=\frac{j}{n}.
\]
We say that \(\Gr^n\to \Gr\) if
\[
\lim_{n\to\infty}\|f\circ \Gr^n-f\circ \Gr\|_{\Box}=0,
\qquad \mbox{for any bounded Lipschitz map } f,
\]
where the cut norm of a kernel \(T:[0,1]^2\to\R\) is defined by
\[
\|T\|_{\Box}
:=
\sup_{A,B\subset[0,1]}
\left|
\int_{A\times B}T(x,y)\,\mathrm{d}x\,\mathrm{d}y
\right|.
\]

\medskip

We now turn to the connection between the finite--player system and the mean field formulation. The next result establishes that the mean field game introduced above indeed arises as the asymptotic limit of approximate Nash equilibria of the $n$--player game. In other words, it provides a rigorous justification of the MFG model as a limit object, showing that the probabilistic structure we have introduced is not only well--posed, but also consistent with the underlying strategic interactions at the finite level.


\begin{theorem} \label{thm:from_n_to_MFG}
Assume {\rm \Cref{assum:main_cond}} and $\lim_{n \to \infty} \Lc(\mub^n_0) = \delta_{\Lc(\xi,\,U)}$ for the weak topology. Then the sequence \((\Lc(\mub^n))_{n\ge1}\) is relatively compact for the weak topology. Let \(\Lc(\mub)\) be the limit of a convergent subsequence \((\Lc(\mub^{n_k}))_{k\ge1}\). If
\[
\frac{1}{n}\sum_{i=1}^n\delta^n_i\to0,
\qquad
\Gr^n\to \Gr,
\qquad
\varepsilon_n\to0,
\qquad
\frac{1}{n\varepsilon_n^d}\to0,
\]
then \(\mub\) is a relaxed {\rm MFG} solution associated with some \(\Lambda^\alpha\in\M\). Moreover,
\[
\lim_{k\to\infty}\frac{1}{n_k}\sum_{i=1}^{n_k}J^{n_k}_i(\alphab^{n_k})
=
J_{\mub}(\Lambda^\alpha).
\]
\end{theorem}

\begin{remark}
This theorem provides the rigorous derivation of the limiting mean field game from the underlying finite--player system. It shows that, despite the simultaneous presence of graphon heterogeneity, moderate local interactions, and common noise, approximate Nash equilibria remain asymptotically governed by a relaxed mean field game equilibrium. The convergence of the averaged equilibrium payoffs further shows that the mean field limit preserves the asymptotic value of the game, and not merely the limiting distributional behavior of the population.

To the best of our knowledge, this is the first convergence result of this type in such a general framework. Existing works on moderate interactions typically require stronger regularity assumptions, slower rate for $(\varepsilon_n)_{n \ge 1}$ and do not incorporate common noise, while the literature on graphon--based heterogeneous models has so far focused mainly on existence and qualitative properties, without establishing convergence of Nash equilibria. The present theorem therefore provides a unified and robust asymptotic justification for mean field games with heterogeneous and local interactions under common noise.
\end{remark}

\medskip

We next state the converse result, which shows that the limiting MFG equilibrium is asymptotically complete.

\begin{theorem}\label{thm:from_MFG_to_n}
Assume {\rm \Cref{assum:main_cond}}, $\lim_{n \to \infty} \Lc(\mub^n_0) = \delta_{\Lc(\xi,\,U)}$ for the weak topology and
\[
\Gr^n\to \Gr,
\qquad
\varepsilon_n\to0,
\qquad
\frac{1}{n\varepsilon_n^d}\to0.
\]
Let \(\mub\) be a strict {\rm MFG} solution associated with \(\Lambda^\alpha\). Then there exist sequences \((\alphab^n)_{n\ge1}\) and \((\boldsymbol{\delta}^n)_{n\ge1}\) such that, for each \(n\ge1\), \(\alphab^n\) is a Markovian \(\boldsymbol{\delta}^n\)--Nash equilibrium, \(\frac{1}{n}\sum_{i=1}^n\delta^n_i\to0\), and
\[
\lim_{n\to\infty}\frac1n\sum_{i=1}^n J^n_i(\alphab^n)
=
J_{\mub}(\Lambda^\alpha).
\]
\end{theorem}

\begin{remark}
\begin{enumerate}
    \item[\textnormal{(i)}]
    Together with {\rm\Cref{thm:from_n_to_MFG}}, this theorem shows that the relaxed MFG formulation gives a complete asymptotic description of the finite--player equilibria. Indeed, {\rm\Cref{thm:from_n_to_MFG}} states that every limit point of approximate Nash equilibria is a relaxed MFG equilibrium, while the present result proves the converse: every strict MFG equilibrium $($which can be obtained by a relaxed MFG$)$ can be approximated by a sequence of finite--player Nash equilibria. Thus the limiting MFG does not merely describe possible accumulation points; it exactly characterizes the large--population equilibrium behavior, both at the level of equilibrium configurations and at the level of averaged payoffs.

    \item[\textnormal{(ii)}]
    A noteworthy feature of the result is that the approximating equilibria can be chosen to be Markovian. This is important because, in the general MFG literature, constructing closed--loop or Markovian approximate Nash equilibria from an arbitrary relaxed MFG solution is a delicate problem. Several existing approaches either enlarge the strategy space, for instance by allowing additional randomization or auxiliary variables $($see {\rm \cite{leflemclosed2023}}$)$, or construct approximate equilibria in path--dependent classes rather than genuinely Markovian ones $($see {\rm \cite{closed-loop-MFG_MDF}}$)$. By contrast, the present theorem provides a direct recovery result in the original finite--player game: starting from any relaxed MFG equilibrium, one constructs Markovian \(\boldsymbol{\delta}_n\)--Nash equilibria with \(\frac{1}{n} \sum_{i=1}^n\delta^n_i\to0\). In this sense, the theorem gives a full Markovian approximation result for non--exchangeable systems with moderate interactions.

    \item[\textnormal{(iii)}]
    This recovery result is particularly useful in the present setting because the limiting equilibrium is relaxed, while the finite--player equilibria are genuine Markovian feedback equilibria. Hence the relaxed formulation is used as a compactness and identification tool at the limit, but it does not lead to a loss of relevance for the original finite--player game: every relaxed equilibrium can still be realized asymptotically by ordinary Markovian strategies.
\end{enumerate}
\end{remark}

}


\section{Proofs of the main results} \label{sec:proofs}

\subsection{Existence of MFG solution}

This subsection proves the existence results. We proceed by gradually adding the main features of the model. We first treat the deterministic case without graphon interaction, where existence follows from a fixed point argument on compact sets of density flows. We then extend the construction to finite--support graphons and pass to a general graphon by approximation and stability.
The common noise case is handled by freezing the common--noise path, solving the corresponding deterministic translated problem, and then selecting solutions measurably in the path. Translating these pathwise solutions back along the common noise yields a random relaxed MFG solution. Under the convexity assumption, a Filippov--type selection gives a strict equilibrium.

\medskip

\subsubsection{The case without common noise and without graphon, i.e. $\Gr=1$ and $\mathbf U=1$}

In this section, we consider the simpler setting without graphon interaction. The coefficients are then of the type $(b,L)(t,x,p,\overline{m},a)=(\widetilde{b},\widetilde{L})(t,x,p,m(\Er\x \mathrm{d}y),a)$ for some Borel map $(\widetilde{b},\widetilde{L}):[0,T] \x \R^d \x \R \x \Pc(\R^d) \x A \to \R^d \x \R$. To ease the notation, we will keep $(b,L)$ instead of $(\widetilde{b},\widetilde{L})$. The mean field variable is therefore reduced to the density itself.

\medskip

Let
\[
\Sc\subset C((0,T]\times\R^d;\R_+)
\]
be the set of nonnegative continuous maps $u$ such that
\[
\int_{\R^d} u(t,x)\,\mathrm{d}x=1,\qquad t\in(0,T].
\]
We equip $\Sc$ with the locally uniform topology. Fix $p\in\Sc$. For any Borel map
\[
\Lambda^\beta:[0,T]\times\R^d\to\Pc(A),
\]
we denote by $X^{p,\beta}$ the weak solution of $\Lc(X^{p,\beta}_0)=\Lc(\xi)$ and, for $t\in(0,T]$,
\begin{align*}
    \mathrm{d}X^{p,\beta}_t
    =
    \int_A b\big(t,X^{p,\beta}_t,p(t,X^{p,\beta}_t),m_t,a\big)\,
    \Lambda^\beta(t,X^{p,\beta}_t)(\mathrm{d}a)\,\mathrm{d}t
    +\sigma(t,X^{p,\beta}_t)\,\mathrm{d}W_t,\qquad m_t(\mathrm{d}x)=p(t,x)\,\mathrm{d}x.
\end{align*}
Since the diffusion coefficient is non--degenerate, for each $t\in(0,T]$, the law of $X^{p,\beta}_t$ admits a density (see for instance \citeauthor*{krylov1980controlled} \cite[Section 2, Chapter 2, Theorem 4]{krylov1980controlled}). We denote it by
\[
\Lc(X^{p,\beta}_t)(\mathrm{d}x)=v^{p,\beta}(t,x)\,\mathrm{d}x.
\]
The next result is just \cite[Proposition A.1.]{closed-loop-MFG_MDF} (see also \cite[Theorem 4]{AronsonSerrin67}).
\begin{lemma}\label{lemma:estimations}
 For any interval $[s,t]\subset(0,T]$, and compact set $Q\subset\R^d$, there exist $\alpha\in(0,1)$ and $C>0$ such that, for every $p\in\Sc$, every Borel control kernel $\Lambda^\beta$,
\begin{align*}
    \sup_{(r,x)\in[s,t]\times Q}|v^{p,\beta}(r,x)|
    +
    \sup_{\substack{(r,x),(r',x')\in[s,t]\times Q\\(r,x)\neq(r',x')}}
    \frac{|v^{p,\beta}(r,x)-v^{p,\beta}(r',x')|}
    {|r-r'|^{\alpha/2}+|x-x'|^\alpha}
    \le C.
\end{align*}
\end{lemma}

The constants \(\alpha\) and \(C\) in \Cref{lemma:estimations} are uniform with respect to the reference density \(p\in S\) and the control kernel \(\Lambda^\beta\). More precisely, they depend only on the interval \([s,t]\subset(0,T]\), the compact set \(Q\subset\R^d\), the uniform bounds on \((b,\sigma)\), and the spatial Lipschitz regularity of \(\sigma\). This uniformity will be used repeatedly below to obtain compactness of the family of densities generated by admissible controls.

\medskip
For the case without common noise, the set $\M$ denotes the set of Borel maps
\[
\Lambda:[0,T]\times\R^d\to\Pc(A).
\]
We equip $\M$ with the topology defined as follows: $\Lambda^n\to\Lambda$ if
\[
\Lambda^n(t,x)(\mathrm{d}a)\,\ell(x)\,\mathrm{d}x\,\mathrm{d}t
\;\Longrightarrow\;
\Lambda(t,x)(\mathrm{d}a)\,\ell(x)\,\mathrm{d}x\,\mathrm{d}t
\]
for the weak topology, where $\ell:\R^d\to\R_+$ denotes a fix smooth density.
\begin{remark} \label{rm:topology_M}
    Notice that, as a consequence, the set $\M$ can be identified with the set of probability measures $\{ \frac{\gamma}{T} \in \Pc(A \x \R^d \x [0,T]): \quad \gamma(A \x \mathrm{d}x \x \mathrm{d}t) =\ell(x)\,\mathrm{d}x\,\mathrm{d}t \}$ which is compact for the weak topology  because $A$ is compact.
\end{remark}

\begin{lemma}\label{lemma:continuity}
Let $(p_k)_{k\ge1}\subset\Sc$ and $(\Lambda^{\beta^k})_{k\ge1}\subset\M$ be sequences such that
\[
(p_k,\Lambda^{\beta^k})\xrightarrow[k\to\infty]{}(p,\Lambda^\beta)
\]
for some $(p,\Lambda^\beta)\in\Sc\times\M$. Then
\[
v^{p_k,\beta^k}\xrightarrow[k\to\infty]{} v^{p,\beta}
\]
for the locally uniform topology.
\end{lemma}

\begin{proof}
By the estimates of \Cref{lemma:estimations}, the sequence $(v^{p_k,\beta^k})_{k\ge1}$ is relatively compact for the locally uniform topology. Let $v$ be the limit of a convergent sub--sequence, which we do not relabel. Since each $v^{p_k,\beta^k}$ belongs to $\Sc$, it is immediate that $v\in\Sc$.
We now identify the limit $v$. Let $\phi:\R^d\to\R$ be a smooth test function, and fix $t\in(0,T]$. By local uniform convergence,
\begin{align*}
    \int_{\R^d} \phi(x)\,v(t,x)\,\mathrm{d}x
    =
    \lim_{k\to\infty}\int_{\R^d} \phi(x)\,v^{p_k,\beta^k}(t,x)\,\mathrm{d}x.
\end{align*}
By definition of $v^{p_k,\beta^k}$, we also have
\begin{align*}
    \int_\R \phi(x)\,v^{p_k,\beta^k}(t,x)\,\mathrm{d}x
    =
    \E\big[\phi(X^{p_k,\beta^k}_t)\big].
\end{align*}
Applying Itô's formula to $\phi(X^{p_k,\beta^k}_\cdot)$ gives
\begin{align*}
    \E\big[\phi(X^{p_k,\beta^k}_t)\big]
    &=
    \E[\phi(X_0)]
    +
    \E\left[
        \int_0^t \nabla \phi(X^{p_k,\beta^k}_s) \cdot
        \int_A b\big(s,X^{p_k,\beta^k}_s,p_k(s,X^{p_k,\beta^k}_s),m^k_s,a\big)\,
        \Lambda^{\beta^k}(s,X^{p_k,\beta^k}_s)(\mathrm{d}a)\,\mathrm{d}s
    \right]
    \\
    &\quad
    +
    \E\left[
        \frac12\int_0^t {\rm Tr} \left(\nabla^2 \phi(X^{p_k,\beta^k}_s) \sigma(s, X^{p_k,\beta^k}_s ) \sigma(s, X^{p_k,\beta^k}_s )^\top \right)\,\mathrm{d}s
    \right].
\end{align*}
Since $X^{p_k,\beta^k}_s$ has density $v^{p_k,\beta^k}(s,\cdot)$, this identity can be rewritten as
\begin{align*}
    \int_{\R^d} \phi(x)\,v^{p_k,\beta^k}(t,x)\,\mathrm{d}x
    =
    \E[\phi(\xi)]
    &+
    \int_0^t\int_{\R^d}
    \nabla \phi(x) \cdot\int_A b\big(s,x,p_k(s,x),m^k_s,a\big)\,
    \Lambda^{\beta^k}(s,x)(\mathrm{d}a)\,
    v^{p_k,\beta^k}(s,x)\,\mathrm{d}x\,\mathrm{d}s
    \\
    &+
    \frac12\int_0^t\int_{\R^d} {\rm Tr} \left(\nabla^2\phi(x) \sigma(s, x ) \sigma(s, x )^\top \right)\,v^{p_k,\beta^k}(s,x)\,\mathrm{d}x\,\mathrm{d}s.
\end{align*}
Passing to the limit as $k\to\infty$, and using:
\begin{itemize}
    \item the locally uniform convergence of $p_k$ to $p$,
    \item the locally uniform convergence of $v^{p_k,\beta^k}$ to $v$,
    \item the weak convergence of $\Lambda^{\beta^k}$ to $\Lambda^\beta$ in the topology of $\M$,
    \item the continuity and boundedness assumptions on $(b,\sigma)$,
\end{itemize}
we obtain
\begin{align*}
    \int_{\R^d} \phi(x)\,v(t,x)\,\mathrm{d}x
    =
    \E[\phi(\xi)]
    &+
    \int_0^t\int_{\R^d}
    \nabla \phi(x) \cdot \int_A b\big(s,x,p(s,x),m_s,a\big)\,
    \Lambda^\beta(s,x)(\mathrm{d}a)\,
    v(s,x)\,\mathrm{d}x\,\mathrm{d}s
    \\
    &+
    \frac12\int_0^t\int_{\R^d} {\rm Tr} \left(\nabla^2 \phi(x) \sigma(s, x ) \sigma(s, x )^\top \right)\,v(s,x)\,\mathrm{d}x\,\mathrm{d}s.
\end{align*}
Hence $v$ is a weak solution of the Fokker--Planck equation associated with the inputs $(p,\Lambda^\beta)$ and the initial condition $\Lc(\xi)$.
By uniqueness of the linear Fokker--Planck equation in this setting (see \Cref{prop:uniqueness_FP} for the general non--linear case), we conclude that
\[
v(t,x)\,\mathrm{d}x=\Lc(X^{p,\beta}_t)(\mathrm{d}x),\qquad t\in(0,T].
\]
In other words,
\[
v=v^{p,\beta}.
\]
Since the limit is uniquely identified, every convergent subsequence has the same limit. Therefore the whole sequence converges:
\[
v^{p_k,\beta^k}\xrightarrow[k\to\infty]{}v^{p,\beta}
\]
for the locally uniform topology.
\end{proof}

\medskip

For each $p\in\Sc$, we now define
\[
\Rc(p):=\big\{v^{p,\beta}:\ \beta\big\},
\qquad
\Ac(p):=\big\{(v^{p,\beta},\Lambda^\beta):\ \beta\big\}.
\]
With $m_t(\mathrm{d}x)=p(t,x)\,\mathrm{d}x$, the payoff functional may be written as
\begin{align*}
    J_m(\Lambda^\beta)
    =
    \Phi_p(\Lambda^\beta)
    &:=
    \int_{\R^d} g(x,m_T)\,v^{p,\beta}(T,x)\,\mathrm{d}x
    \quad+
    \int_{[0,T]\times\R^d\times A}
    L\big(t,x,p(t,x),m_t,a\big)\,
    \Lambda^\beta(t,x)(\mathrm{d}a)\,
    v^{p,\beta}(t,x)\,\mathrm{d}x\,\mathrm{d}t.
\end{align*}
Therefore,
\[
\sup_\beta J_m(\Lambda^\beta)
=
\sup_{(u,\Gamma)\in\Ac(p)} \Phi_p(\Gamma).
\]

By \Cref{lemma:continuity} and the continuity assumptions on the coefficients, the map
\[
(p,\Gamma)\longmapsto \Phi_p(\Gamma)
\]
is continuous.

\medskip

We next introduce the subset $\Sc_c\subset\Sc$ of functions $u\in\Sc$ such that, for every interval $[s,t]\subset(0,T]$ and every compact set $Q\subset\R$,
\begin{align*}
    \sup_{(r,x)\in[s,t]\times Q}|u(r,x)|
    +
    \sup_{\substack{(r,x),(r',x')\in[s,t]\times Q\\(r,x)\neq(r',x')}}
    \frac{|u(r,x)-u(r',x')|}
    {|r-r'|^{\alpha/2}+|x-x'|^\alpha}
    \le C,
\end{align*}
where the constants $\alpha$ and $C$ are those given by \Cref{lemma:estimations}.

\medskip

For each $p\in\Sc_c$, \Cref{lemma:estimations} implies that
\[
\Rc(p)\subset\Sc_c,
\]
and moreover $\Rc(p)$ is compact in $\Sc$ for the locally uniform topology. Thus, $\Rc$ may be viewed as a set--valued map from $\Sc_c$ into the family of compact subsets of $\Sc_c$.

\medskip
Using \Cref{lemma:continuity}, one checks that the correspondence $\Rc$ is both lower and upper hemicontinuous.

We now define the maximizing correspondence
\[
\Rc^\star(p)
:=
\left\{
u\in\Rc(p):
\ \Phi_p(\Gamma)=\sup_{(u',\Gamma')\in\Ac(p)}\Phi_p(\Gamma')
\text{ for some }(u,\Gamma)\in\Ac(p)
\right\}.
\]
In other words, $\Rc^\star(p)$ is the set of densities generated by optimal controls when the reference density is $p$.

\medskip
By Berge's maximum theorem (\cite[Theorem~17.31]{aliprantis2006infinite}), the correspondence $\Rc^\star$ has nonempty compact values and is upper hemicontinuous. In particular, its graph is closed. Since the values are convex and $\Sc_c$ is a convex compact subset of a locally convex topological vector space, we may apply the Kakutani--Fan--Glicksberg fixed--point theorem (see \cite[Corollary~17.55]{aliprantis2006infinite}) to deduce the existence of
\[
p^\star\in\Rc^\star(p^\star).
\]
Thus there exists some $\Lambda^{\beta^\star}\in\M$ such that
\[
(p^\star,\Lambda^{\beta^\star})\in\Ac(p^\star)
\]
and
\[
\Phi_{p^\star}(\Lambda^{\beta^\star})
=
\sup_{(u',\Gamma')\in\Ac(p^\star)}\Phi_{p^\star}(\Gamma').
\]
This exactly means that the pair $(p^\star,\Lambda^{\beta^\star})$ satisfies the consistency and optimality conditions of the mean field game. Therefore, by setting $m^\star_t(\mathrm{d}x):=p^\star(t,x)\,\mathrm{d}x$, we easily check that $m^\star$ is a relaxed MFG solution associated to $\Lambda^{\beta^\star}$.

\subsubsection{The case without common noise with finite--support graphon}

We now consider the case where the graphon is approximated by a step function with finite support.

\medskip

Let
\[
(I_i^k)_{1\le i\le k}
\]
be a measurable partition of $[0,1]$ into pairwise disjoint sets {\color{black} of length $\frac{1}{k}$}, and for each $1\le i\le k$, let
\[
u_i^k\in I_i^k.
\]
We assume that the graphon $\Gr$ is approximated by the step kernel $\Gr^k$ defined by
\[
\Gr^k(u,v):=\Gr(u_i^k,u_j^k),\qquad (u,v)\in I_i^k\times I_j^k.
\]
In other words, $\Gr^k$ is constant on each rectangle $I_i^k\times I_j^k$.

\medskip

Fix $k$ control kernels
\[
\Lambda^{i,\beta}:[0,T]\times\R^d\to\Pc(A),\qquad 1\le i\le k,
\]
and $k$ density profiles
\[
(p^1,\dots,p^k) \subset \Sc.
\]
For each $i\in\{1,\dots,k\}$, we consider the state process $X^i$ solving
\begin{align*}
    \mathrm{d}X^i_t
    =
    \int_A b\big(t,X^i_t,\overline p_i^k(t,X^i_t),\Rr_{\overline{m}_t}(u^k_i),a\big)\,
    \Lambda^{i,\beta}(t,X^i_t)(\mathrm{d}a)\,\mathrm{d}t
    +\sigma(t,X^i_t)\,\mathrm{d}W_t,\quad \overline{m}_t(\mathrm{d}x,\mathrm{d}u)
    :=
    \sum_{j=1}^k p^j(t,x)\,\1_{I^k_j}(u)\,\mathrm{d}u\,\mathrm{d}x. 
\end{align*}
where
\begin{align*}
    \overline p_i^k(t,x)
    :=
    \frac1k\sum_{j=1}^k \Gr^k(u_i^k,u_j^k)\,p^j(t,x).
\end{align*}
Thus, for a representative agent of class $i$, the interaction term is obtained by averaging the densities
$(p^j)_{1\le j\le k}$ against the weights induced by the step graphon $\Gr^k$.
Moreover, for each $i$, we denote by
\[
m_t^{i,\beta}=\Lc(X_t^i)
\]
the law of $X_t^i$, and we write
\[
m_t^{i,\beta}(\mathrm{d}x)=q^{i,\beta}(t,x)\,\mathrm{d}x.
\]

\medskip

Given $(p^1,\dots,p^k)$ and controls $(\Lambda^{1,\beta},\dots,\Lambda^{k,\beta})$, we define the associated payoff by
\begin{align*}
    &\Phi^k_{p^1,\dots,p^k}
    \big(q^{1,\beta},\dots,q^{k,\beta},\Lambda^{1,\beta},\dots,\Lambda^{k,\beta}\big)
    \\
    &:=
    \frac1k\sum_{i=1}^k
    \int_{\R^d} g(x,\Rr_{\overline{m}_T}(u^k_i))\,q^{i,\beta}(T,x)\,\mathrm{d}x
    +
    \frac1k\sum_{i=1}^k
    \int_{[0,T]\times\R^d\times A}
    L\big(t,x,\overline p_i^k(t,x),\Rr_{\overline{m}_t}(u^k_i),a\big)\,
    \Lambda^{i,\beta}(t,x)(\mathrm{d}a)\,
    q^{i,\beta}(t,x)\,\mathrm{d}x\,\mathrm{d}t.
\end{align*}
This is the natural finite--dimensional analogue of the payoff introduced in the graphon--free case : one averages over the $k$ classes, and each class $i$ evolves according to its own control kernel $\Lambda^{i,\beta}$ and the interaction term $\overline p_i^k$ generated by the family $(p^j)_{1\le j\le k}$.

\medskip

Our objective is to find a feedback profile
\[
\alpha=(\Lambda^{1,\alpha},\dots,\Lambda^{k,\alpha})
\]
such that the corresponding laws are consistent with the prescribed densities, that is,
\[
m_t^{i,\alpha}(\mathrm{d}x)=p^i(t,x)\,\mathrm{d}x,
\qquad 1\le i\le k,\quad t\in[0,T],
\]
and such that this profile is optimal in the sense that
\begin{align*}
    \sup_{\beta}
    \Phi^k_{p^1,\dots,p^k}
    \big(q^{1,\beta},\dots,q^{k,\beta},\Lambda^{1,\beta},\dots,\Lambda^{k,\beta}\big)
    =
    \Phi^k_{p^1,\dots,p^k}
    \big(p^1,\dots,p^k,\Lambda^{1,\alpha},\dots,\Lambda^{k,\alpha}\big).
\end{align*}

\medskip

The construction is obtained by repeating the argument of the previous subsection, but now in dimension $k$ (or dimension $d \x k$ if we incorporate the space dimension $d$), that is, for the vector of densities
\[
(p^1,\dots,p^k).
\]
More precisely, one considers the corresponding set--valued map which associates, to a $k$-tuple of density profiles, the family of attainable densities generated by admissible controls. By the same compactness estimates as before, one obtains uniform local Hölder bounds for each component $q^{i,\beta}$, and hence compactness of the associated image sets. The analogue of \Cref{lemma:continuity} also holds componentwise: if the tuple of reference densities and the tuple of control kernels converge, then the corresponding tuple of induced densities converges locally uniformly. As in the one--population case, the payoff functional is continuous with respect to these variables.

\medskip
Therefore, by applying Berge's maximum theorem to the finite--dimensional optimization problem, one obtains a nonempty compact--valued maximizing correspondence. The same fixed--point argument as in the previous subsection, now carried out on the product space of $k$ density profiles, yields the existence of a fixed point
\[
(p^1,\dots,p^k)
\]
together with a maximizing family of control kernels
\[
(\Lambda^{1,\alpha},\dots,\Lambda^{k,\alpha}),
\]
such that the consistency condition
\[
m_t^{i,\alpha}(\mathrm{d}x)=p^i(t,x)\,\mathrm{d}x,\qquad 1\le i\le k,
\]
is satisfied and the associated payoff is optimal.

\medskip

Finally, from the family $(\Lambda^{i,\alpha})_{1\le i\le k}$, we define a control kernel on
$[0,T]\times\R^d\times[0,1]$ by
\[
\Lambda^{\hat\alpha}(t,x,u):=\Lambda^{i,\alpha}(t,x),
\qquad u\in I_i^k.
\]
Since the graphon $\Gr^k$ is constant on each cell $I_i^k\times I_j^k$, this piecewise-defined control is well adapted to the finite--support graphon structure. The corresponding measure flow
\[
\overline m^\alpha_t(\mathrm{d}x,\mathrm{d}u)
=
\sum_{i=1}^k p^i(t,x)\,\1_{I^k_i}(u)\,\mathrm{d}u\,\mathrm{d}x
\]
then satisfies both the consistency and optimality requirements of the mean field game associated with $\Gr^k$.
Consequently, $\overline m^\alpha$ is a relaxed MFG solution associated with the control kernel $\Lambda^{\hat\alpha}$.

\subsubsection{The case without common noise with general graphon (and proof of {\rm \Cref{thm:existence_mfg_no_common}})} \label{sec_proof_no_common_noise}

We now consider the general case of an arbitrary graphon $\Gr$.
Let $\Gr:[0,1]^2\to\Er$ be a graphon i.e. a Borel measurable map. By standard approximation results, there exists a sequence of step graphons
\[
(\Gr^k)_{k\ge1}
\]
such that, for each $k\ge1$, the graphon $\Gr^k$ is piecewise constant on each rectangle
\[
I_i^k\times I_j^k,\qquad 1\le i,j\le k,
\]
and
\[
\Gr^k(u,v)\xrightarrow[k\to\infty]{}\Gr(u,v)
\qquad\text{for a.e. }(u,v)\in[0,1]^2.
\]

For each $k\ge1$, by the previous subsection, there exists a mean field game solution
\[
\overline m^k=(\overline m_t^k)_{t\in[0,T]}
\]
associated with some admissible control kernel $\Lambda^{\alpha^k}$ for the graphon $\Gr^k$.

\medskip

By \Cref{prop:cong_graphon} (see in Appendix), the sequence $(\overline m^k)_{k\ge1}$ is relatively compact in the weak convergent topology. Therefore, up to extraction of a subsequence (which we do not relabel), there exists a flow of probability measures
\[
\overline m=(\overline m_t)_{t\in[0,T]}
\]
such that
\[
\overline m^k \longrightarrow \overline m.
\]

Moreover, again by \Cref{prop:cong_graphon}, the limit flow admits a density representation of the form
\[
\overline m_t(\mathrm{d}x,\mathrm{d}u)=p(t,x,u)\,\mathrm{d}x\,\mathrm{d}u,\qquad t\in[0,T],
\]
and there exists an admissible control kernel
\[
\Lambda^\alpha:[0,T]\times\R^d\times[0,1]\to\Pc(A)
\]
such that, if $(X^\alpha,U)$ denotes the associated controlled state process, then
\[
\overline m_t=\Lc(X_t^\alpha,U),\qquad t\in[0,T],
\]
and $(X^\alpha,U)$ satisfies
\begin{align*}
    \mathrm{d}X_t^\alpha
    =
    \int_A
    b\Big(
        t,
        X_t^\alpha,
        \overline p(t,X_t^\alpha,U),
        \Rr_{\overline m_t}(U),
        a
    \Big)\,
    \Lambda^\alpha(t,X_t^\alpha,U)(\mathrm{d}a)\,\mathrm{d}t
    +\sigma(t,X^\alpha_t)\,\mathrm{d}W_t,
\end{align*}
where
\begin{align*}
    \overline p(t,x,u)
    :=
    \int_0^1 \Gr(u,v)\,p(t,x,v)\,\mathrm{d}v.
\end{align*}
In addition, the payoff is stable along the approximation:
\begin{align*}
    \lim_{k\to\infty} J_{\overline m^k}(\Lambda^{\alpha^k})=J_{\overline m}(\Lambda^\alpha).
\end{align*}

\medskip

It remains to verify the optimality condition for the limit pair $(\overline m,\Lambda^\alpha)$.
Let $\Lambda^\beta:[0,T]\times\R^d\times[0,1]\to\Pc(A)$ be an admissible control kernel. For simplicity, and in order to justify the approximation argument directly, let us first consider the case where $\Lambda^\beta$ is Lipschitz in $(x,u)$ (uniformly in $t$). Then we may construct a sequence of admissible control kernels
\[
(\Lambda^{\beta^k})_{k\ge1}
\]
such that, for each $k\ge1$, the map
\[
u\longmapsto \Lambda^{\beta^k}(t,x,u)
\]
is piecewise constant on each interval $I_i^k$, and
\[
\Lambda^{\beta^k}(t,x,u)\xrightarrow[k\to\infty]{}\Lambda^\beta(t,x,u)
\qquad\text{for a.e. }(t,x,u)\in[0,T]\times\R^d\times[0,1].
\]
For instance, one may take
\[
\Lambda^{\beta^k}(t,x,u):=\Lambda^\beta(t,x,u_i^k),\qquad u\in I_i^k.
\]
The Lipschitz continuity of $\Lambda^\beta$ ensures that this approximation is consistent as the mesh of the partition tends to zero.

By the same stability result used above (namely the continuity statement underlying \Cref{prop:cong_graphon}), the associated payoffs satisfy
\begin{align*}
    \lim_{k\to\infty} J_{\overline m^k}(\Lambda^{\beta^k})=J_{\overline m}(\Lambda^\beta).
\end{align*}

Since, for each $k\ge1$, the pair $(\overline m^k,\Lambda^{\alpha^k})$ is a mean field game solution for the graphon $\Gr^k$, we have
\[
J_{\overline m^k}(\Lambda^{\alpha^k})\ge J_{\overline m^k}(\Lambda^{\beta^k}),\qquad \forall k\ge1.
\]
Passing to the limit as $k\to\infty$ yields
\begin{align*}
    J_{\overline m}(\Lambda^\alpha)
    =
    \lim_{k\to\infty}J_{\overline m^k}(\Lambda^{\alpha^k})
    \ge
    \lim_{k\to\infty}J_{\overline m^k}(\Lambda^{\beta^k})
    =
    J_{\overline m}(\Lambda^\beta).
\end{align*}
Hence
\[
J_{\overline m}(\Lambda^\alpha)\ge J_{\overline m}(\Lambda^\beta)
\]
for every admissible Lipschitz control kernel $\Lambda^\beta$.

\medskip

Under the standing density/approximation assumptions on the admissible class, Lipschitz controls are dense in the full class of admissible controls for the topology used in the stability result. Therefore, by the density argument of \Cref{prop:cong_graphon} and continuity of the payoff functional, the above inequality extends from Lipschitz controls to every admissible control kernel $\Lambda^\beta$.
Consequently, the pair $(\overline m,\Lambda^\alpha)$ satisfies the optimality condition. Since we already know that
\[
\overline m_t=\Lc(X_t^\alpha,U),\qquad t\in[0,T],
\]
the consistency condition also holds.
We conclude that $\overline m$ is a relaxed mean field game solution associated with $\Lambda^\alpha$ for the graphon $\Gr$.

\subsubsection{The case without common noise with general graphon (and proof of {\rm \Cref{thm:existence_mfg_no_common}}, the case of strict MFG)}
\medskip
Let us now turn to the existence of strict MFG solutions. When \Cref{assum:conv_cond} is satisfied, proof is identical to the proof of existence of strict MFG solution as in \Cref{thm:existence} (see the proof of \Cref{thm:existence} in the next section). We therefore focus on the proof when \Cref{assum:hamiltonian_cond} is satisfied. Let $\overline{m}$ be the previous relaxed MFG solution associated to $\Lambda^\alpha$. Consider the BSDE
\begin{align} \label{eq:BSDE_bellman}
    Y_t
    &=
    g(X_0+W_T,\Rr_{\overline{m}_T}(U))
    +
    \int_t^T
            H(s,X_0+W_s, \overline{p}_{\overline{m}_s}(U),\Rr_{\overline{m}_s}(U),Z_s)
        \mathrm{d}s
    -\int_t^T Z_s\,\mathrm{d}W_s,
\end{align}

Under standard Markovian BSDE assumptions (e.g.\ coefficients measurable and Lipschitz in the state, square integrability, see for instance \citeauthor*{JinRepresentation2002} \cite{JinRepresentation2002}), one may apply the usual Markovian representation result to obtain that
\[
    Z_t = v(t,X_0+W_t,U), \qquad \text{for a.e.\ }t\in[0,T],
\]
for some Borel map $v:[0,T]\times\R^d\x [0,1]\to\R^{d}$.

\medskip
Fix a Borel map $\Lambda^\beta:[0,T] \x \R^d \x [0,1] \to \Pc(A)$. Let $Z^\beta$ be the stochastic exponential defined by $Z^\beta_0=1$ and
\[
    \mathrm{d}Z^\beta_t
    =
    Z^\beta_t
    \left(
        \int_{A}
            \sigma(t,X_0+W_t)^{-1}\,b\bigl(t,X_0+W_t, \overline{p}_{\overline{m}_t}(X_0+W_t,U),\Rr_{\overline{m}_t}(U), a) \,\Lambda^\beta(t,X_0+W_t,U)(\mathrm{d}a)\bigr)
    \right)\cdot \mathrm{d}W_t,
\]
and define $\P^\beta$ by $\mathrm{d}\P^\beta := Z^\beta_T\,\mathrm{d}\P$.
By Girsanov's theorem and the fact that $W \perp (X_0,U)$, under $\P^\beta$ the law of $(X_0+W,U)$ coincides with the law of the controlled state process $(X^{\overline{m},\beta},U)$, i.e.
\[
    \Lc^{\P^\beta}(X_0+W_\cdot, U)=\Lc(X^{\overline{m},\beta},U).
\]
Standard BSDE arguments then yield the characterization (see for instance \citeauthor*{el1997backward} \cite[Proposition 3.4.]{el1997backward})
\[
    \sup_{\beta} J_{\overline{m}}(\Lambda^\beta)=\E[Y_0],
\]
and, moreover,
\begin{align*}
    J_{\overline{m}}(\Lambda^\alpha)-\E[Y_0]
    =
    \E^{\P^\alpha}\!\bigg[
        \int_0^T
        \Big(
            &-H(s,X_0+W_s, \overline{p}_{\overline{m}_s}(X_0+W_s,U),\Rr_{\overline{m}_s}(U),Z_s)
            \\
            &+
            \int_{A}
                h\bigl(s,X_0+W_s, \overline{p}_{\overline{m}_s}(X_0+W_s,U),\Rr_{\overline{m}_s}(U),Z_s,a\bigr)
            \Lambda^\alpha(s,X_0+W_s)(\mathrm{d}a)
        \Big)\mathrm{d}s
    \bigg].
\end{align*}
Since $(\overline{m},\Lambda^\alpha)$ is an equilibrium, $ J_{\overline{m}}(\Lambda^\alpha)=\sup_{\beta} J_{\overline{m}}(\Lambda^\beta)$, hence the left--hand side is zero. By definition of $H$ as a supremum, the integrand inside the last expectation is non--positive, and therefore it must vanish $\mathrm{d}s\otimes\P^\alpha$--a.e.:
\[
    H(s,X_0+W_s, \overline{p}_{\overline{m}_s}(X_0+W_s,U),\Rr_{\overline{m}_s}(U),Z_s)
            =
            \int_{A}
                h\bigl(s,X_0+W_s, \overline{p}_{\overline{m}_s}(X_0+W_s,U),\Rr_{\overline{m}_s}(U),Z_s,a\bigr)
            \Lambda^\alpha(s,X_0+W_s)(\mathrm{d}a).
\]
Using the representation $Z_s=v(s,X_0+W_s, U)$ and the uniqueness of the maximizer provided by $\widehat a$,
we conclude that, for a.e.\ $(s,x)$,
\begin{align}\label{eq:dirac_solution}
    \Lambda^\alpha(t,x,u)
    =
    \delta_{\hat a\bigl(t,x, \overline{p}_{\overline{m}_t}(x,u),\Rr_{\overline{m}_s}(u),v(s,x,u)\bigr)},
    \qquad\text{for a.e.\ }(t,x,u).
\end{align}
This proves the claim.

\subsubsection{A characterization of a MFG solution (proof of \Cref{prop:representation_mfg_stric} and \Cref{cor:representation_mfg_stric})}

This is essentially a direct application of the previous section together with \Cref{eq:dirac_solution}. The only additional point is to justify the equivalence between \Cref{eq:BSDE_bellman} and \Cref{eq:gradient_cond}.

\medskip

Let \((Y,Z)=(Y,Z^1,\dots,Z^d)\) be a solution of \Cref{eq:BSDE_bellman}. By {\color{black}\cite[Theorem 4.2.]{JinRepresentation2002}}, the martingale integrand admits a Markovian representation of the form
\[
Y_t = \mathcal{U}(t,X_0+W_t,U),\qquad Z^i_t = \Vc^i(t,X_0+W_t,U),\qquad i=1,\dots,d,
\]
for some measurable map $\Uc$ and \(\Vc=(\Vc^1,\dots,\Vc^d)\) satisfying \Cref{eq:gradient_cond}. Since \(U\) is independent of \(W\), the presence of the additional label variable does not create any difficulty for applying their argument.  \Cref{eq:gradient_cond} essentially comes down to the fact that $\partial_{x_i} \Uc=\Vc^i$ by using an integration by part.

\medskip

Conversely, assume that \(v=(v^1,\dots,v^d)\) is a solution of \Cref{eq:gradient_cond}. Define
\[
\widehat{Z}_t^i := v^i(t,X_0+W_t,U),\qquad i=1,\dots,d,
\]
and
\begin{align*}
    \widehat{Y}_t
    &:=
    \E\Bigg[
        g\bigl(X_0+W_T,\Rr_{\overline{m}_T}(U)\bigr)
        \\
        &\hspace{1.4cm}
        +
        \int_t^T
        H\Big(
            s,
            X_0+W_s,
            \overline{p}_{\overline{m}_s}(X_0+W_s,U),
            \Rr_{\overline{m}_s}(U),
            v(s,X_0+W_s,U)
        \Big)\,\mathrm{d}s
        \,\Big|\,
        X_0,\,(W_r)_{0\le r\le t},\,U
    \Bigg]
    \\
    &=\widehat{\Uc}(t,X_0+W_t,U),
\end{align*}
with a map $\widehat{\Uc}$ obtained by calculating the conditional expectation using the independence of the increment of the Brownian motion. We can check $\partial_{x_i} \widehat{\Uc}=v^i$.
By construction, \(\widehat{Y}\) is adapted to the filtration generated by \((X_0,W,U)\). Applying the martingale representation theorem, together with the definition of \(\widehat{Z}\), and then computing \(\mathrm{d}\widehat{Y}_t\), we obtain that
\[
(\widehat{Y},\widehat{Z})=(\widehat{Y},\widehat{Z}^1,\dots,\widehat{Z}^d)
\]
solves \Cref{eq:BSDE_bellman}.

\medskip

Therefore, \Cref{eq:BSDE_bellman} and \Cref{eq:gradient_cond} are equivalent. The desired characterization of the MFG solution then follows directly from the representation result established in the previous section.

{\color{black}

\subsubsection{Existence : the case with common noise (and proof of \Cref{thm:existence})}

We now explain how to construct a relaxed mean field game solution in the presence of common noise. The strategy relies on the classical translation argument. More precisely, for each fixed realization of the common noise, we first consider the associated deterministic problem obtained by freezing the common--noise path. This yields a family of deterministic mean field game solutions parametrized by the common noise trajectory. We then reconstruct a random equilibrium flow by means of a measurable selection procedure in the spirit of \cite{karoui2013capacities}. Similar measurable selection arguments have previously been employed in the context of mean field control in \cite{djete2019mckean} and for mean field games in \cite{djete2023stackelbergmeanfieldgames}.

\medskip

Let $c\in C([0,T];\R^d)$ be a deterministic path. For such a path, we define the translated coefficients
\begin{align*}
    b^c(t,y,e,M,a)
    &:=
    b\bigl(t,\,y+ c(t),\,e,\,M^c_t,\,a\bigr),\qquad \sigma^c(t,y)
    :=
    \sigma\bigl(t,\,y+c(t)\bigr)
\end{align*}
where
\begin{align*}
    M^c_t
    :=
    \Lc^M\bigl(E,\,X+c(t)\bigr).
\end{align*}
In other words, if under the probability measure $M$ the pair $(E,X)$ has law $M$, then $M^c_t$ denotes the law of $(E,X+ c(t))$, that is, the push--forward of $M$ under the map $(e,x)\longmapsto (e,x+ c(t)).$

\medskip
We recall that we denote by $p_{m}(x,u)$ is the density of a probability $m$ when it exists.
For a candidate reference flow $\overline m=(\overline m_t)_{t\in[0,T]}$ with
\[
q_{\overline{m}_t}(y,u):=p_{({\rm Id}_1 + c(t),\,{\rm Id}_2) \# \overline{m}_t}(y+ c(t),u),
\]
where the map $({\rm Id}_1 + c(t),\,{\rm Id}_2)$ is such that $({\rm Id}_1 + c(t),\,{\rm Id}_2)(x,u)=(x+c(t),u)$ for all $(x,u) \in \R^d \x [0,1]$. We define, as before,
\begin{align*}
    \overline q_{\overline{m}_t}(x,u)
    :=
    \int_0^1 \Gr(u,v)\,q_{\overline{m}_t}(x,v)\,\mathrm{d}v.
\end{align*}
We also use the notation
\[
\Rr_t(u)=\Rr_{({\rm Id}_1 + c(t),\, {\rm Id}_2) \# \overline{m}_t}(u)
\]
for the corresponding graphon--induced environment law.

\medskip

Let $\widetilde{m}=(\widetilde{m}_t)_{t\in[0,T]}$ be another deterministic flow of probability measures on $\R^d\times[0,1]$, and let
\[
\Lambda:[0,T]\times\R^d\times[0,1]\to\Pc(A)
\]
be an admissible control kernel. For each $t\in[0,T]$ and each smooth compactly supported test function $f\in C_c^\infty(\R^d\times[0,1]),$
we define
\begin{align*}
    N^{t,f}_{c,\overline m}[\widetilde{m},\Lambda]
    :=
    \langle f,\widetilde{m}_t\rangle
    -\langle f,\nub\rangle
    &-
    \int_0^t\int_{\R^d\times[0,1]\times A}
    \nabla_x f(x,u)\cdot b^c\bigl(s,x,\overline q_{\overline{m}_s}(x,u),\Rr_s(u),a\bigr)
    \,\Lambda(s,x,u)(\mathrm{d}a)\,\widetilde{m}_s(\mathrm{d}x,\mathrm{d}u)\,\mathrm{d}s
    \\
    &-
    \frac12
    \int_0^t\int_{\R^d\times[0,1]}
    {\rm Tr}\!\Big[\sigma^c(s,x)\sigma^c(s,x)^\top \nabla_x^2 f(x,u)\Big]\,
    \widetilde{m}_s(\mathrm{d}x,\mathrm{d}u)\,\mathrm{d}s.
\end{align*}
Here,
\[
\langle f,\widetilde{m}_t\rangle:=\int_{\R^d\times[0,1]} f(x,u)\,\widetilde{m}_t(\mathrm{d}x,\mathrm{d}u).
\]
Thus, the condition
\[
N^{t,f}_{c,\overline m}[\widetilde{m},\Lambda]=0
\]
for all $t$ and all smooth test functions $f$ is precisely the weak formulation of the Fokker--Planck equation associated with the frozen common--noise path $c$ and the reference flow $\overline m$.

\medskip

For each fixed pair $(t,f)$, the map
\begin{align*}
    C([0,T];\R^d)\times C([0,T];\Pc(\R^d\times[0,1]))^2\times \M
    \ni
    (c,\overline m,\widetilde{m},\Lambda)
    \longmapsto
    N^{t,f}_{c,\overline m}[\widetilde{m},\Lambda]
    \in\R
\end{align*}
{\color{black}is Borel measurable (see \Cref{rm:topology_M} for the topology considered for $\M$ where we just add the uniform distribution on $[0,1]$ for the associated component)}. Indeed, all operations involved (evaluation at time $t$, integration against continuous test functions, composition with the translated coefficients, and integration with respect to the measure arguments) are Borel measurable under the standing continuity and boundedness assumptions on the coefficients.

\medskip

We now define
\begin{align*}
    \Rc^{t,f}
    :=
    \Bigl\{
        (c,\overline m,\widetilde{m},\Lambda)
        \;:\;
        N^{t,f}_{c,\overline m}[\widetilde{m},\Lambda]=0
    \Bigr\}.
\end{align*}
Let $\X$ be a fixed countable dense subset of
\[
[0,T]\times C_c^\infty(\R^d\times[0,1]).
\]
We then set
\begin{align*}
    \Rc
    :=
    \bigcap_{(t,f)\in\X}\Rc^{t,f}.
\end{align*}
Since each $\Rc^{t,f}$ is Borel, the set $\Rc$ is Borel measurable as a countable intersection of Borel sets.
For $(c,\overline m,\widetilde{m},\Lambda)\in \Rc$, we define the frozen--path payoff functional by
\begin{align*}
    \Phi_{c,\overline m}(\widetilde{m},\Lambda)
    &:=
    \int_{\R^d\times[0,1]}
    g^c\bigl(x,\Rr_T(u)\bigr)\,\widetilde{m}_T(\mathrm{d}x,\mathrm{d}u)
    +
    \int_0^T\int_{\R^d\times[0,1]\times A}
    L^c\bigl(t,x,\overline q_{\overline{m}_t}(x,u),\Rr_t(u),a\bigr)\,
    \Lambda(t,x,u)(\mathrm{d}a)\,
    \widetilde{m}_t(\mathrm{d}x,\mathrm{d}u)\,\mathrm{d}t,
\end{align*}
where $g^c$ and $L^c$ denote the translated terminal and running reward functions associated with the path $c$.

We then introduce the value function
\begin{align*}
    V(c,\overline m)
    :=
    \sup\Bigl\{
        \Phi_{c,\overline m}(\widetilde{m},\Lambda)
        \;:\;
        (c,\overline m,\widetilde{m},\Lambda)\in\Rc
    \Bigr\}.
\end{align*}

\medskip

Since $\Rc$ is Borel measurable and the map
\begin{align*}
    (c,\overline m,\widetilde{m},\Lambda)
    \longmapsto
    \Phi_{c,\overline m}(\widetilde{m},\Lambda)
\end{align*}
is Borel measurable, the measurable projection theorem of \cite[Proposition 2.21]{karoui2013capacities} implies that the map
\[
(c,\overline m)\longmapsto V(c,\overline m)
\]
is universally measurable.

\medskip

We now define the set of frozen--path mean field game solutions:
\begin{align*}
    \Rc^\star
    :=
    \Bigl\{
        (c,\overline m,\Gamma)
        \;:\;
        V(c,\overline m)\le \Phi_{c,\overline m}(\overline m,\Gamma)
    \Bigr\}
    \cap
    \bigcap_{(t,f)\in\X}
    \Bigl\{
        (c,\overline m,\Gamma)
        \;:\;
        N^{t,f}_{c,\overline m}[\overline m,\Gamma]=0
    \Bigr\}.
\end{align*}
Thus, $(c,\overline m,\Gamma)\in\Rc^\star$ means that:
\begin{itemize}
    \item the pair $(\overline m,\Gamma)$ satisfies the frozen Fokker--Planck equation associated with $c$;
    \item the same pair is optimal among all admissible competitors for the frozen problem.
\end{itemize}
Since $V$ is universally measurable and the constraints are Borel, the set $\Rc^\star$ is universally measurable and non--empty by an application of {\rm \Cref{thm:existence_mfg_no_common}} to the framework with shifted coefficients (see the proof in \Cref{sec_proof_no_common_noise}).

\medskip

By \cite[Theorem 2.20]{karoui2013capacities}, there exists a universally measurable map
\[
\Psi:C([0,T];\R^d)\to C([0,T];\Pc(\R^d\times[0,1]))\times\M
\]
such that, for every $c\in C([0,T];\R^d)$,
\[
\Psi(c)=(\overline m^c,\Gamma^c)
\qquad\text{and}\qquad
(c,\Psi(c))\in\Rc^\star.
\]
In particular, for every deterministic path $c$, the pair $(\overline m^c,\Gamma^c)$ is a solution of the frozen mean field game problem associated with $c$, and
\begin{align} \label{eq:FP-frozen}
N^{t,f}_{c,\overline m^c}[\overline m^c,\Gamma^c]=0,
\qquad \forall t\in[0,T],\ \forall f\in C_c^\infty(\R^d\times[0,1]).
\end{align}

\medskip

We now lift this family of frozen solutions into a genuine common--noise solution.
We recall that $W^\circ$ is a Brownian motion, independent of $(W,X_0,U)$. For each $\omega\in\Omega$ and $t\in[0,T]$, define the random flow
\begin{align*}
    \mub_t(\omega)(\mathrm{d}x,\mathrm{d}u)
    :=
    \int_{\R^d}
    \delta_{y+\sigma_\circ W_t^\circ(\omega)}(\mathrm{d}x)\,
    \overline m_t^{\sigma_0W^\circ(\omega)}(\mathrm{d}y,\mathrm{d}u).
\end{align*}
That is, $\mub_t(\omega)$ is obtained by translating the frozen--path solution $\overline m_t^{\sigma_\circ\, W^\circ(\omega)}$ by the common--noise displacement $\sigma_0W_t^\circ(\omega)$.
Similarly, define the random relaxed control measure by
\begin{align*}
    \widehat\Gamma(t,\omega)(\mathrm{d}a,\mathrm{d}x,\mathrm{d}u)
    :=
    \int_{\R^d}
    \Gamma^{\sigma_\circ W^\circ(\omega)}\bigl(t,x-\sigma_\circ W_t^\circ(\omega),u\bigr)(\mathrm{d}a)\,
    \mub_t(\omega)(\mathrm{d}x,\mathrm{d}u).
\end{align*}
In other words, at the translated state $x$, we apply the frozen--path feedback corresponding to the shifted state
\[
x-\sigma_\circ W_t^\circ(\omega).
\]

\medskip
Notice that the processes $(\mub,\widehat\Gamma)$ are not necessarily adapted to the common noise filtration i.e. $(\sigma\{ \sigma_\circ\, W^\circ_s:\,s \in [0,t]\})_{t \in [0,T]}$. However, they are adapted to their own natural filtration.
Since, for each path $c$, the measure $\overline m_t^c$ is absolutely continuous with respect to the Lebesgue measure in the $x$--variable, the same is true for the translated random measure $\mub_t(\omega)$. We check that
\begin{align*}
    q_{\overline{m}^{\sigma_\circ  W^\circ_t(\om)}}(y,u)=p_{\mub_t(\om)}(y+\sigma_{\circ} W^\circ_t,u)
\end{align*}

\medskip

We then define the conditional feedback kernel
\begin{align} \label{eq:mfg_control}
    \widehat\Gamma(t,x,u,\mub_t)
    :=
    \E\Big[
        \Gamma^{\sigma_\circ W^\circ}\bigl(t,x-\sigma_\circ W_t^\circ,u\bigr)(\mathrm{d}a)
        \,\big|\,\mub_t
    \Big].
\end{align}
This is a Borel kernel on
\[
[0,T]\times\R^d\times[0,1]\times \Pc(\R^d\times[0,1]).
\]

\medskip

Let us now consider the filtration
\[
\Gc_t:=\sigma\{\sigma_\circ W_s^\circ,\mub_s:\ 0\le s\le t\}.
\]
Then, the superposition principle associated to the frozen Fokker--Planck in \Cref{eq:FP-frozen} combined with the use of \cite[Corollary 1.6]{Lacker-Shkolnikov-Zhang_2020} (see also \cite[Theorem 1.5]{Lacker-Shkolnikov-Zhang_2020} combined with \cite[Theorem 1.3]{Lacker-Shkolnikov-Zhang_2020}), the law of the random flow satisfies
\[
\Lc(\mub_t)=\Lc\bigl(\Lc(S_t,U\mid \Gc_t)\bigr),
\]
where $(S,U)$ solves : $S_0=\xi$, $(S_0,U) \perp W \perp (\sigma_\circ\,W^\circ, \mu)$ and
\begin{align*}
    \mathrm{d}S_t
    =
    \int_A
    b\bigl(
        t,
        S_t,
        \overline p_{\mub_t}(S_t,U),
        \Rr_{\mub_t}(U),
        a
    \bigr)\,
    \widehat\Gamma(t,S_t,U,\mub_t)(\mathrm{d}a)\,\mathrm{d}t
    +\sigma(t,S_t)\,\mathrm{d}W_t
    +\sigma_\circ\,\mathrm{d}W_t^\circ,
\end{align*}
with
\[
\overline p_{\mub_t}(x,u)
:=
\int_0^1 \Gr(u,v)\,p_{\mub_t}(x,v)\,\mathrm{d}v.
\]
Thus, the law of the random flow \(\mu_t\) coincides with the law of the conditional distribution of the controlled state.

\medskip

We now verify optimality.
Let
\[
\Lambda^\beta:[0,T]\times\R^d\times[0,1] \times\Pc(\R^d\times[0,1])\to\Pc(A)
\]
be an arbitrary admissible feedback kernel, that is, {\color{black}$\Lambda^\beta\in\M$}.
Let $(Y,U)$ be the corresponding state process solving
\begin{align*}
    \mathrm{d}Y_t
    =
    \int_A
    b\bigl(
        t,
        Y_t,
        \overline p_{\mub_t}(Y_t,U),
        \Rr_{\mub_t}(U),
        a
    \bigr)\,
    \Lambda^\beta(t,Y_t,U,\mub_t)(\mathrm{d}a)\,\mathrm{d}t
    +\sigma(t,Y_t)\,\mathrm{d}W_t
    +\sigma_\circ\,\mathrm{d}W_t^\circ.
\end{align*}

For each $\omega\in\Omega$, define the translated conditional law
\begin{align*}
    \ell_t(\omega)
    :=
    \Lc\bigl(Y_t-\sigma_\circ W_t^\circ\,, U\big|\,\Gc_t\bigr)(\omega) = \Lc\bigl(Y_t-\sigma_\circ W_t^\circ\,, U\big|\,\Gc_T\bigr)(\omega),
\end{align*}
and define the associated frozen--path control by
\begin{align*}
    \Lambda^{\sigma_\circ  W^\circ(\omega)}(t,x,u)
    :=
    \Lambda^\beta\bigl(t,x+\sigma_\circ W_t^\circ(\omega),u,\mub_t(\omega)\bigr).
\end{align*}
Then, for almost every $\omega$, the pair
\[
\bigl(\ell(\omega),\,\Lambda^{\sigma_\circ W^\circ(\omega)}\bigr)
\]
is an admissible competitor for the frozen problem associated with the path $c=\sigma_\circ W^\circ(\omega)$.
Therefore, by construction of the frozen--path solution $(\overline m^{\sigma_\circ W^\circ(\omega)},\Gamma^{\sigma_\circ W^\circ(\omega)})$, we have
\begin{align*}
    \Phi_{\sigma_\circ W^\circ(\omega),\,\overline m^{\sigma_\circ W^\circ(\omega)}}
    \bigl(
        \ell(\omega),\Lambda^{\sigma_\circ W^\circ(\omega)}
    \bigr)
    \le
    \Phi_{\sigma_\circ W^\circ(\omega),\,\overline m^{\sigma_\circ W^\circ(\omega)}}
    \bigl(
        \overline m^{\sigma_\circ W^\circ(\omega)},\Gamma^{\sigma_\circ W^\circ(\omega)}
    \bigr)
\end{align*}
for $\P$-a.e.\ $\omega$.
Integrating over $\omega$, we obtain
\begin{align*}
    J_{\mub}(\Lambda^\beta)
    =
    \int_\Omega
    \Phi_{\sigma_\circ W^\circ(\omega),\,\overline m^{\sigma_\circ W^\circ(\omega)}}
    \bigl(
        \ell(\omega),\Lambda^{\sigma_\circ W^\circ(\omega)}
    \bigr)\,
    \P(\mathrm{d}\omega),
\end{align*}
and
\begin{align*}
    J_{\mub}(\widehat\Gamma)
    =
    \int_\Omega
    \Phi_{\sigma_\circ W^\circ(\omega),\,\overline m^{\sigma_\circ W^\circ(\omega)}}
    \bigl(
        \overline m^{\sigma_\circ W^\circ(\omega)},\Gamma^{\sigma_\circ W^\circ(\omega)}
    \bigr)\,
    \P(\mathrm{d}\omega).
\end{align*}
Hence
\[
J_{\mub}(\Lambda^\beta)\le J_{\mub}(\widehat\Gamma),
\qquad \forall \Lambda^\beta\in\widehat{\M}.
\]

This proves that the random flow $\mub$ satisfies the optimality condition. Since we already identified that the law of $\mub_t$ is the law of the conditional law of the state given the common--noise filtration, the consistency condition also holds.
We conclude that $\mub$ is a relaxed mean field game solution in the presence of common noise.

\medskip
Let us turn to the existence of strict MFG game. By \Cref{assum:conv_cond}, there exists a Borel measurable map $\widetilde{\alpha}: [0,T] \x \R^d \x \R \x \Pc(\Er \x \R^d) \x  \Pc(A) \to A$ such that 
\begin{align*}
     b\left(t,x,p,r,\widetilde{\alpha}(t,x,p,r,\pi) \right)
     =
     \int_A b(t,x,p,r,a) \pi(\mathrm{d}a)\quad\mbox{and}\quad \int_A L(t,x,p,r,a) \pi(\mathrm{d}a) \le L(t,x,p,r,\widetilde{\alpha}(t,x,p,r,\pi)).
\end{align*}

With the control $\widehat{\Gamma}$ of \eqref{eq:mfg_control}, we set 
\begin{align*}
    \overline{\alpha}(t,x,u,\overline{m}):=\widetilde{\alpha}\left(t,x,\overline{p}_{\overline{m}}(x,u),\Rr_{\overline{m}}(u),\widehat{\Gamma}(t,x,u,\overline{m}) \right)
\end{align*}
and the couple $(X,U)$,
\begin{align*}
    \mathrm{d}X_t
    =
    b\bigl(
        t,
        X_t,
        \overline p_{\mub_t}(X_t,U),
        \Rr_{\mub_t}(U),
        \overline{\alpha}(t,X_t,U,\mub_t)
    \bigr)\,
    \mathrm{d}t
    +\sigma(t,X_t)\,\mathrm{d}W_t
    +\sigma_\circ\,\mathrm{d}W_t^\circ.
\end{align*}
Again by \cite[Corollary 1.6]{Lacker-Shkolnikov-Zhang_2020}, we have $\Lc( \Lc(X_t,U \mid \Gc_t)) =\Lc(\mub_t)$ for each $t \in [0,T]$, and $J_{\mub}(\widehat{\Gamma}) \le J_{\mu}(\Lambda^{\overline{\alpha}})$ with $\Lambda^{\overline{\alpha}}(t,x,u,\overline{m})=\delta_{\overline{\alpha}(t,x,u,\overline{m})}$. This is enough to deduce that $\mub$ is a strict MFG solution associated to $\overline{\alpha}$.

}

\subsection{Uniqueness of MFG solution (proof of {\rm \Cref{prop:uniqueness}})}

Let $(\mub^1,\Lambda^1)$ and $(\mub^2,\Lambda^2)$ be two mean field game solutions. We show that
\[
\mub^1=\mub^2
\qquad\text{a.e.}
\]
Throughout this subsection, we work under \Cref{assum:uniqueness_cond}.

\medskip

For each $i\in\{1,2\}$, the consistency property of the MFG solution yields
\[
\mub^i_t=\Lc(X^i_t,U\mid \Gc^i_t),
\qquad
\Gc^i_t:=\sigma\big\{\mub^i_s,\sigma_\circ W^\circ_s:\ s\in[0,t]\big\},
\]
where the state process $X^i$ solves
\begin{align*}
    \mathrm{d}X^i_t
    =
    \int_A b(t,X^i_t,a)\,
    \Lambda^i(t,X^i_t,U,\mub^i_t)(\mathrm{d}a)\,\mathrm{d}t
    +\sigma(t,X^i_t)\,\mathrm{d}W_t
    +\sigma_\circ\,\mathrm{d}W^\circ_t.
\end{align*}

\medskip

We now evaluate the control $\Lambda^2$ against the reference flow $\mub^1$. To this end, we define
\begin{align*}
    \overline{J}_{\mub^1}(\Lambda^2)
    :=
    \E\Bigg[
        g\big(X^2_T,\Rr_{\mub^1_T}(U)\big)
        +
        \int_0^T\int_A
        \overline{L}(t,X^2_t,a)\,
        \Lambda^2(t,X^2_t,U,\mub^2_t)(\mathrm{d}a)\,\mathrm{d}t
    \Bigg]
    +
    \E\Bigg[
        \int_0^T
        \underline{L}\Big(
            t,X^2_t,
            \overline{p}_{\mub^1_t}(X^2_t,U),
            \Rr_{\mub^1_t}(U)
        \Big)\,\mathrm{d}t
    \Bigg].
\end{align*}
Similarly, one defines $\overline{J}_{\mub^2}(\Lambda^1)$ by exchanging the roles of $(\mub^1,\Lambda^1)$ and $(\mub^2,\Lambda^2)$.

\medskip

One may use the admissible control $\Lambda^2$ as a competitor in the optimization problem associated with the flow $\mub^1$. This yields
\[
\overline{J}_{\mub^1}(\Lambda^2)
\le
\sup_{\beta} J_{\mub^1}(\Lambda^\beta)
=
J_{\mub^1}(\Lambda^1).
\]
Likewise,
\[
\overline{J}_{\mub^2}(\Lambda^1)
\le
\sup_{\beta} J_{\mub^2}(\Lambda^\beta)
=
J_{\mub^2}(\Lambda^2).
\]
Summing the two inequalities, we obtain
\begin{align*}
    0
    &\le
    J_{\mub^1}(\Lambda^1)-\overline{J}_{\mub^1}(\Lambda^2)
    +
    J_{\mub^2}(\Lambda^2)-\overline{J}_{\mub^2}(\Lambda^1).
\end{align*}

\medskip

We now expand the right--hand side. Since the terms involving $\overline{L}$ cancel pairwise, we get
\begin{align*}
    0
    &\le
    \E\Big[
        g\big(X^1_T,\Rr_{\mub^1_T}(U)\big)
        -
        g\big(X^2_T,\Rr_{\mub^1_T}(U)\big)
        +
        g\big(X^2_T,\Rr_{\mub^2_T}(U)\big)
        -
        g\big(X^1_T,\Rr_{\mub^2_T}(U)\big)
    \Big]
    \\
    &\quad+
    \E\Bigg[
        \int_0^T
        \underline{L}\Big(
            t,X^1_t,\overline{p}_{\mub^1_t}(X^1_t,U),\Rr_{\mub^1_t}(U)
        \Big)
        -
        \underline{L}\Big(
            t,X^2_t,\overline{p}_{\mub^1_t}(X^2_t,U),\Rr_{\mub^1_t}(U)
        \Big)\,\mathrm{d}t
    \Bigg]
    \\
    &\quad+
    \E\Bigg[
        \int_0^T
        \underline{L}\Big(
            t,X^2_t,\overline{p}_{\mub^2_t}(X^2_t,U),\Rr_{\mub^2_t}(U)
        \Big)
        -
        \underline{L}\Big(
            t,X^1_t,\overline{p}_{\mub^2_t}(X^1_t,U),\Rr_{\mub^2_t}(U)
        \Big)\,\mathrm{d}t
    \Bigg].
\end{align*}

\medskip

Next, we use the conditional independence structure of the MFG solution. Since, by construction,
\[
(\mub^1,\mub^2)\perp (X_0,U)\perp (W,W^\circ),
\]
we may enlarge the conditioning and write, for each $i\in\{1,2\}$,
\[
\mub^i_t=\Lc(X^i_t,U\mid \Gch_t),
\qquad
\Gch_t:=\sigma\big\{\mub^1_s,\mub^2_s,\sigma_\circ W^\circ_s:\ s\in[0,t]\big\}.
\]
Therefore, conditioning with respect to $\Gch_t$ and using the definition of the conditional laws, the previous inequality becomes
\begin{align*}
    0
    &\le
    \E\Bigg[
        \int_{\R^d\times[0,1]}
        \Big(
            g\big(x,\Rr_{\mub^1_T}(u)\big)
            -
            g\big(x,\Rr_{\mub^2_T}(u)\big)
        \Big)
        \Big(
            \mub^1_T-\mub^2_T
        \Big)(\mathrm{d}x,\mathrm{d}u)
    \Bigg]
    \\
    &\quad+
    \E\Bigg[
        \int_0^T
        \int_{\R^d\times[0,1]}
        \Big(
            \underline{L}\big(
                s,x,\overline{p}_{\mub^1_s}(x,u),\Rr_{\mub^1_s}(u)
            \big)
            -
            \underline{L}\big(
                s,x,\overline{p}_{\mub^2_s}(x,u),\Rr_{\mub^2_s}(u)
            \big)
        \Big)
        \Big(
            \mub^1_s-\mub^2_s
        \Big)(\mathrm{d}x,\mathrm{d}u)\,\mathrm{d}s
    \Bigg].
\end{align*}

\medskip
At this stage, the right--hand side is exactly the monotonicity expression associated with the pair $(g,\underline L)$ evaluated at the two flows $\mub^1$ and $\mub^2$.
If the optimization admits a unique solution and $\mub^1 \neq \mub^2$ on a set of positive probability, the previous non--negative inequality is a positive inequality. However, by \Cref{assum:uniqueness_cond}, this previous inequality is a non--positivity. This leads to a contradiction. Therefore, we deduce that  $\mub^1 = \mub^2$. {\color{black} By \cite[Theorem 1.5]{kurtz2014weak}, we can deduce that the unique solution $\mub$ is adapted to the common noise filtration $(\{\sigma_\circ\,W^\circ_s:\quad s \in [0,t]\})_{t \in [0,T]}$. This is enough to deduce the second part of the proposition}.

\subsection{The {\it{n}}--player formulation}

This section proves the asymptotic results linking the finite--player game to the limiting MFG. We first establish compactness and identify the limit of empirical flows associated with arbitrary strategy profiles. A key step is to show that the moderate interaction term converges to the graphon--weighted density appearing in the limiting equation.

We then analyze one--player deviations. This provides the stability estimate needed to pass the Nash equilibrium property to the limit. Combining these ingredients, we prove that limits of approximate Nash equilibria are relaxed MFG solutions. Finally, we prove the converse direction by constructing Markovian approximate Nash equilibria from a given relaxed MFG equilibrium, thereby completing the finite--player/mean-field correspondence.

\subsubsection{Technical results}

{\color{black}
We recall that the framework consider here has been introduced in \Cref{sec_n_player}.}
\begin{proposition} \label{prop:weak_poc}
Let us assume that $(\varepsilon_n)_{n\ge1}$ is satisfying
\[
\varepsilon_n\to0,
\qquad
\frac{1}{n\varepsilon^d_n}\to0.
\]
Then the sequence $(\Lc(\mub^n))_{n\ge1}$ is relatively compact for the weak topology.
Moreover, along any convergent subsequence $(\Lc(\mub^{n_k}))_{k\ge1}$, there exist
$\Lambda\in\M$ and an associated pair $(X,U,\mub)$ verifying {\rm\Cref{eq:mc_kean_representative}} such that, for each $t\in[0,T]$,
\[
\Lc(\mub^{n_k}_t)\xrightarrow[k\to\infty]{}\Lc(\mub_t)=\Lc( \Lc(X_t,U \mid \Gc_t)),
\]
and for any bounded continuous map $F$,
\begin{align*}
&\lim_{k\to\infty}\frac{1}{n_k}\sum_{i=1}^{n_k}\E\Bigg[\int_0^T
F\Big(t,X^{i,n_k}_t,\widehat{V}^n_{i,t}(X^{i,n_k}_t),\Rr^{n_k}_{i,t}, \alpha^{i,n}(t,\Xbb^n)\Big)\,\mathrm{d}t\Bigg]
\\
&=
\E\Bigg[\int_{[0,T] \x A} F\big(t,X_t,\overline{p}_{\mub_t}(X_t,U),\Rr_{\mub_t}(U), a\big)\,\Lambda(t,X_t,U,\mub_t)(\mathrm{d}a)\,\mathrm{d}t\Bigg].
\end{align*}
\end{proposition}

\begin{proof}
\textbf{Step 0 (relative compactness).}
By the boundedness of $(b,\sigma)$, the tightness of $(\Lc(\mub^n))_{n\ge1}$ in $\Wc_0$ follows from similar arguments as in \cite[Theorem A.2.]{djete2019general}.
We now prove the second part. Fix a sub--sequence (still denoted $n$) such that $\Lc(\mub^n)\Rightarrow \Lc(\mub)$,
and we identify the limit.

\medskip
\textbf{Step 1 (a priori control of the local average).}
Fix an integer $k\ge1$ with $k<n$ (the number $n$ will go to infinity) and a smooth $h:\R^d \to \R$ with compact support. Expanding the $k$--th power yields
\begin{align}\label{eq:develop_better}
    \E\Big[\big| \Vh^n_{i,t}(X^{i,n}_t)\big|^kh(X^{i,n}_t)\Big]
    &=
    \frac{1}{n^k}\frac{1}{\varepsilon_n^{dk}}\,
    \E\Bigg[\Big|\sum_{j=1}^n \xi^n_{i,j}V\Big(\frac{X^{i,n}_t-X^{j,n}_t}{\varepsilon_n}\Big)\Big|^kh(X^{i,n}_t)\Bigg]\nonumber\\
    &=
    \frac{1}{n^k}\frac{1}{\varepsilon_n^{dk}}\,
    \sum_{j_1=1,\dots,j_k=1}^n
    \E\Bigg[\prod_{\ell=1}^k
    \xi^n_{i,j_{\ell}}V\Big(\frac{X^{i,n}_t-X^{j_\ell,n}_t}{\varepsilon_n}\Big)h(X^{i,n}_t)\Bigg].
\end{align}

For indices $j_1,\dots,j_k$ all distinct and different from $i$, the vector
$(X^{i,n}_t,X^{j_1,n}_t,\dots,X^{j_k,n}_t)$ solves a non--degenerate SDE and therefore admits a density
$p^{i,j_1,\dots,j_k}(t,\cdot)$ w.r.t. Lebesgue measure (see for instance \cite{krylov1980controlled}, the presence or not of a common noise is not relevant as long as $\sigma$ is non--degenerate). A change of variables gives
\begin{align}\label{eq:change_var}
    &\frac{1}{\varepsilon_n^{dk}}\E\Bigg[\prod_{\ell=1}^k
    \xi^n_{i,j_{\ell}}V\Big(\frac{X^{i,n}_t-X^{j_\ell,n}_t}{\varepsilon_n}\Big) h(X^{i,n}_t)\Bigg]\nonumber\\
    &\qquad=
    \int_{\R^{d(k+1)}} \Big(\prod_{\ell=1}^k \xi^n_{i,j_{\ell}}V(x_\ell)\Big)h(y)\,
    p^{i,j_1,\dots,j_k}\big(t,y,y+\varepsilon_n x_1,\dots,y+\varepsilon_n x_k\big)\,
    \mathrm{d}y\,\mathrm{d}x_1\cdots\mathrm{d}x_k.
\end{align}

\medskip
\textbf{Step 2 (negligibility of repeated indices).}
For each $i$, set
\[
\Jc_i:=\Big\{(j_1,\dots,j_k)\in\{1,\dots,n\}^k:\ j_1,\dots,j_k \text{ are pairwise distinct and } j_\ell\neq i\Big\}.
\]
We claim that
\begin{align}\label{eq:cong_to_zero_better}
\lim_{n\to\infty}\frac1n\sum_{i=1}^n
\frac{1}{n^k}\frac{1}{\varepsilon_n^{dk}} 
\sum_{(j_1,\dots,j_k)\notin\Jc_i}
\int_0^T
\E\Bigg[\prod_{\ell=1}^k
\xi^n_{i,j_{\ell}}V\Big(\frac{X^{i,n}_t-X^{j_\ell,n}_t}{\varepsilon_n}\Big)h(X^{i,n}_t)\Bigg]\mathrm{d}t
=0.
\end{align}

For clarity, we show it for $k=2$.
The contribution of diagonal terms $(j_1=j_2\neq i)$ satisfies
\begin{align*}
\frac{1}{n^2}\frac{1}{\varepsilon_n^{2d}}\sum_{j_1\neq i}^n
\E\Big[|\xi^n_{i,j_{1}}|^2V\Big(\frac{X^{i,n}_t-X^{j_1,n}_t}{\varepsilon_n}\Big)^2 h(X^{i,n}_t)\Big]
&=
\frac{1}{n}\frac{1}{\varepsilon^d_n}\int_{\R^{2d}} V(x)^2\,h(y)
\Big(\frac{1}{n}\sum_{j_1\neq i} |\xi^n_{i,j_{1}}|^2 p^{i,j_1}(t,y,y+\varepsilon_n x)\Big)\,\mathrm{d}y\,\mathrm{d}x.
\end{align*}
Since $j_1\neq i$, $p^{i,j_1}$ is the density of a $2d$--dimensional uniformly non--degenerate SDE with Lipschitz diffusive
coefficients. Hence for any compact $K\subset(0,T]\times\R^{2d}$, there exist $C>0$ and $\gamma\in(0,1)$ (depending only on the
dimension and boundedness of the coefficients, see for instance \Cref{lemma:estimations}) such that
\begin{align}\label{eq:density_analysis_better}
\sup_{(s,y_1,y_2)\in K}|p^{i,j_1}(s,y_1,y_2)|
+
\sup_{\substack{(t,x_1,x_2),(t',x_1',x_2')\in K}}
\frac{|p^{i,j_1}(t,x_1,x_2)-p^{i,j_1}(t',x_1',x_2')|}
{|x_1-x_1'|^\gamma+|x_2-x_2'|^\gamma+|t-t'|^{\gamma/2}}
<\infty.
\end{align}
Consequently, the sequence
\[
(t,y_1,y_2)\longmapsto \frac{1}{n^2}\sum_{i=1}^n\sum_{j_1\neq i} |\xi^n_{i,j_{1}}|^2 p^{i,j_1}(t,y_1,y_2)
\]
is relatively compact for the locally uniform topology; let $p$ denote the limit of a convergent subsequence (not relabeled). Since $h$ has compact support, we have
$\int_{[0,T]\times\R} h(y)p(t,y,y)\,\mathrm{d}y\,\mathrm{d}t<\infty$. Moreover, by assumption on $V$ i.e. a bounded probability density, $\int_{\R}V(x)^2\,\mathrm{d}x<\infty$.
Passing to the limit yields
\[
\lim_{n\to\infty}\int_{[0,T]\times\R^{2d}} V(x)^2\,h(y)\,\frac{1}{n^2}\sum_{i=1}^n\sum_{j_1\neq i} |\xi^n_{i,j_{1}}|^2
p^{i,j_1}(t,y,y+\varepsilon_n x)\,\mathrm{d}y\,\mathrm{d}x\,\mathrm{d}t
=
\int_{[0,T]\times\R^{2d}} V(x)^2\,h(y)\,p(t,y,y)\,\mathrm{d}y\,\mathrm{d}x\,\mathrm{d}t.
\]
Therefore,
\begin{align*}
&\lim_{n\to\infty}\frac1n\sum_{i=1}^n\int_0^T
\frac{1}{n^2}\frac{1}{\varepsilon_n^{2d}}\sum_{j_1\neq i}
\E\Big[|\xi^n_{i,j_1}|^2V\Big(\frac{X^{i,n}_t-X^{j_1,n}_t}{\varepsilon_n}\Big)^2h(X^{i,n}_t)\Big]\mathrm{d}t
\\
&=
\lim_{n\to\infty}\frac{1}{n}\frac{1}{\varepsilon_n^d}
\int_{[0,T]\times\R^{2d}} V(x)^2\,h(y)\,p(t,y,y)\,\mathrm{d}y\,\mathrm{d}x\,\mathrm{d}t
=0,
\end{align*}
since $\frac{1}{n\varepsilon_n^d}\to0$. This proves \eqref{eq:cong_to_zero_better} for $k=2$; the general case is analogous.

\medskip
\textbf{Step 3 (limit for the non--diagonal contribution).}
For $(j_1,\dots,j_k)\in\Jc_i$, the same regularity estimate as in \eqref{eq:density_analysis_better} implies that the family
\[
(t,y,x_1,\dots,x_k)\longmapsto
\frac{1}{n}\sum_{i=1}^n\frac{1}{n^k}\sum_{(j_1,\dots,j_k)\in\Jc_i}
\xi^n_{i,j_1} \x \cdots \x \xi^n_{i,j_k} p^{i,j_1,\dots,j_k}(t,y,x_1,\dots,x_k)
\]
is relatively compact in the locally uniform topology. Let $p^k$ denote the limit of a convergent subsequence (not relabeled).
Combining \eqref{eq:develop_better}, \eqref{eq:cong_to_zero_better} and \eqref{eq:change_var} yields
\begin{align}\label{eq:limit_density_better}
&\lim_{n\to\infty}\frac1n\sum_{i=1}^n\int_0^T
\E\Big[\big|\Vh^n_{i,t}(X^{i,n}_t)\big|^kh(X^{i,n}_t)\Big]\mathrm{d}t \nonumber
\\
&=
\lim_{n\to\infty}\int_{[0,T]\times\R^{d(k+1)}}
\Big(\prod_{\ell=1}^k V(x_\ell)\Big)\,h(y)\,
p^k\big(t,y,y+\varepsilon_n x_1,\dots,y+\varepsilon_n x_k\big)\,
\mathrm{d}y\,\mathrm{d}x_1\cdots\mathrm{d}x_k\,\mathrm{d}t \nonumber\\
&=
\int_{[0,T]\times\R^{d(k+1)}}
\Big(\prod_{\ell=1}^k V(x_\ell)\Big)\,h(y)\,p^k(t,y,y,\dots,y)\,
\mathrm{d}y\,\mathrm{d}x_1\cdots\mathrm{d}x_k\,\mathrm{d}t \nonumber\\
&=
\int_{[0,T]\times\R^d} h(y)\,p^k(t,y,y,\dots,y)\,\mathrm{d}y\,\mathrm{d}t.
\end{align}

\medskip
\textbf{Step 4 (identification of $p^k$).}
Let $g:[0,T]\times\R^d\to\R$ and $f^1,\dots,f^k:[0,T]\times\R^d\to\R$ be smooth maps with compact support. By the very definition of the densities combined with the convergence of $\mu^n$, of $\Gr^n$ to $\Gr$ and \cite[Proposition A.6]{djete2025nonexchangeablemeanfieldcontrol},
\begin{align*}
&\int_{[0,T]\times\R^{d(k+1)}} g(t,y)\prod_{\ell=1}^k f^\ell(t,x_\ell)\,
p^k(t,y,x_1,\dots,x_k)\,\mathrm{d}y\,\mathrm{d}x_1\cdots\mathrm{d}x_k\,\mathrm{d}t \\
&\qquad=
\lim_{n\to\infty}\int_0^T
\E\Bigg[\frac1n\sum_{i=1}^n g(t,X^{i,n}_t)\prod_{\ell=1}^k\Big(\frac1n\sum_{j=1}^n \xi^n_{i,j} f^\ell(t,X^{j,n}_t)\Big)\Bigg]\mathrm{d}t \\
&\qquad=
\int_0^T \E\left[\int_{\R^d \x [0,1]} g(t,y) \prod_{\ell=1}^k \int_{\R^d \x [0,1]} \Gr(v,u)f^\ell(t,x)\mub_t(\mathrm{d}x,\mathrm{d}u) \mub_t(\mathrm{d}y,\mathrm{d}v)\right]\mathrm{d}t \\
&\qquad=
\int_0^T \E\Bigg[ \int_{\R^d \x [0,1]}\int_{(\R^d \x [0,1])^{k}} g(t,y)\prod_{\ell=1}^k \Gr(v,u_\ell) f^\ell(t,x_\ell)\,
\mub_t(\mathrm{d}y,\mathrm{d}v)\mub_t(\mathrm{d}x_1,\mathrm{d}u_1)\cdots\mub_t(\mathrm{d}x_k,\mathrm{d}u_k)\Bigg]\mathrm{d}t.
\end{align*}
By using similar techniques as in \cite[Proposition 9.1.]{10.1214/23-AAP1993}, the probability $\mub_t$ admits a density i.e. there exists a Borel map
$p:\Pc(\R^d \x [0,1])\times\R^d \x [0,1]\to\R_+$ such that $\mub_t(\mathrm{d}x,\mathrm{d}u)=p_{\mub_t}(x,u)\,\mathrm{d}x\, \mathrm{d}u$ and $\int_{\R} p_{\mub_t}(x,u)\,\mathrm{d}x=1$ for a.e. $u \in [0,1]$.
Therefore,
\begin{align*}
&\int_{[0,T]\times\R^{d(k+1)}} g(t,y)\prod_{\ell=1}^k f^\ell(t,x_\ell)\,
p^k(t,y,x_1,\dots,x_k)\,\mathrm{d}y\,\mathrm{d}x_1\cdots\mathrm{d}x_k\,\mathrm{d}t \\
&\qquad=
\int_0^T \E\Bigg[\int_{(\R^d \x [0,1])^{k+1}} g(t,y)\prod_{\ell=1}^k f^\ell(t,x_\ell)\,
\Gr(v,u_\ell)\,p_{\mub_t}(y,v)p_{\mub_t}(x_1,u_1)\cdots p_{\mub_t}(x_k,u_k)\,\mathrm{d}y\,\mathrm{d}x_1\mathrm{d}u_1\cdots\mathrm{d}x_k \mathrm{d}u_k\Bigg]\mathrm{d}t.
\end{align*}
Since this holds for all smooth maps $g,f^1,\dots,f^k$, we deduce that for a.e. $(t,y,x_1,\dots,x_k)$,
\[
p^k(t,y,x_1,\dots,x_k)=\E\left[\int_{[0,1]^{k+1}}\Gr(v,u_1) \x \cdots \x \Gr(v,u_k)p_{\mub_t}(y,v)p_{\mub_t}(x_1,u_1)\cdots p_{\mub_t}(x_k,u_k)\, \mathrm{d}v\, \mathrm{d}u_1 \cdots \mathrm{d}u_k\right].
\]
Plugging this into \eqref{eq:limit_density_better} yields
\[
\lim_{n\to\infty}\frac1n\sum_{i=1}^n\int_0^T
\E\Big[\big|\Vh^n_{i,t}(X^{i,n}_t)\big|^kh(X^{i,n}_t)\Big]\mathrm{d}t
=
\E\Bigg[\int_0^T\int_{\R^d \x [0,1]} \left|\int_0^1\Gr(v,u)p_{\mub_t}(y,u)\mathrm{d}u\right|^k\,h(y)\,\mub_t(\mathrm{d}y,\mathrm{d}v)\,\mathrm{d}t\Bigg].
\]
A diagonal extraction provides a single subsequence along which the above convergence holds for all integers $k\ge1$.

\medskip
\textbf{Step 5 (joint convergence with empirical test functions).}
Repeating the same argument with additional empirical factors (tightness of the corresponding multi--point densities,
identification of the limit, and diagonal extraction), we obtain that along a subsequence (not relabeled),
for all $k,q\ge1$, smooth maps $g:[0,T]\times\R^d \x [0,1]\to\R$ and $f^1,\dots,f^q$, with $\muh^{i,n}_t:=\frac{1}{n} \sum_{j=1}^n \delta_{(X^{j,n}_t,\, \xi^n_{ij},\,u^n_j)}$,
\begin{align}\label{eq:cong_polynom_better}
&\lim_{n\to\infty}\frac1n\sum_{i=1}^n\int_0^T
\E\Big[\big|\Vh^n_{i,t}(X^{i,n}_t)\big|^k\, g(t,X^{i,n}_t,u^n_i)\,\prod_{\ell=1}^q\langle f^\ell,\muh^{i,n}_t\rangle\Big]\mathrm{d}t \nonumber
\\
&=
\E\Bigg[\int_0^T\int_{\R^d \x [0,1]} \left|\int_0^1 \Gr(v,u)p_{\mub_t}(y,u)\mathrm{d}u\right|^k\,g(t,y,v)\,\prod_{\ell=1}^q \int_{\R^d \x [0,1]}f^\ell(x',\Gr(v,u'),u')\mub_t(\mathrm{d}x',\mathrm{d}u')\,\mub_t(\mathrm{d}y,\mathrm{d}v)\,\mathrm{d}t\Bigg].
\end{align}

\medskip
\textbf{Step 6 (identification of the limit probability measure associated to $\mu$).}
Define the random probability measure
\[
\overline{\Gamma}^{n}(\mathrm{d}y,\mathrm{d}v,\mathrm{d}e,\mathrm{d}m,\mathrm{d}a)\mathrm{d}t
:=
\frac{1}{n}\sum_{i=1}^n\delta_{(X^{i,n}_t,u^n_i,p^{i,n}_t,\hat \mu^{i,n}_t,\alpha^{i,n}(t,\Xbb^n))}
(\mathrm{d}y,\mathrm{d}v,\mathrm{d}e,\mathrm{d}m,\mathrm{d}a)\mathrm{d}t,
\qquad
p^{i,n}_t:=\Vh^n_{i,t}(X^{i,n}_t),
\]
and set
\[
\Pr^n:= \Lc(\mu^n,\overline{\Gamma}^n).
\]
The convergence \eqref{eq:cong_polynom_better} implies that 
\begin{align} \label{eq:growth_density}
   \sup_{n \ge 1} \frac1n\sum_{i=1}^n\int_0^T
\E\left[\big|p^{i,n}_t\big|^k\right]\mathrm{d}t < \infty,\qquad \mbox{for any }k \ge 1.
\end{align}

We can deduce that
$(\Pr^n)_{n\ge1}$ is relatively compact for the weak topology. Let $\Pr=\P\circ(\mu,\overline{\Gamma})^{-1}$ be the limit of a convergent subsequence. We are taking the same convergent sub--sequence for $(\Lc(\mub^n))_{n\ge1}$ and $(\Pr^n)_{n\ge1}$. In addition, because of \eqref{eq:growth_density}, we have the convergence in the following sense : for any continuous map $(t,x,u,e,m,a) \mapsto H(t,x,u,e,m,a)$ with polynomial growth in $(u,e)$, and bounded in $(t,x,m,a)$, we have
\begin{align*}
    \Lim_{n \to \infty}\frac{1}{n}\sum_{i=1}^n\E \left[\int_0^T H(t,X^{i,n}_t,u^n_i,p^{i,n}_t,\hat \mu^{i,n}_t,\alpha^{i,n}(t,\Xbb^n)) \,\mathrm{d}t \right] = \E \left[\int_0^T \langle H, \overline{\Gamma}_t \rangle \,\mathrm{d}t \right].
\end{align*}

\medskip
By \eqref{eq:cong_polynom_better}, for any integers \(k,q\ge1\), any smooth map
\(g\), and any smooth maps \(f^1,\dots,f^q\) as above, we have
\begin{align} \label{eq:eq_identification}
&\E\Bigg[\int_0^T \int_{\R^d \times [0,1]\times\R_+ \times \Pc( \Er \times \R^d \times [0,1])}
 |e|^k\,g(t,y,v)\,\prod_{\ell=1}^q\langle f^\ell,m\rangle\,
\overline{\Gamma}_t(\mathrm{d}y,\mathrm{d}v,\mathrm{d}e,\mathrm{d}m,A)\,\mathrm{d}t\Bigg] 
\\
&=
\E\Bigg[\int_0^T \int_{\R^d \times [0,1]}
\left|\int_0^1\Gr(v,u)p_{\mub_t}(y,u)\,\mathrm{d}u\right|^k
g(t,y,v)
\nonumber\\
&\hspace{4cm}\times
\prod_{\ell=1}^q
\int_{\R^d \times [0,1]}
f^\ell(x',\Gr(v,u'),u')\,\mub_t(\mathrm{d}x',\mathrm{d}u')\,
\mub_t(\mathrm{d}y,\mathrm{d}v)\,\mathrm{d}t
\Bigg]. \nonumber
\end{align}
The purpose of the next step is to identify the two additional coordinates carried by
\(\overline{\Gamma}\): the variable \(e\), which represents the limiting local density,
and the variable \(m\), which represents the graphon environment seen from the label \(v\).

We introduce the deterministic measure
\begin{align*}
    \Kr_t(\mathrm{d}y,\mathrm{d}v,\mathrm{d}e,\mathrm{d}m)\,\mathrm{d}t
    :=
    \E\big[
        \overline{\Gamma}_t(\mathrm{d}y,\mathrm{d}v,\mathrm{d}e,\mathrm{d}m,A)
    \big]\,\mathrm{d}t.
\end{align*}
Let
\[
\widehat{\Kr}(t,y,v,m)(\mathrm{d}e)
\]
be a regular conditional distribution of the \(e\)-coordinate given
\((t,y,v,m)\), so that
\begin{align}  \label{eq:charac_2}
    \Kr_t(\mathrm{d}y,\mathrm{d}v,\mathrm{d}e,\mathrm{d}m)
    =
    \widehat{\Kr}(t,y,v,m)(\mathrm{d}e)\,
    \Kr_t(\mathrm{d}y,\mathrm{d}v,\R,\mathrm{d}m).
\end{align}

We first identify the \((y,v,m)\)-marginal of \(\Kr\). Since
\eqref{eq:eq_identification} holds for arbitrary \(g\), arbitrary \(q\), and arbitrary
test functions \(f^1,\dots,f^q\), the class of products $g(t,y,v)\prod_{\ell=1}^q\langle f^\ell,m\rangle$
separates the relevant variables. Therefore,
\begin{align} \label{eq:charac_1}
    \Kr_t(\mathrm{d}y,\mathrm{d}v,\R_+,\mathrm{d}m)
    =
    \E\left[
        \delta_{\Rc(t,\mub_t,v)}(\mathrm{d}m)\,
        \mub_t(\mathrm{d}y,\mathrm{d}v)
    \right],
\end{align}
where
\[
\Rc(t,\mub_t,v)(\mathrm{d}x,\mathrm{d}r,\mathrm{d}u)
:=
\delta_{\Gr(v,u)}(\mathrm{d}r)\,\mub_t(\mathrm{d}x,\mathrm{d}u).
\]
In other words, under the limiting measure, the environmental coordinate \(m\)
is forced to be the graphon push--forward of \(\mub_t\) seen from the label \(v\).
We now identify the \(e\)-coordinate. Combining
\eqref{eq:eq_identification}, \eqref{eq:charac_1}, and the disintegration
\eqref{eq:charac_2}, we obtain that, for every \(k\ge1\),
\[
\int_{\R_+}|e|^k\,
\widehat{\Kr}\bigl(t,y,v,V(t,\mub_t,v)\bigr)(\mathrm{d}e)
=
\big|\overline p_{\mub_t}(y,v)\big|^k,\quad \overline p_{\mub_t}(y,v)
:=
\int_0^1\Gr(v,u)p_{\mub_t}(y,u)\,\mathrm{d}u
\]
for \(\mathrm{d}t\otimes\mathrm{d}\P\otimes\mub_t(\mathrm{d}y,\mathrm{d}v)\)-a.e.
\((t,\omega,y,v)\).
Thus the probability measure
\[
\widehat{\Kr}\bigl(t,y,v,\Rc(t,\mub_t,v)\bigr)(\mathrm{d}e)
\]
and the Dirac measure
\[
\delta_{\overline p_{\mub_t}(y,v)}(\mathrm{d}e)
\]
have the same moments of every order. Since the latter measure is compactly
supported i.e. it is a singleton, \Cref{lemma:proba_equality} implies that
\begin{align*}
    \widehat{\Kr}\bigl(t,y,v,\Rc(t,\mub_t,v)\bigr)(\mathrm{d}e)
    =
    \delta_{\overline p_{\mub_t}(y,v)}(\mathrm{d}e),
\end{align*}
for \(\mathrm{d}t\otimes\mathrm{d}\P\otimes\mub_t(\mathrm{d}y,\mathrm{d}v)\)-a.e.
\((t,\omega,y,v)\).
Consequently, the limiting measure \(\overline{\Gamma}\) is identified on the
\((y,v,e,m)\)-coordinates. Namely, for \(\mathrm{d}t\otimes\mathrm{d}\P\)-a.e.,
\begin{align}\label{eq:density_equality_better}
\E\big[
    \overline{\Gamma}_t(\mathrm{d}y,\mathrm{d}v,\mathrm{d}e,\mathrm{d}m,A)
\big]
=
\E\left[
    \delta_{\overline p_{\mub_t}(y,v)}(\mathrm{d}e)\,
    \delta_{\Rc(t,\mub_t,v)}(\mathrm{d}m)\,
    \mub_t(\mathrm{d}y,\mathrm{d}v)
\right].
\end{align}
This is precisely the desired identification: the limiting local--density coordinate
is \(\overline p_{\mub_t}(y,v)\), and the limiting environment coordinate is the
graphon push-forward \(\Rc(t,\mub_t,v)\).

\medskip
\textbf{Step 7 (martingale problem for $\mu$).}
Let $f:\R^d \x [0,1]\to\R$ be smooth. By It\^o's formula,
\begin{align} \label{eq:ito_empirical}
\langle f,\mub^n_t\rangle
&=
\langle f,\mub^n_0\rangle
+
\int_0^t \frac1n\sum_{i=1}^n
\nabla_x f(X^{i,n}_s,u^n_i)\cdot
b\big(s,X^{i,n}_s,\Vh^n_{i,s}(X^{i,n}_s),\,\Rr^n_{i,s},\alpha^{i,n}(s,\Xbb^n)\big)\,\mathrm{d}s \nonumber
\\
&\quad+
\frac12\int_0^t \frac1n\sum_{i=1}^n
{\rm Tr}\!\left[
\nabla^2_{xx}f(X^{i,n}_s,u^n_i)
\Big(
\sigma(s,X^{i,n}_s)\sigma^\top(s,X^{i,n}_s)
+
\sigma_\circ\sigma_\circ^\top
\Big)
\right]\mathrm{d}s \nonumber
\\
&\quad+
\int_0^t \frac1n\sum_{i=1}^n
\nabla_x f(X^{i,n}_s,u^n_i)\cdot \sigma(s,X^{i,n}_s)\,\mathrm{d}W^i_s
+
\int_0^t \frac1n\sum_{i=1}^n
\nabla_x f(X^{i,n}_s,u^n_i)\cdot \sigma_\circ\,\mathrm{d}W^\circ_s.
\end{align}
Since \(\nabla_x f\) is bounded, the idiosyncratic martingale term still vanishes. Indeed,
\[
\E\left[
\left|
\int_0^T \frac1n\sum_{i=1}^n 
\nabla_x f(X^{i,n}_s,u_i^n)\cdot \sigma(s,X^{i,n}_s)\,\mathrm{d}W^i_s
\right|^2
\right]
=
\E\left[
\int_0^T \frac1{n^2}\sum_{i=1}^n
\left|
\sigma^\top(s,X^{i,n}_s)\nabla_x f(X^{i,n}_s,u_i^n)
\right|^2\,\mathrm{d}s
\right]
\le
\frac{C\,T\|\nabla_x f\|_\infty^2}{n}
\xrightarrow[n\to\infty]{}0.
\]
On the other hand, the common--noise martingale does not vanish. It is expected to converge a limiting object equivalent to
\[
\int_0^t
\int_{\R^d\times[0,1]}
\nabla_x f(y,v)\cdot \sigma_\circ\,
\mub_s(\mathrm{d}y,\mathrm{d}v)\,\mathrm{d}W^\circ_s .
\]
However, here we cannot go to the limit that way since we only have a weak convergence.

\medskip
\textbf{Step 8 (identification of the conditional law and conclusion).}
Let $q\ge1$, let $f^1,\dots,f^q$ be smooth test functions on $\R^d\times[0,1]$, and let
$F:\R^q\to\R$ be smooth. Using the previous identity \eqref{eq:ito_empirical}, the standard It\^o formula applied to
\[
\Big(\langle f^1,\mub^n_t\rangle,\dots,\langle f^q,\mub^n_t\rangle\Big),
\]
the weak convergence of $(\Pr^n)_{n \ge 1}$
and \eqref{eq:density_equality_better}, we obtain
\begin{align*}
&\E\Big[F(\langle f^1,\mub_t\rangle,\dots,\langle f^q,\mub_t\rangle)\Big]
=
\E\Big[F(\langle f^1,\mub_0\rangle,\dots,\langle f^q,\mub_0\rangle)\Big]
\\
&\quad+
\sum_{\ell=1}^q
\E\Bigg[
\int_0^t
\partial_\ell F(\langle f^1,\mub_s\rangle,\dots,\langle f^q,\mub_s\rangle)
\\
&\hspace{2.2cm}\times
\int_{\R^d\times[0,1]\times A}
\nabla_x f^\ell(y,v)\cdot
b\big(s,y,\overline{p}_{\mub_s}(y,v),\Rr_{\mub_s}(v),a\big)\,
\Gamma(s,y,v,\mub_s)(\mathrm{d}a)\,
\mub_s(\mathrm{d}y,\mathrm{d}v)\,\mathrm{d}s
\Bigg]
\\
&\quad+
\frac12\sum_{\ell=1}^q
\E\Bigg[
\int_0^t
\partial_\ell F(\langle f^1,\mub_s\rangle,\dots,\langle f^q,\mub_s\rangle)
\\
&\hspace{2.2cm}\times
\int_{\R^d\times[0,1]}
{\rm Tr}\!\left[
\nabla^2_{xx}f^\ell(y,v)
\Big(
\sigma(s,y)\sigma^\top(s,y)+\sigma_\circ\sigma_\circ^\top
\Big)
\right]\,
\mub_s(\mathrm{d}y,\mathrm{d}v)\,\mathrm{d}s
\Bigg]
\\
&\quad+
\frac12\sum_{\ell,m=1}^q
\E\Bigg[
\int_0^t
\partial_{\ell m}^2F(\langle f^1,\mub_s\rangle,\dots,\langle f^q,\mub_s\rangle)
\\
&\hspace{2.2cm}\times
\left(
\int_{\R^d\times[0,1]}
\nabla_x f^\ell(y,v)\cdot\sigma_\circ\,
\mub_s(\mathrm{d}y,\mathrm{d}v)
\right)
\left(
\int_{\R^d\times[0,1]}
\nabla_x f^m(y,v)\cdot\sigma_\circ\,
\mub_s(\mathrm{d}y,\mathrm{d}v)
\right)
\mathrm{d}s
\Bigg].
\end{align*}
Here
\[
\Gamma(s,y,v,\mub_s)(\mathrm{d}a)
:=
\widehat{\Gamma}
\big(s,y,v,\overline{p}_{\mub_s}(y,v),\overline{\Rr}_{\mub_s}(v)\big)(\mathrm{d}a),
\]
with
\[
\overline{\Rr}_{\mub_s}(v)
=
\Lc^{\mub_s}\big(\Gr(v,U),S_s,U\big),
\]
and $\widehat{\Gamma}$ is the disintegration kernel satisfying
\[
\widehat{\Gamma}(t,y,v,e,m)(\mathrm{d}a)\,
\overline{\Gamma}_t(\mathrm{d}y,\mathrm{d}v,\mathrm{d}e,\mathrm{d}m,A)
=
\overline{\Gamma}_t(\mathrm{d}y,\mathrm{d}v,\mathrm{d}e,\mathrm{d}m,\mathrm{d}a).
\]

The additional last term above comes from the quadratic variation generated by the common noise. Indeed,
\[
\mathrm{d}\langle \langle f^\ell,\mub\rangle,\langle f^m,\mub\rangle\rangle_s
=
\left(
\int_{\R^d\times[0,1]}\nabla_x f^\ell(y,v)\cdot\sigma_\circ\,\mub_s(\mathrm{d}y,\mathrm{d}v)
\right)
\left(
\int_{\R^d\times[0,1]}\nabla_x f^m(y,v)\cdot\sigma_\circ\,\mub_s(\mathrm{d}y,\mathrm{d}v)
\right)
\mathrm{d}s.
\]

Using the superposition principle with common noise (see \cite[Theorem 1.5., Theorem 1.3]{Lacker-Shkolnikov-Zhang_2020},), we obtain
\[
\Lc(\mub_t)=\Lc(\nub_t),\qquad t\in[0,T],
\]
where, $(Y,V,\nub)$ satisfies : $Y_0=\xi$, $(Y_0,U) \perp W \perp (\nub,\sigma_\circ\, W^\circ)$
\[
\nub_t=\Lc(Y_t,V\mid \nub,\sigma_\circ \,W^\circ),
\]
and
\begin{align*}
\mathrm{d}Y_t
=
\int_A
b\big(t,Y_t,\overline{p}_{\nub_t}(Y_t,V),\Rr_{\nub_t}(V),a\big)\,
\Gamma(t,Y_t,V,\nub_t)(\mathrm{d}a)\,\mathrm{d}t
+
\sigma(t,Y_t)\,\mathrm{d}W_t
+
\sigma_\circ\,\mathrm{d}W^\circ_t .
\end{align*}
This yields the desired identification of the limit and concludes the proof since the second convergence follows from the convergence involving $\overline{\Gamma}$.
\end{proof}

\medskip
Let us take here $\sigma_\circ=0.$ Given a Borel map $\alpha : [0,T] \x \R^d \x [0,1] \to A$, we denote by $(X,U)$ a pair such that : $\Lc(U)={\rm Unif}([0,1])$, $(X_0,U) \perp W$,
\begin{align} \label{eq:no_common_noise_prop}
    \mathrm{d}X_t
    &= b\big(t,X_t,\overline{p}_{\overline{m}_t}(X_t,U),\Rr_{\overline{m}_t}(U),\alpha(t,X_t,U)\big)\,\mathrm{d}t
    +\sigma(t,X_t)\,\mathrm{d}W_t,
    \qquad
    \overline{m}_t=\Lc(X_t,U),
\end{align}
{\color{black}
There exists a sequence of Lipschitz maps $(\alpha^\ell : [0,T] \x \R^d \x [0,1] \to A)_{\ell \ge 1}$ such that $\lim_{\ell \to \infty} \alpha^\ell=\alpha$ a.e.. We introduce
\begin{align*}
    \alpha^{\ell,i,n}(t,x_1,u^n_1,\cdots,x_n,u^n_n)
    :=
    \alpha^\ell(t,x_i,u^n_i).
\end{align*}
}
We set $\Xbb^{\ell,n}:=\Xbb^{n,\alphab^{\ell,n}}$, $\widehat V_{i,t}^{\ell,n}(x)
:=
\frac1n\sum_{j=1}^n \xi^n_{ij}V_n(x-X^{\ell,j,n}_t)$, $\Rr^{\ell,n}_{i,t}
=
\frac1n\sum_{j=1}^n \delta_{(\xi^n_{ij},X^{\ell,j,n}_t)}$ and $\mub^{\ell,n}_t:=\frac{1}{n} \sum_{i=1}^n \delta_{(X^{\ell,i,n}_t,u^n_i)}$, for each $t \in [0,T]$.

\begin{proposition} \label{corollary:weak_poc}
For each $t\in[0,T]$, for the weak topology,
\[
\lim_{\ell \to \infty}\lim_{n \to \infty}\Lc(\mub^{\ell,n}_t)={}\delta_{\overline{m}_t},
\]
and for any bounded continuous map $F$,
\begin{align*}
&\lim_{\ell\to\infty}\lim_{n\to\infty}\frac{1}{n}\sum_{i=1}^{n}\E\Bigg[\int_0^T
F\Big(t,X^{\ell,i,n}_t,\Vh^{\ell,n}_{i,t}(X^{\ell,i,n}_t),\Rr^{\ell,n}_{i,t}, \alpha^{\ell,i,n}(t,\Xbb^{\ell,n})\Big)\,\mathrm{d}t\Bigg]
\\
&=
\E\Bigg[\int_{0}^T F\big(t,X_t,\overline{p}_{\overline{m}_t}(X_t,U),\Rr_{\overline{m}_t}(U), \alpha(t,X_t,U)\big)\,\mathrm{d}t\Bigg].
\end{align*}

\end{proposition}

\begin{proof}
    If we denote by $(X^\ell,U)$ the weak solution of \Cref{eq:no_common_noise_prop}, by \Cref{prop:cong_graphon} associated with $\alpha^\ell$, with $\overline{m}^\ell_t=\Lc(X^\ell_t,U)$, we have
    $$
        \lim_{\ell \to \infty}\delta_{\overline p_{\overline{m}^\ell_t}(x,u)}(\mathrm{d}e)\,
\delta_{\alpha(t,x,u)}(\mathrm{d}a)\,
p_{\overline{m}^\ell_t}(x,u)\,\mathrm{d}x\,\mathrm{d}u\,\mathrm{d}t = \delta_{\overline p_{\overline{m}_t}(x,u)}(\mathrm{d}e)\,
\delta_{\alpha(t,x,u)}(\mathrm{d}a)\,
p_{\overline{m}_t}(x,u)\,\mathrm{d}x\,\mathrm{d}u\,\mathrm{d}t.
$$
    Let us now fix $\ell \ge 1$. We essentially apply \Cref{prop:weak_poc} with $\sigma_\circ=0$. Let $q\ge1$, $f^1,\dots,f^q$ smooth and $F:\R^q\to\R$ smooth. Using Ito formula at the $n$--particle level and going to the limit with \Cref{prop:weak_poc}. 
\begin{align*}
&\E\Big[F(\langle f^1,\mub_t\rangle,\dots,\langle f^q,\mub_t\rangle)\Big] = \lim_{n \to \infty}\E\Big[F(\langle f^1,\mub^n_t\rangle,\dots,\langle f^q,\mub^n_t\rangle)\Big]
\\
&=
\E\Big[F(\langle f^1,\mub_0\rangle,\dots,\langle f^q,\mub_0\rangle)\Big] \\
&\quad+
\sum_{\ell=1}^q\E\Bigg[\int_0^t \int_{\R^d \x [0,1]}
\partial_\ell F(\langle f^1,\mub_s\rangle,\dots,\langle f^q,\mub_s\rangle)\,\nabla_x f^\ell(y,v)\cdot\,
b\big(s,y,\overline{p}_{\overline{\mu}_s}(y,v),\Rr_{\mub_s}(v),\alpha^\ell(s,y,v)\big)\,\mub_s(\mathrm{d}y,\mathrm{d}v)\,\mathrm{d}s\Bigg] \\
&\quad+
\frac12\sum_{\ell=1}^q\E\Bigg[\int_0^t \partial_\ell F(\langle f^1,\mub_s\rangle,\dots,\langle f^q,\mub_s\rangle)\,
\int_{\R^d \x [0,1]} {\rm Tr} \left[ \nabla_x^2 f^\ell(y,v) \sigma(s,y) \sigma^\top(s,y) \right]\,\mub_s(\mathrm{d}y,\mathrm{d}v)\,\mathrm{d}s\Bigg].
\end{align*}
This being true for any $q, (f^1,\dots,f^q)$ and $F$. By the uniqueness of $(\overline{m}^\ell_t)_{t \in [0,T]}$ in \Cref{prop:uniqueness_FP}, we deduce that $\lim_{n \to \infty} \Lc(\mub^{\ell,n}_t)=\Lc(\mub_t)=\delta_{\overline{m}^\ell_t}$ for each $t \in [0,T]$. By combining the convergence in $n$ and $\ell$, we obtain the desired result. The second result follows directly.
\end{proof}

Let us consider a sequence of deviating controls $(\beta^{i,n})_{i \le n,\,n\ge1}$ with $\beta^{i,n}\in\Ac_n$. For each player $i\in\{1,\dots,n\}$,
we denote by $\boldsymbol{\beta}^n=(\beta^{1,n},\cdots,\beta^{n,n}),$
\[
\Ybb^i:=\Xbb^{n,(\alphab^{-i},\beta^{i,n})}=\big(Y^{i,1},\dots,Y^{i,n}\big)
\]
the state vector of the $n$--particle system when player $i$ uses $\beta^{i,n}$ and all other players use the reference profile $\alphab$.
We also introduce the associated empirical measures
\[
\mut^{i,n}_t:=\frac1n\sum_{j=1}^n\delta_{(\xi^n_{i,j},\,Y^{i,j}_t,\,u^n_j)},
\qquad\mub^{i,n}_t:=\frac1n\sum_{j=1}^n\delta_{(Y^{i,j}_t,\,u^n_j)},
\qquad
\mub^{n,\betab^n}_t:=\frac1n\sum_{i=1}^n\delta_{(Y^{i,i}_t,\, u^n_i)}.
\]
Hence $\mu^{i,n}_t$ is the empirical measure of the whole population under the deviation of player $i$, while
$\mub^{n,\betab^n}_t$ is the empirical distribution of the deviating player's own state across the population of deviations.

\medskip
Fix now a continuous $\F$--adapted process $\nub=(\nub_t)_{t\in[0,T]}$, independent of $(X_0,U,W)$, and a Borel kernel
$\Lambda^\beta:[0,T]\times\Pc(\R^d \x [0,1])\times\R^d \x [0,1]\to\Pc(A)$. We define $(X^\beta,U)$ as the (weak) solution of : $(\nub, W^\circ) \perp (X_0,U) \perp W$,
\begin{align} \label{eq:deviating_limit}
    \mathrm{d}X^\beta_t
    =
    \int_A b\big(t,X^\beta_t,\overline{p}_{\nub_t}(X^\beta_t,U),\Rr_{\nub_t}(U),a\big)\,
    \Lambda^\beta(t,X^\beta_t,U,\nub_t)(\mathrm{d}a)\,\mathrm{d}t
    +\sigma(t,X_t)\,\mathrm{d}W_t + +\sigma_\circ\,\mathrm{d}W^\circ_t,
    \quad
    \mub^\beta_t=\Lc(X^\beta_t,U\mid \Gc_t),
\end{align}
where $\Gc_t=\sigma\{ \nub_s,\,\sigma_\circ W^\circ_s,\,s \in [0,t]\}$ and, we recall that $p_{\nub_t}(\cdot,\cdot)$ denotes a version of the density of $\nub_t$ and $\overline{p}_{\nub_t}(x,u)=\int_0^1 \Gr(u,v)p_{\nub_t}(x,v)\mathrm{d}v$. We recall that we denote by $\mub^n_t$ the empirical distribution associated with the reference profile $\alphab^n=(\alpha^{1,n},\cdots,\alpha^{n,n})$, i.e., $\mub^n_t=\frac{1}{n} \sum_{j=1}^n \delta_{(X^{i,n,\alphab^n}_t,\,u^n_i)}$.

\begin{proposition} \label{weak_deviating}
Then, the sequence
\[
\left( \Lc(\mub^{n,\betab^n})\right)_{n\ge1}
\]
is relatively compact in $\Wc_0$.
Moreover, along any convergent subsequence $\left(\Lc(\mub^{n_k,\, \betab^{n_k}})\right)_{k\ge1}$, assume that
\[
\Lc(\mub^{n_k}_t)\xrightarrow[k\to\infty]{}\Lc(\nub_t)\qquad\text{for each }t\in[0,T].
\]
Then there exist a kernel $\Lambda^\beta$ and an associated pair $(X^\beta,U,\mub^\beta)$ as above in {\rm\Cref{eq:deviating_limit}} such that
\[
\Lc(\mub^{n,\betab^n}_t)\xrightarrow[k\to\infty]{}\Lc(\mub^\beta_t)
\qquad\text{for each }t\in[0,T],
\]
and, for every bounded continuous function $F$,
\begin{align*}
&\lim_{k\to\infty}\frac1{n_k}\sum_{i=1}^{n_k}\E\Bigg[\int_0^T
F\Big(t,Y^{i,i}_t,\Vh^n_{i,t}(Y^{i,i}_t),\Rr^{n_k}_{i,t},\beta^{n_k}(t,\Ybb^i)\Big)\,\mathrm{d}t\Bigg]
\\
&=
\E\Bigg[\int_0^T\int_A
F\big(t,X^\beta_t,\overline{p}_{\nub_t}(X^\beta_t,U),\Rr_{\nub_t}(U),a\big)\,
\Lambda^\beta(t,X^\beta_t,U,\nub_t)(\mathrm{d}a)\,\mathrm{d}t\Bigg].
\end{align*}
\end{proposition}

\begin{proof}
\textbf{Step 1 (moment estimates for the local interaction term).}
Arguing exactly as in the proof of \Cref{prop:weak_poc} (tightness of averaged multi--point densities and diagonal extraction),
we can find a subsequence (not relabeled) such that for each integer $k\ge1$,
\begin{align}\label{eq:dev_moment_limit}
\lim_{n\to\infty}\frac1n\sum_{i=1}^n \int_0^T
\E\Big[\big|\Vh^n_{i,t}(Y^{i,i}_t)\big|^k \Big]\mathrm{d}t
=
\int_{[0,T]\times\R^d} p^{k}(t,y,y,\dots,y)\,\mathrm{d}y\,\mathrm{d}t,
\end{align}
where, for $(j_1,\dots,j_k)$ pairwise distinct and different from $i$,
\[
\Lc\big(Y^{i,i}_t,Y^{i,j_1}_t,\dots,Y^{i,j_k}_t\big)(\mathrm{d}x_i,\mathrm{d}x_{j_1}\cdots\mathrm{d}x_{j_k})
=
p^{i,j_1,\dots,j_k}(t,x_i,x_{j_1},\dots,x_{j_k})\,\mathrm{d}x_i\,\mathrm{d}x_{j_1}\cdots\mathrm{d}x_{j_k},
\]
and $p^k$ is the locally uniform limit of the averaged densities
\[
\frac{1}{n}\sum_{i=1}^n\frac{1}{n^k}\sum_{(j_1,\dots,j_k)\in\Jc_i}
\xi^n_{i,j_1} \x \cdots \x \xi^n_{i,j_k}\,p^{i,j_1,\dots,j_k}\xrightarrow[n\to\infty]{}p^{k}.
\]

\medskip
\textbf{Step 2 (Girsanov representation of deviating systems).}
We follow a similar argument to \cite[Proposition 5.6.]{Lacker-closedloop} (see also \cite[Lemma 3.9.]{closed-loop-MFG_MDF}). Define
\[
\mathrm{d}Z^{i,n}_t=Z^{i,n}_t\,\phi^{i,n}_t\,\mathrm{d}W^i_t,\qquad Z^{i,n}_0=1,
\]
with
\[
\phi^{i,n}_t
=
\sigma(t,X^{i,n}_t)^{-1}\left(b\Big(t,X^{i,n}_t,\Vh^n_{i,t}(X^{i,n}_t),\Rr^n_{i,t},\beta^{i,n}(t,\Xbb^n)\Big)
-
b\Big(t,X^{i,n}_t,\Vh^n_{i,t}(X^{i,n}_t),\Rr^n_{i,t},\alpha^{i,n}(t,\Xbb^n)\Big) \right).
\]
Let $\mathrm{d}\P^{i,n}:=Z^{i,n}_T\,\mathrm{d}\P$.The map $b$ and $\sigma$ being bounded and invertible respectively, by Girsanov's theorem,
\[
\Lc^{\P^{i,n}}(\Xbb^n)=\Lc(\Ybb^i),\qquad 1\le i\le n.
\]
Introduce the empirical law
\[
\Qr^n
:=
\frac1n\sum_{i=1}^n
\P\circ\big(X^{i,n},u^n_i,Z^{i,n},\Gamma^{i,n},\mub^n,\overline{\Gamma}^n,\,W^\circ\big)^{-1},
\]
where
\[
\Gamma^{i,n}:=\delta_{(p^{i,n}_t,\,\beta^{i,n}(t,\Xbb^n))}(\mathrm{d}e,\mathrm{d}a)\,\mathrm{d}t,
\qquad
p^{i,n}_t:=\Vh^n_{i,t}(X^{i,n}_t),
\]
and
\[
\overline{\Gamma}^{n}_t(\mathrm{d}y,\mathrm{d}e,\mathrm{d}a)\,\mathrm{d}t
:=
\frac{1}{n}\sum_{i=1}^n\delta_{(X^{i,n}_t,p^{i,n}_t,\alpha^{i,n}(t,\Xbb^n))}
(\mathrm{d}y,\mathrm{d}e,\mathrm{d}a)\mathrm{d}t.
\]
Using the fact that $A$ is compact, the coefficients $(b,\sigma)$ are bounded, similarly to the way we argued in \Cref{prop:weak_poc}, the sequence $(\Qr^n)_{n\ge1}$ is relatively compact for the weak topology; along a subsequence (not relabeled) we have
\[
\Qr^n \Longrightarrow \Qr:=\P\circ(X,U,Z,\Gamma,\mub,\overline{\Gamma},W^\circ)^{-1}.
\]
Since we can control the moments of random variables $(U,Z, \Gamma)$ (through their sequence), we can check that the convergence involving $(\Qr^n)_{n \ge 1}$ is also true in the sense that : for any continuous map $(x,u,z,\gamma,\overline{m},\overline{\gamma}) \mapsto G(x,u,z,\gamma,\overline{m},\overline{\gamma}) $ with polynomial growth in $(u,z,\gamma)$ and bounded in $(x,\overline{m},\overline{\gamma})$, we have
\begin{align*}
    \lim_{n \to \infty} \frac{1}{n} \sum_{i=1}^n\E \left[ G(X^{i,n},u^n_i,Z^{i,n},\Gamma^{i,n},\mub^n,\overline{\Gamma}^n) \right] = \E \left[ G(X,U,Z,\Gamma,\mub,\overline{\Gamma}) \right].
\end{align*}
\medskip
\textbf{Step 3 (identification of the limiting densities $p^k$).}
Let $g,f^1,\dots,f^k$ be smooth maps on $[0,T]\times\R^d$. By definition of $p^k$,
\begin{align*}
&\int_{[0,T]\times\R^{d(k+1)}} g(t,y)\prod_{\ell=1}^k f^\ell(t,x_\ell)\,
p^k(t,y,x_1,\dots,x_k)\,\mathrm{d}y\,\mathrm{d}x_1\cdots\mathrm{d}x_k\,\mathrm{d}t \\
&\qquad=
\lim_{n\to\infty}\int_0^T
\E\Bigg[\frac1n\sum_{i=1}^n g(t,Y^{i,i}_t)\prod_{\ell=1}^k\Big(\frac1n\sum_{j=1}^n \xi^n_{i,j}f^\ell(t,Y^{i,j}_t)\Big)\Bigg]\mathrm{d}t.
\end{align*}
Using $\Lc^{\P^{i,n}}(\Xbb^n)=\Lc(\Ybb^i)$ and $\mathrm{d}\P^{i,n}=Z^{i,n}_T\,\mathrm{d}\P$, we rewrite the right--hand side as
\begin{align*}
&\lim_{n\to\infty}\int_0^T
\E\Bigg[\frac1n\sum_{i=1}^n Z^{i,n}_t\, g(t,X^{i,n}_t)\prod_{\ell=1}^k\Big(\frac1n\sum_{j=1}^n \xi^n_{i,j}f^\ell(t,X^{j,n}_t)\Big)\Bigg]\mathrm{d}t
\\
&=
\int_0^T \E\Big[Z_t\, g(t,X_t)\prod_{\ell=1}^k \int_{\R^d \x [0,1]} \Gr(U,v) f^\ell(t,y) \mub_t(\mathrm{d}y,\mathrm{d}v) \Big]\mathrm{d}t.
\end{align*}

By the similar techniques of \cite[Lemma 3.9.]{closed-loop-MFG_MDF} (martingale property of the limit exponential), conditionally to $\widehat{\Gc}_T$, the process $Z$ is a Dol\'eans exponential with
$\E[Z_T\mid \widehat{\Gc}_T]=1$. Define the conditional probability $\P^{\mub}$ by
\[
\mathrm{d}\P^{\mub} := Z_T\,\mathrm{d}\P^{\widehat{\Gc}_T},
\qquad \widehat{\Gc}_t:=\sigma(\mub_s,\,\sigma_\circ W^\circ_s:s\le t).
\]
Moreover, by \cite[Proposition 9.1]{10.1214/23-AAP1993} (or in \cite[Lemma 3.9.]{closed-loop-MFG_MDF} where we can see that the equation satisfied by $X$ under $\P^{\mub}$), for each $t>0$ the law $\Lc^{\P^{\mub}}(X_t,U)$ admits a density
$h(t,\mub,\cdot,\cdot)$. Therefore,
\begin{align*}
&\int_0^T \E\Big[Z_t\, g(t,X_t)\prod_{\ell=1}^k \int_{\R^d \x [0,1]} \Gr(U,v) f^\ell(t,y)\, \mub_t(\mathrm{d}y,\mathrm{d}v) \Big]\mathrm{d}t
\\
&=
\int_0^T \E\Bigg[\int_{\R^d} g(t,x)\, \prod_{\ell=1}^k  \int_{\R^d \x [0,1]} \Gr(u,v) f^\ell(t,y) \mub_t(\mathrm{d}y,\mathrm{d}v) h(t,\mub,x,u)\,\mathrm{d}x\,\mathrm{d}u\Bigg]\mathrm{d}t.
\end{align*}
Using again that $\mub_t(\mathrm{d}x,\mathrm{d}u)=p_{\mub_t}(x,u)\,\mathrm{d}x\,\mathrm{d}u$ (as in  \Cref{prop:weak_poc}), we obtain
\begin{align*}
&\int_{[0,T]\times\R^{d(k+1)}} g(t,y)\prod_{\ell=1}^k f^\ell(t,x_\ell)\,
p^k(t,y,x_1,\dots,x_k)\,\mathrm{d}y\,\mathrm{d}x_1\cdots\mathrm{d}x_k\,\mathrm{d}t\\
&=
\int_0^T \E\Bigg[\int_{(\R^d \x [0,1])^{k+1}} g(t,y)\prod_{\ell=1}^k \Gr(v,u_\ell)f^\ell(t,x_\ell)\,
h(t,\mub,y,v)\,p_{\mub_t}(x_1,u_1)\cdots p_{\mub_t}(x_k,u_k)\,
\mathrm{d}y\,\mathrm{d}v\,\mathrm{d}x_1\,\mathrm{d}u_1\cdots\mathrm{d}x_k\,\mathrm{d}u_k\Bigg]\mathrm{d}t.
\end{align*}
Since this holds for all bounded continuous test functions, we deduce that for a.e. $(t,y,x_1,\dots,x_k)$,
\[
p^k(t,y,x_1,\dots,x_k)=\E\left[\int_{[0,1]^{k+1}}\prod_{\ell=1}^k \Gr(v,u_\ell) h(t,\mub,y,v)\,p_{\mub_t}(x_1,u_1)\cdots p_{\mub_t}(x_k,u_k)\mathrm{d}v\,\mathrm{d}u_1 \cdots \mathrm{d}u_k\right].
\]
Plugging into \eqref{eq:dev_moment_limit} gives
\begin{align}\label{eq:dev_moment_identified}
\lim_{n\to\infty}\frac1n\sum_{i=1}^n\int_0^T
\E\Big[\big|\Vh^{n}_{i,t}(Y^{i,i}_t)\big|^k\Big]\mathrm{d}t
&=
\E\Bigg[\int_0^T\int_{\R^d} \left|\int_0^1 \Gr(v,u)\,p_{\mub_t}(y,u)\,\mathrm{d}u\right|^k\,h(t,\mub,y,v)\,\mathrm{d}y\,\mathrm{d}v\,\mathrm{d}t\Bigg] \nonumber
\\
&=
\E\Bigg[\int_0^T Z_t\,\left|\int_0^1 \Gr(U,u)\,p_{\mub_t}(X_t,u)\,\mathrm{d}u\right|^k\,\mathrm{d}t\Bigg].
\end{align}
A diagonal extraction yields a single subsequence along which \eqref{eq:dev_moment_identified} holds for all $k\ge1$.
By the same extension as in  \Cref{prop:weak_poc}, we also obtain the joint convergence: for all $k,q\ge1$ and bounded continuous
$g$ and $f^1,\dots,f^q$,
\begin{align}\label{eq:conv_deviating_better}
&\lim_{n\to\infty}\frac1n\sum_{i=1}^n\int_0^T
\E\Big[\big|\Vh^n_{i,t}(Y^{i,i}_t)\big|^k\, g(t,Y^{i,i}_t)\,\prod_{\ell=1}^q\langle f^\ell,\mut^{i,n}_t\rangle\Big]\mathrm{d}t \nonumber
\\
&=
\E\Bigg[\int_0^T Z_t\,|\overline{p}_{\mub_t}(X_t,U)|^k\,g(t,X_t)\,\prod_{\ell=1}^q \int f^\ell (y,\Gr(U,v),v) \mub_t(\mathrm{d}y,\mathrm{d}v)\,\mathrm{d}t\Bigg].
\end{align}

\medskip
\textbf{Step 4 (identification of the limiting control kernel).}
Let $\varphi:\R\to\R$ be bounded continuous, $g$ bounded continuous, and $L$ a bounded continuous functional of $\mu_t$.
Again as in the proof of \Cref{prop:weak_poc}, \eqref{eq:conv_deviating_better} implies
\[
\E\Bigg[\int_{[0,T] \x \R} Z_t\,\varphi(e)\,g(t,X_t,U)\,L(\mub_t)\,\Gamma_t(\mathrm{d}e,A)\,\mathrm{d}t\Bigg]
=
\E\Bigg[\int_0^T Z_t\,\varphi\big(\overline{p}_{\mub_t}(X_t,U)\big)\,g(t,X_t,U)\,L(\mub_t)\,\mathrm{d}t\Bigg].
\]
In particular, $\Gamma_t(\mathrm{d}e,A)$ concentrates at $e=\overline{p}_{\mub_t}(X_t,U)$.
Define the measure on $\R^d \x [0,1]\times\R\times\Pc(\R^d\x [0,1])\times A$ by
\[
\Lambda(t)(\mathrm{d}y, \mathrm{d}v,\mathrm{d}e,\mathrm{d}m,\mathrm{d}a)
:=
\E\Big[Z_t\,\delta_{(X_t,U)}(\mathrm{d}y,\mathrm{d}v)\,\delta_{\mub_t}(\mathrm{d}m)\,\Gamma_t(\mathrm{d}e,\mathrm{d}a)\Big].
\]
Let $\widehat{\Gamma}(t,y,v,e,m)$ be a regular conditional probability such that
\begin{align} \label{eq:desintegration}
\widehat{\Gamma}(t,y,v,e,m)(\mathrm{d}a)\,\Lambda(t)(\mathrm{d}y,\mathrm{d}v,\mathrm{d}e,\mathrm{d}m,A)
=
\Lambda(t)(\mathrm{d}y,\mathrm{d}v,\mathrm{d}e,\mathrm{d}m,\mathrm{d}a).
\end{align}
We then set, for $t\in[0,T]$,
\[
\Lambda^\beta(t,m,y,v):=\widehat{\Gamma}\big(t,y,v,\overline{p}_{m}(y,v),m\big)\in\Pc(A).
\]

\medskip
{\color{black}

For \(1\le \ell\le q\), set
\[
M_t^{i,n,\ell}:=\langle f^\ell,\mub_t^{i,n}\rangle
=\frac1n\sum_{j=1}^n f^\ell(Y_t^{i,j},u_j^n),
\qquad
\mathbf{M}_t^{i,n}:=\big(M_t^{i,n,1},\dots,M_t^{i,n,q}\big).
\]
Assume that, for each \(j=1,\dots,n\),
\[
\mathrm{d}Y_t^{i,j}
=
B_t^{i,j}\,\mathrm{d}t+\Sigma_t^{i,j}\,\mathrm{d}W_t^j+\sigma_\circ\,\mathrm{d}W_t^\circ.
\]
Then, using the It\^o formula, we observe that for each \(\ell=1,\dots,q\),
\begin{align*}
\mathrm{d}M_t^{i,n,\ell}
&=
b_t^{i,n,\ell}\,\mathrm{d}t
+\frac1n\sum_{j=1}^n \nabla_x f^\ell(Y_t^{i,j},u_j^n)\cdot \Sigma_t^{i,j}\,\mathrm{d}W_t^j
+\bar{\sigma}_t^{i,n,\ell}\,\mathrm{d}W_t^\circ,
\end{align*}
where
\begin{align*}
b_t^{i,n,\ell}
&:=
\frac1n\sum_{j=1}^n
\nabla_x f^\ell(Y_t^{i,j},u_j^n)\cdot B_t^{i,j}
+\frac1{2n}\sum_{j=1}^n
{\rm Tr}\!\Big(
\nabla_{xx}^2 f^\ell(Y_t^{i,j},u_j^n)
\big(\Sigma_t^{i,j}(\Sigma_t^{i,j})^\top+\sigma_\circ\sigma_\circ^\top\big)
\Big),
\\
\bar{\sigma}_t^{i,n,\ell}
&:=
\frac1n\sum_{j=1}^n \nabla_x f^\ell(Y_t^{i,j},u_j^n)^\top \sigma_\circ.
\end{align*}
Moreover,
\begin{align*}
\mathrm{d}\langle Y^{i,i}\rangle_t
&=
\Big(\Sigma_t^{i,i}(\Sigma_t^{i,i})^\top+\sigma_\circ\sigma_\circ^\top\Big)\,\mathrm{d}t,
\\
\mathrm{d}\langle Y^{i,i},M^{i,n,\ell}\rangle_t
&=
\Bigg(
\frac1n\Sigma_t^{i,i}(\Sigma_t^{i,i})^\top \nabla_x f^\ell(Y_t^{i,i},u_i^n)
+\sigma_\circ(\bar{\sigma}_t^{i,n,\ell})^\top
\Bigg)\,\mathrm{d}t,
\\
\mathrm{d}\langle M^{i,n,\ell},M^{i,n,m}\rangle_t
&=
\Bigg(
\frac1{n^2}\sum_{j=1}^n
\nabla_x f^\ell(Y_t^{i,j},u_j^n)^\top
\Sigma_t^{i,j}(\Sigma_t^{i,j})^\top
\nabla_x f^m(Y_t^{i,j},u_j^n)
+\bar{\sigma}_t^{i,n,\ell}(\bar{\sigma}_t^{i,n,m})^\top
\Bigg)\,\mathrm{d}t.
\end{align*}
Therefore, for any \(F\in C^2(\R^d\times\R^q)\),
\begin{align*}
&F\big(Y_t^{i,i},\mathbf{M}_t^{i,n}\big)
=
F\big(Y_0^{i,i},\mathbf{M}_0^{i,n}\big)
+\int_0^t \nabla_x F\big(Y_s^{i,i},\mathbf{M}_s^{i,n}\big)\cdot \mathrm{d}Y_s^{i,i}
+\sum_{\ell=1}^q \int_0^t \partial_\ell F\big(Y_s^{i,i},\mathbf{M}_s^{i,n}\big)\,\mathrm{d}M_s^{i,n,\ell}
\\
&\quad
+\frac12\int_0^t
{\rm Tr}\!\Big[
\nabla_{xx}^2F\big(Y_s^{i,i},\mathbf{M}_s^{i,n}\big)\,\mathrm{d}\langle Y^{i,i}\rangle_s
\Big]
+\sum_{\ell=1}^q \int_0^t
{\rm Tr}\!\Big[
\nabla_{x\ell}^2F\big(Y_s^{i,i},\mathbf{M}_s^{i,n}\big)\,\mathrm{d}\langle Y^{i,i},M^{i,n,\ell}\rangle_s
\Big]
\\
&\quad
+\frac12\sum_{\ell,m=1}^q \int_0^t
\partial_{\ell m}^2F\big(Y_s^{i,i},\mathbf{M}_s^{i,n}\big)\,\mathrm{d}\langle M^{i,n,\ell},M^{i,n,m}\rangle_s.
\end{align*}

}
Set
\[
\mathbf M_t:=\big(\langle f^1,\mub_t\rangle,\dots,\langle f^q,\mub_t\rangle\big).
\]
For \(1\le \ell\le q\), define
\begin{align*}
A_s^\ell
&:=
\int_{\R^d\times[0,1]\times A}
\nabla_x f^\ell(y,v)\cdot b(s,y,e,a)\,
\overline{\Gamma}(s,y,v,\overline p_{\mub_s}(y,v))(\mathrm{d}a)\,
\mub_s(\mathrm{d}y,\mathrm{d}v)
\\
&\quad+
\frac12
\int_{\R^d\times[0,1]}
{\rm Tr}\!\left[
\nabla^2_{xx}f^\ell(y,v)
\big(\sigma(s,y)\sigma^\top(s,y)+\sigma_\circ\sigma_\circ^\top\big)
\right]\mub_s(\mathrm{d}y,\mathrm{d}v),
\\
C_s^\ell
&:=
\int_{\R^d\times[0,1]}
\sigma_\circ^\top\nabla_x f^\ell(y,v)\,
\mub_s(\mathrm{d}y,\mathrm{d}v).
\end{align*}
Then, using the convergence of
\(\P\circ(X,U,Z,\Gamma,\mub,\overline\Gamma)^{-1}\), the limiting dynamics, and the disintegration in
\eqref{eq:desintegration}, we obtain
\begin{align} \label{eq:limit_deviation}
&\E\left[Z_t F(X_t,U,\mathbf M_t)\right]
=
\E\left[F(X_0,U,\mathbf M_0)\right]
+
\E\Bigg[
\int_0^t
Z_s
\int_{\R_+\times A}
\nabla_x F(X_s,U,\mathbf M_s)\cdot b(s,X_s,e,a)\,
\Gamma_s(\mathrm{d}e,\mathrm{d}a)\,\mathrm{d}s
\Bigg]
\\
&\quad+
\frac12\E\Bigg[
\int_0^t
Z_s\,
{\rm Tr}\!\left[
\nabla^2_{xx}F(X_s,U,\mathbf M_s)
\big(\sigma(s,X_s)\sigma^\top(s,X_s)+\sigma_\circ\sigma_\circ^\top\big)
\right]\mathrm{d}s
\Bigg]
+
\sum_{\ell=1}^q
\E\Bigg[
\int_0^t
Z_s\,\partial_{\ell+1}F(X_s,U,\mathbf M_s)\,
A_s^\ell\,\mathrm{d}s
\Bigg] \nonumber
\\
&\quad+
\sum_{\ell=1}^q
\E\Bigg[
\int_0^t
Z_s\,
\nabla_x\partial_{\ell+1}F(X_s,U,\mathbf M_s)\cdot
\sigma_\circ C_s^\ell\,\mathrm{d}s
\Bigg]
+
\frac12\sum_{\ell,m=1}^q
\E\Bigg[
\int_0^t
Z_s\,\partial_{\ell+1,m+1}^2F(X_s,U,\mathbf M_s)\,
C_s^\ell\cdot C_s^m\,\mathrm{d}s
\Bigg]. \nonumber
\end{align}

\medskip
\textbf{Step 5 (limit dynamics and convergence of costs).}
We first record a useful marginal conditioning observation. From the identification obtained above for the limiting empirical flow, we know that
\[
\Lc(\mub_t)=\Lc(\nub_t),\qquad t\in[0,T].
\]
Moreover, since
\[
\E[Z_t\mid \mub_t]
=
\E\big[\E[Z_t\mid \widehat{\Gc}_T]\mid \mub_t\big]
=1,
\]
we may define a Borel transition kernel $\ell(t,m)(\mathrm{d}y,\mathrm{d}v)\in\Pc(\R^d\times[0,1])$
such that
\[
\ell(t,\mub_t)(\mathrm{d}y,\mathrm{d}v)
=
\E\big[Z_t\,\delta_{(X_t,U)}(\mathrm{d}y,\mathrm{d}v)\mid \mub_t\big].
\]
Consequently, for every bounded Borel map $\Hc:\R^d\times[0,1]\times\Pc(\R^d\times[0,1])\to\R,$
we have
\[
\E\big[Z_t\,\Hc(X_t,U,\mub_t)\big]
=
\E\left[
\int_{\R^d\times[0,1]}
\Hc(y,v,\mub_t)\,\ell(t,\mub_t)(\mathrm{d}y,\mathrm{d}v)
\right].
\]
Since \(\Lc(\mub_t)=\Lc(\nub_t)\), the last expression can equivalently be written as
\[
\E\left[
\int_{\R^d\times[0,1]}
\Hc(y,v,\nub_t)\,\ell(t,\nub_t)(\mathrm{d}y,\mathrm{d}v)
\right].
\]
Thus, in the marginal identities derived from \eqref{eq:limit_deviation}, the occurrence of \(\mub_t\) may be replaced by \(\nub_t\). Notice that this is only a marginal replacement at each fixed time \(t\), and not an identification of the full processes \(\mub\) and \(\nub\).

We now define the limiting deviating dynamics. Let \((Y,V)\) solve
\[
(Y_0,V)\perp W\perp(\nub,W^\circ),
\qquad
\Lc(Y_0,V)=\Lc(X_0,U),
\]
and
\[
\mathrm{d}Y_t
=
\int_A b\big(t,Y_t,\overline p_{\nub_t}(Y_t,V),a\big)\,
\Lambda^\beta(t,Y_t,V,\nub_t)(\mathrm{d}a)\,\mathrm{d}t
+\sigma(t,Y_t)\,\mathrm{d}W_t
+\sigma_\circ\,\mathrm{d}W^\circ_t.
\]
Using \eqref{eq:limit_deviation}, together with the preceding marginal conditioning argument, and applying the superposition principle with common noise \cite[Theorems 1.3 and 1.5]{Lacker-Shkolnikov-Zhang_2020} to the enlarged state variable \((Y,V)\), we obtain that, for every \(t\in[0,T]\),
\[
\Lc(Y_t,V,\nub_t)(\mathrm{d}y,\mathrm{d}v,\mathrm{d}m)
=
\E\Big[
Z_t\,\delta_{(X_t,U)}(\mathrm{d}y,\mathrm{d}v)\,
\delta_{\mub_t}(\mathrm{d}m)
\Big].
\]
Again, this identity is to be understood at the level of the time marginal \(t\), and not as an equality in law of the full trajectories.
Equivalently, setting
\[
\mub_t^\beta:=\Lc(Y_t,V\mid \nub,W^\circ),
\]
we have
\[
\mub_t^\beta=\Lc(Y_t,V\mid \Gc_t),
\qquad
\Gc_t:=\sigma\{\nub_s,\sigma_\circ W^\circ_s:0\le s\le t\},
\]
and \(\Lc(\mub_t^\beta)\) coincides with the limit of the corresponding subsequence of
\(\Lc(\mub_t^{n,\betab^n})\). Hence the limiting dynamics are precisely those generated by the deviating relaxed control
\(\Lambda^\beta\).

\medskip

Finally, let \(F\) be bounded and continuous. Using again
\[
\Lc^{\P^{i,n}}(\Xbb^n)=\Lc(\Ybb^i),
\]
together with the convergence of \(\Qr^n\) toward \(\Qr\), the disintegration of \(\Gamma\), and the identification of the interaction variable, we obtain
\begin{align*}
&\lim_{n\to\infty}\frac1n\sum_{i=1}^n
\E\Bigg[\int_0^T
F\Big(
t,Y^{i,i}_t,(V_n*\mub^{i,n}_t)(Y^{i,i}_t),
\mub^{i,n}_t,\beta^{i,n}(t,\Ybb^i)
\Big)\,\mathrm{d}t\Bigg]
=
\E\Bigg[
\int_0^T\int_{\R_+\times A}
Z_t\,F(t,X_t,e,\mub_t,a)\,
\Gamma_t(\mathrm{d}e,\mathrm{d}a)\,\mathrm{d}t
\Bigg]
\\
&=
\E\Bigg[
\int_0^T\int_A
Z_t\,F\big(t,X_t,\overline p_{\mub_t}(X_t,U),\mub_t,a\big)\,
\Lambda^\beta(t,X_t,U,\mub_t)(\mathrm{d}a)\,\mathrm{d}t
\Bigg]
\\
&=
\E\Bigg[
\int_0^T\int_A
F\big(t,Y_t,\overline p_{\nub_t}(Y_t,V),\nub_t,a\big)\,
\Lambda^\beta(t,Y_t,V,\nub_t)(\mathrm{d}a)\,\mathrm{d}t
\Bigg].
\end{align*}
The last equality follows from the identity of the tilted law
\[
\E\big[
Z_t\,\delta_{(X_t,U,\mub_t)}
\big]
=
\Lc(Y_t,V,\nub_t).
\]
This is exactly the desired convergence of the deviating costs and concludes the proof.
\end{proof}

\medskip
We now specialize \Cref{weak_deviating} to specific feedback  deviations. Let
\[
\beta:[0,T]\times\R^d\times[0,1]\times\Pc(\R^d\times[0,1])\to A
\]
be Lipschitz continuous, and define the finite--player deviation by
\[
\beta^{i,n}(t,x_1,\dots,x_n)
:=
\beta(t,x_i,u_i^n,\overline m^n),
\qquad
\overline m^n:=\frac1n\sum_{j=1}^n\delta_{(x_j,u_j^n)}.
\]
Let \(X^\beta\) be the solution of
\begin{align*}
    \mathrm{d}X^\beta_t
    &=
    b\big(t,X^\beta_t,\overline p_{\nub_t}(X^\beta_t,U),
    \beta(t,X^\beta_t,U,\nub_t)\big)\,\mathrm{d}t
    +\sigma(t,X^\beta_t)\,\mathrm{d}W_t
    +\sigma_\circ\,\mathrm{d}W^\circ_t,
    \\
    \mub^\beta_t&:=\Lc(X^\beta_t,U\mid \nub,W^\circ).
\end{align*}

\begin{corollary}\label{weak_deviating_corollary}
Assume that, along a subsequence \((n_k)_{k\ge1}\),
\[
\Lc(\mub^{n_k}_t)\xrightarrow[k\to\infty]{}\Lc(\nub_t),
\qquad t\in[0,T],
\]
for some continuous process \(\mub\). Then, for every \(t\in[0,T]\),
\[
\Lc(\mub^{n_k,\beta^{n_k}}_t)
\xrightarrow[k\to\infty]{}
\Lc(\mub^\beta_t).
\]
Moreover, for every bounded continuous function \(F\),
\begin{align*}
&\lim_{k\to\infty}\frac1{n_k}\sum_{i=1}^{n_k}
\E\Bigg[\int_0^T
F\Big(
t,Y^{i,i}_t,
\Vh_{i,t}^{n_k}(Y^{i,i}_t),
\mub^{i,n_k}_t,
\beta^{i,n_k}(t,\Ybb^i)
\Big)\,\mathrm{d}t\Bigg]
=
\E\Bigg[\int_0^T
F\Big(
t,X^\beta_t,
\overline p_{\nub_t}(X^\beta_t,U),
\nub_t,
\beta(t,X^\beta_t,U,\nub_t)
\Big)\,\mathrm{d}t\Bigg].
\end{align*}
\end{corollary}

\begin{proof}
This is a direct consequence of \Cref{weak_deviating}. Indeed, the Lipschitz continuity of \(\beta\) implies that, along any convergent realization of the empirical measures,
\[
\beta^{i,n_k}(t,\Ybb^i)
=
\beta(t,Y^{i,i}_t,u_i^{n_k},\mub^{i,n_k}_t)
\]
will be associated in the weak convergence limit to
\[
\beta(t,X^\beta_t,U,\nub_t).
\]
The limiting relaxed control kernel in \Cref{weak_deviating} is therefore the Dirac kernel
\[
\Lambda^\beta(t,x,u,\overline{m})(\mathrm{d}a)
=
\delta_{\beta(t,x,u,\overline{m})}(\mathrm{d}a).
\]
Substituting this kernel into the conclusion of \Cref{weak_deviating} gives the stated convergence of the deviating empirical flows and of the associated averaged costs.
\end{proof}

\subsubsection{From approximate Nash equilibria to relaxed MFG solution (proof of {\rm \Cref{thm:from_n_to_MFG}})}

Let $(\alphab^n=(\alpha^{1,n},\cdots,\alpha^{n,n}))_{n \ge 1}$ be a sequence s.t. $\alphab^n$ is an $\delta_n$--Nash equilibrium for each $n \ge 1$ where $(\delta_n)_{n \ge 1}$ is a non--negative sequence with $\lim_{n \to \infty} \delta_n=0.$ We recall that we denote by $(\mub^n_t)_{t \in [0,T]}$ the empirical distribution of states and labels associated to $\alphab^n$ for each $n \ge 1$.

\begin{proposition}
    The sequence $(\Lc(\mub^n))_{n \ge 1}$ is relatively compact for the weak topology. For each convergent sub--sequence $(\Lc(\mub^{n_k}))_{n_k \ge 1}$, there exists a relaxed MFG solution $\mub=(\mub_t)_{t \in [0,T]}$ associated to $\Lambda \in {\M}$ s.t. for each $t \in [0,T]$,
    \begin{align*}
        \lim_{k \to \infty} \Lc(\mub^{n_k}_t)=\Lc(\mub_t).
    \end{align*}
\end{proposition}

\begin{proof}
    The relative compactness is straightforward. Let us take the limit $\Lc(\mub)$ of a convergent sub--sequence (not relabeled).
    By using \Cref{prop:weak_poc}, $\mub$ is a solution of a Fokker--Planck equation associated to a certain $\Lambda^\alpha \in {\M}$ and 
    \begin{align*}
        \lim_{n \to \infty} \frac{1}{n}\sum_{i=1}^nJ_i^n(\alphab^n)
        =
        J_{\mub}(\Lambda^\alpha).
    \end{align*}
    Let us verify the optimality condition. Let $\beta:[0,T] \x \R^d \x [0,1] \x \Pc(\R^d \x [0,1])  \to A$ be a Lipschitz map and consider $(X^\beta,U)$ the weak solution of :
\begin{align*}
    \mathrm{d}X^\beta_t
    = b\big(t,X^\beta_t,\overline{p}_{\mub_t}(X^\beta_t,U),\beta(t,X^\beta_t,U,\mub_t)\big)\,
    (\mathrm{d}a)\,\mathrm{d}t
    +\sigma(t,X^\beta_t)\,\mathrm{d}W_t + \sigma_\circ\,\mathrm{d}W^\circ_t,
    \qquad
    \mub^\beta_t=\Lc(X^\beta_t,U\mid \Gc_t),
\end{align*}
We set $\beta^{i,n}(t,x_1,\dots,x_n):=\beta(t,x_i,u^n_i,\overline{m}^n)$ where $\overline{m}^n:=\frac{1}{n} \sum_{i=1}^n \delta_{(x_i,u^n_i)}.$
By \Cref{weak_deviating} (in particular \Cref{weak_deviating_corollary}), with $\Lambda^\beta(t,x,u,m)(\mathrm{d}a):=\delta_{\beta(t,x,u,m)}(\mathrm{d}a)$, we have 
\begin{align*}
    \lim_{n \to \infty} \frac{1}{n} \sum_{i=1}^n J_i^n \left( (\alphab^n)^{-i},\beta^{i,n} \right) =J_{\mub}(\Lambda^\beta).
\end{align*}
Therefore,
\begin{align*}
    J_{\mub}(\Lambda^\beta)=\lim_{n \to \infty} \frac{1}{n} \sum_{i=1}^n J_i^n \left( (\alphab^n)^{-i},\beta^{i,n} \right) \le \lim_{n \to \infty} \frac{1}{n} \sum_{i=1}^n J_i^n \left( \alphab^n \right) = J_{\mub}(\Lambda^\alpha).
\end{align*}
This is true for any $\beta$ Lipschitz. Notice that the supremum in $\sup_{\Lambda^{\beta'} \in \M} J_{\mub}(\beta')$ can be taken over $\Lambda^{\beta'}(t,x,u,m)(\mathrm{d}a)=\delta_{\beta(t,x,u,m)}(\mathrm{d}a)$ with $\beta$ Lipchitz {\color{black}(see \Cref{prop:cong_graphon} and \Cref{rm:commonnoise_approx})}. This is enough to deduce the optimality and conclude.
\end{proof}

\subsubsection{From MFG solution to approximate Nash equilibria  (proof of {\rm \Cref{thm:from_MFG_to_n}})}

Let $\mub$ be a strict MFG solution associated to a control $\alpha:[0,T] \x \R^d \x [0,1] \x \Pc(\R^d \x [0,1]) \to A$. In particular, we have $\overline{\mu}_t=\Lc(X_t,U \mid \Gc_t)$ with $p_{\mub_t}(\cdot,\cdot)$ the density of $\overline{\mu}_t$, $\overline{p}_{\mub_t}(x,u):=\int_0^1 \Gr(u,v) p_{\mub_t}(x,v)\,\mathrm{d}v$ and
\begin{align*}
    \mathrm{d}X_t
    =
    b \left(t, X_t, \overline{p}_{\mub_t}(X_t,U),\Rr_{\mub_t}(U),\alpha(t,X_t,U,\mub_t) \right)\mathrm{d}t + \sigma(t,X_t)\,\mathrm{d}W_t + \sigma_\circ \,\mathrm{d}W^\circ_t.
\end{align*}
We consider a sequence of Lipschitz maps $(\alpha^\ell:[0,T] \x \R^d \x [0,1] \x \Pc(\R^d \x [0,1]) \to A)_{\ell \ge 1}$ such that $\lim_{\ell \to \infty} \alpha^\ell=\alpha$ a.e.
With $(\mub,W^\circ) \perp (X^i_0,W^i)_{i \ge 1}$, we introduce $\Xbb^{\ell,n}=(X^{\ell,1,n},\cdots,X^{\ell,n,n})$,
\begin{align*}
    \mathrm{d}X^{\ell,i,n}_t
    =
    b \left(t, X^{\ell,i,n}_t, M^{\ell,i,n}_t,\Rr^{\ell,n}_{i,t},\alpha^\ell(t,X^{\ell,i,n}_t,u^n_i,\mub_t) \right)\mathrm{d}t + \sigma(t,X^{\ell,i,n}_t)\,\mathrm{d}W^i_t + \sigma_\circ \,\mathrm{d}W^\circ_t,\, M^{\ell,i,n}_t:=\frac{1}{n} \sum_{j=1}^n \xi^n_{ij}V_n(X^{\ell,i,n}_t-X^{\ell,j,n}_t).
\end{align*}
We set $\alpha^{\ell,i,n}_{\mub}(t,x_1,\cdots, x_n):=\alpha^\ell(t,x_i,u^n_i,\mub)$.
Notice that, $(t,m,x_1,u^n_1,\dots,x_n,u^n_n) \mapsto \alpha^{\ell,i,n}_{m}(t,x_1,u^n_1,\cdots,x_n,u^n_n) $ is Borel measurable and a.e. $\om$, 
$$
    \alphab^{\ell,n}_{\mu(\om)}:=(\alpha^{\ell,1,n}_{\mub(\om)},\dots,\alpha^{\ell,n,n}_{\mub(\om)}) \in \Ac_n^n
$$
and 
\begin{align*}
    \Lc\left(X^{\ell,1,n}-\sigma_\circ\,W^\circ,\cdots, X^{\ell,n,n}- \sigma_\circ\,W^\circ \mid \Gc_T \right)(\om) = \Lc\left(X^{1,n,\,\alphab^{\ell,n}_{\mub(\om)}} - \sigma_\circ\,W^\circ(\om),\cdots, X^{n,n,\,\alphab^{\ell,n}_{\mub(\om)}} - \sigma_\circ\,W^\circ(\om) \right)
\end{align*}

\medskip
Let $\mub^{\ell,n}_t:=\frac{1}{n} \sum_{i=1}^n \delta_{(X^{\ell,i,n}_t,u^n_i)}$. By applying \Cref{corollary:weak_poc} with $\overline{m}_t=\mub_t(\om)$ and translated process $\Big(X^{1,n,\,\alphab^{\ell,n}_{\mub(\om)}} - \sigma_\circ\,W^\circ(\om),\cdots, X^{n,n,\,\alphab^{\ell,n}_{\mub(\om)}} - \sigma_\circ\,W^\circ(\om) \Big)$, we have for each $t \in [0,T]$, a.e. $\om$, $\lim_{\ell \to \infty}\lim_{n \to \infty} \Lc(\mub^{\ell,n}_t \mid \Gc_T)(\om)=\delta_{\overline{\mu}_t(\om)}$. In addition, for any continuous map $F$,
\begin{align} \label{eq:conv_mfg_proof}
\lim_{\ell \to \infty}\lim_{n\to\infty}\frac{1}{n}\sum_{i=1}^{n}\E\Bigg[\int_0^T
F\Big(t,X^{\ell,i,n}_t,M^{\ell,i,n}_t,\Rr^{\ell,n}_{i,t}, \alpha^{\ell,i,n}_{\mub}(t,\Xbb^{\ell,n})\Big)\,\mathrm{d}t\Bigg]
=
\E\Bigg[\int_{0}^T F\big(t,X_t,\overline{p}_{\mub_t}(X_t,U),\Rr_{\overline{\mu}_t}(U), \alpha(t,X_t,U,\mub_t)\big)\,\mathrm{d}t\Bigg].
\end{align}
Let us fix $\ell \ge 1$ and $n \ge 1$ for a moment. We are now treating $\Xbb^n$ as a multi--dimensional non--degenerate SDE. 
We will now apply some approximation result for non--degenerate SDE. Let $G(x_1,\cdots,x_n)$ be a smooth density on $(\R^d)^n$. There exists a sequence of Lipschitz maps $(\gamma_q:[0,T] \x (\R^d)^n \to \Pc(\R^d \x [0,1]))_{q \ge 1}$ s.t.
\begin{align*}
    \lim_{q \to \infty}\delta_{\gamma_q(t,\xbb)}(\mathrm{d}m)\,G(\xbb)\,\mathrm{d}\xbb = \Lc(\mub_t \mid \Xbb_t^n=\xbb)(\mathrm{d}m) \,G(\xbb)\, \mathrm{d}\xbb,\quad \xbb=(x_1,\dots,x_n).
\end{align*}
We introduce $M^{q,\ell,i,n}_t:=\frac{1}{n} \sum_{j=1}^n \xi^n_{ij}V_n(X^{q,\ell,i,n}_t-X^{\ell,j,n}_t)$, $\Rr^{q,\ell,n}_{i,t}:=\frac{1}{n} \sum_{j=1}^n \delta_{(\xi^n_{i,j},\,X^{q,\ell,j,n}_t)}$ and
\begin{align*}
    \mathrm{d}X^{q,\ell,i,n}_t
    =
    b \left(t, X^{q,\ell,i,n}_t, M^{q,\ell,i,n}_t,\Rr^{q,\ell,n}_{i,t},\alpha^\ell(t,X^{q,\ell,i,n}_t,u^n_i,\gamma_\ell(t,\Xbb^{q,\ell,n}_t)) \right)\mathrm{d}t + \sigma(t,X^{q,\ell,i,n}_t)\,\mathrm{d}W^i_t + \sigma_\circ \,\mathrm{d}W^\circ_t.
\end{align*}
By \cite[Proposition A.5.]{closed-loop-MFG_MDF}, for each $t \in [0,T]$, for the weak topology,
\begin{align} \label{eq:conv_markovian}
    \lim_{q \to \infty} \Lc(\Xbb^{q,\ell,n}_t)=\Lc(\Xbb^{\ell,n}_t).
\end{align}

Let us set $\alpha
^{q,\ell,i,n}(t,x_1,\cdots,x_n):=\alpha^\ell(t,x_i,u^n_i,\gamma_\ell(t,\xbb))$, $\alphab^{q,\ell,n}=(\alpha^{q,\ell,1,n},\cdots,\alpha^{q,\ell,n,n})$,
\begin{align*}
    \varepsilon_{q,\ell,i,n}
    := 
    \sup_{\beta} J_i^n \left( (\alphab^{q,\ell,n})^{-i}, \beta \right) - J^n_i(\alphab^{q,\ell,n}).
\end{align*}
We consider a sequence $(\beta^{q,\ell,i,n})_{1 \le i \le n}$ verifying $\varepsilon_{q,\ell,i,n} - \frac{1}{2^n} \le J_i^n \left( (\alphab^{q,\ell,n})^{-i}, \beta^{q,\ell,i,n} \right) - J^n_i(\alphab^{q,\ell,n})$, for each $n \ge 1$. Let us take the sub--sequence $(n_k, \ell_{n_k},q_{n_k})_{k \ge 1}$ realizing $\limsup_{\ell \to \infty}\limsup_{n \to \infty} \limsup_{q \to \infty} \frac{1}{n} \sum_{i=1}^n \varepsilon_{q,\ell,i,n}$. By \Cref{eq:conv_mfg_proof} (i.e. \Cref{corollary:weak_poc}) and \Cref{eq:conv_markovian} for each $t \in [0,T]$, with $\mub^{q,\ell,n}_t:=\frac{1}{n} \sum_{j=1}^n \delta_{(X^{q,\ell,j,n}_t,u^n_i)}$,
$$
    \lim_{\ell \to \infty}\lim_{n \to \infty} \lim_{q \to \infty} \Lc(\mub^{q,\ell,n}_t)=\lim_{\ell \to \infty}\lim_{n \to \infty} \Lc(\mub^{\ell,n}_t)=\Lc(\mub_t),
$$
By the previous limit  and  \Cref{weak_deviating}, there exists $\Lambda^\beta:[0,T]  \x \R^d \x [0,1] \x \Pc(\R^d \x [0,1]) \to \Pc(A)$ s.t.
\begin{align*}
    &\lim_{k \to \infty}\frac{1}{n_k} \sum_{i=1}^{n_k}J_i^{n_k} \left( (\alphab^{q_{n_k},\ell_{n_k},n_k})^{-i}, \beta^{q_{n_k},\ell_{n_k},i,n_k} \right)
    \\
    &=
    \E \left[ g \left( X^\beta_T, \Rr_{\mub_T}(U) \right) + \int_0^T L \left( t, X^\beta_t, \overline{p}_{\mub_t}(X^\beta_t,U), \Rr_{\mub_t}(U), a \right) \Lambda^\beta(t,X^\beta_t,U,\overline{\mu}_t) \mathrm{d}t \right] =J_{\overline{\mu}}(\Lambda^\beta).
\end{align*}
Since $\mub$ is a MFG solution, we deduce that $\limsup_{\ell \to \infty} \limsup_{n \to \infty} \limsup_{q \to \infty} \frac{1}{n} \sum_{i=1}^n \varepsilon_{q,\ell,i,n} = 0$. We deduce that $\alphab^{q,\ell,n}$ is an $(\varepsilon_{q,\ell,1,n},\dots,\varepsilon_{q,\ell,n,n})$--Nash equilibrium with $\lim_{\ell \to \infty} \lim_{n \to \infty} \lim_{q \to \infty} \frac{1}{n} \sum_{i=1}^n \varepsilon_{q,\ell,i,n} = 0$.

\begin{appendix}

\section{Some technical results}

\begin{lemma} \label{lemma:proba_equality}
Let $\mu$ and $\nu$ be two probability measures on $\R_+$ such that there exists $K_\nu>0$ with
\[
\nu([0,K_\nu])=1,
\]
and
\[
\int_{\R_+} x^k\,\nu(\mathrm{d}x)=\int_{\R_+} x^k\,\mu(\mathrm{d}x),\qquad \forall k\in\{1,2,\dots\}.
\]
Then $\mu=\nu$.
\end{lemma}

\begin{proof}
Fix $\varepsilon>0$ and $M>K_\nu$, and define
\[
g(x):=\mathbf{1}_{\{x\ge K_\nu+\varepsilon\}}\mathbf{1}_{\{x\le M\}},\qquad x\in\R_+.
\]
\textbf{Step 1 (continuous approximation of $g$).}
There exists a sequence of nonnegative continuous functions with compact support in $[0,M+\varepsilon]$,
$(f^k)_{k\ge1}\subset C_c(\R_+)$, such that
\[
0\le f^k\le 1,\qquad {\rm supp}(f^k)\subset[0,M+\varepsilon],\qquad f^k(x)\xrightarrow[k\to\infty]{} g(x)\quad\text{for a.e. }x\in\R_+.
\]
Moreover, $\sup_{k\ge1}\|f^k\|_\infty\le 1$.

\textbf{Step 2 (polynomial approximation of $f^k$).}
For each fixed $k\ge1$, since $f^k$ is continuous on the compact set $[0,M+\varepsilon]$, by Stone--Weierstrass
there exists a sequence of polynomials $(p^{n,k})_{n\ge1}$ such that
\[
\sup_{x\in[0,M+\varepsilon]}\big|p^{n,k}(x)-f^k(x)\big|\xrightarrow[n\to\infty]{}0.
\]

\textbf{Step 3 (pass to the limit under $\nu$ and use moment identities).}
Since $\nu([0,K_\nu])=1$ and $g\equiv0$ on $[0,K_\nu]$, we have
\[
\int_{\R_+} g(x)\,\nu(\mathrm{d}x)=0.
\]
On the other hand, by dominated convergence (using $0\le f^k\le 1$ and $ {\rm supp}(f^k)\subset[0,M+\varepsilon]$),
\[
\int_{\R_+} g\,\mathrm{d}\nu
=
\int_{0}^{M+\varepsilon} g\,\mathrm{d}\nu
=
\lim_{k\to\infty}\int_{0}^{M+\varepsilon} f^k\,\mathrm{d}\nu.
\]
For each fixed $k$, the uniform convergence $p^{n,k}\to f^k$ on $[0,M+\varepsilon]$ yields
\[
\int_{0}^{M+\varepsilon} f^k\,\mathrm{d}\nu
=
\lim_{n\to\infty}\int_{0}^{M+\varepsilon} p^{n,k}(x)\,\nu(\mathrm{d}x)
=
\lim_{n\to\infty}\int_{\R_+} p^{n,k}(x)\,\nu(\mathrm{d}x),
\]
where the last equality uses $ {\rm supp}(\nu)\subset[0,K_\nu]\subset[0,M+\varepsilon]$.
Next, since each $p^{n,k}$ is a polynomial and $\mu,\nu$ have identical moments of all orders,
\[
\int_{\R_+} p^{n,k}(x)\,\nu(\mathrm{d}x)=\int_{\R_+} p^{n,k}(x)\,\mu(\mathrm{d}x),\qquad \forall n,k.
\]
Combining the above, we obtain
\begin{align*}
0
=
\int_{\R_+} g\,\mathrm{d}\nu
&=
\lim_{k\to\infty}\lim_{n\to\infty}\int_{\R_+} p^{n,k}(x)\,\nu(\mathrm{d}x)
=
\lim_{k\to\infty}\lim_{n\to\infty}\int_{\R_+} p^{n,k}(x)\,\mu(\mathrm{d}x).
\end{align*}

\textbf{Step 4 (identify the mass of $\mu$ on $[K_\nu+\varepsilon,M]$).}
Write, for each $n,k$,
\[
\int_{\R_+} p^{n,k}\,\mathrm{d}\mu
=
\int_{0}^{M+\varepsilon} p^{n,k}\,\mathrm{d}\mu
+
\int_{\R_+\setminus[0,M+\varepsilon]} p^{n,k}\,\mathrm{d}\mu.
\]
Since $p^{n,k}\to f^k$ uniformly on $[0,M+\varepsilon]$, we have
\[
\lim_{n\to\infty}\int_{0}^{M+\varepsilon} p^{n,k}\,\mathrm{d}\mu
=
\int_{0}^{M+\varepsilon} f^k\,\mathrm{d}\mu.
\]
Moreover, by construction $p^{n,k}\ge0$, hence the tail term is nonnegative. Therefore,
\[
0
=
\lim_{k\to\infty}\lim_{n\to\infty}\int_{\R_+} p^{n,k}\,\mathrm{d}\mu
\ge
\lim_{k\to\infty}\int_{0}^{M+\varepsilon} f^k\,\mathrm{d}\mu.
\]
Letting $k\to\infty$ and using dominated convergence again (since $0\le f^k\le 1$ and $ {\rm supp}(f^k)\subset[0,M+\varepsilon]$),
\[
\lim_{k\to\infty}\int_{0}^{M+\varepsilon} f^k\,\mathrm{d}\mu
=
\int_{0}^{M+\varepsilon} g\,\mathrm{d}\mu
=
\mu([K_\nu+\varepsilon,M]).
\]
Hence $\mu([K_\nu+\varepsilon,M])=0$ for every $\varepsilon>0$ and every $M>K_\nu$.

\medskip
\textbf{Step 5 (compact support for $\mu$ and conclusion).}
Letting $M\to\infty$ yields $\mu([K_\nu+\varepsilon,\infty))=0$ for all $\varepsilon>0$, hence $\mu([0,K_\nu])=1$.
Thus both $\mu$ and $\nu$ are supported in the compact interval $[0,K_\nu]$ and have identical moments.

Finally, on $[0,K_\nu]$, polynomials are dense in $C([0,K_\nu])$ (Stone--Weierstrass). Since $\mu$ and $\nu$ agree on
all polynomial test functions, they agree on all continuous test functions, hence $\mu=\nu$.
\end{proof}

\medskip

We now establish a stability result with respect to the graphon variable. This compactness property will be the key ingredient in passing from the finite--support graphon approximation to a general graphon and also for regularizing controls. The idea is that, when a sequence of graphons converges in cut norm, the corresponding controlled McKean--Vlasov systems remain tight, and any limit point can again be identified as a controlled dynamics associated with the limiting graphon.

\medskip

Let
\[
\Pc_{\U}(\R^d)
:=
\Big\{
m\in\Pc(\R^d\times[0,1]):
m(\R^d\times \mathrm{d}u)=\mathrm{d}u
\Big\}.
\]
Let
\[
F:[0,T]\times\R^d\times[0,1]\times\R_+\times\Pc_{\U}(\R^d)\times A\to\R^d
\]
be bounded and continuous in \((e,p,a)\), uniformly in \((t,x,u)\).
Fix a graphon \(\Gr\) and a control kernel \(\Lambda^\alpha\). We denote by
\[
(X^{\alpha,\Gr},U)=(X,U)
\]
a weak solution of
\begin{align*}
    \mathrm{d}X_t
    &=
    \int_A
    F\Big(
        t,X_t,U,\overline p^{\alpha,\Gr}(t,X_t,U),
        \overline m_t^{\alpha,\Gr},a
    \Big)\,
    \Lambda^\alpha(t,X_t,U)(\mathrm{d}a)\,\mathrm{d}t
    +\sigma(t,X_t)\,\mathrm{d}W_t,
\end{align*}
where
\[
\overline m_t^{\alpha,\Gr}
=
\Lc(X_t,U)
=
p^{\alpha,\Gr}(t,x,u)\,\mathrm{d}x\,\mathrm{d}u
\]
and
\[
\overline p^{\alpha,\Gr}(t,x,u)
=
\int_0^1 \Gr(u,v)\,p^{\alpha,\Gr}(t,x,v)\,\mathrm{d}v.
\]

Let \((\Lambda^{\alpha^k},\Gr^k)_{k\ge1}\) be such that
\[
\Gr^k\to\Gr
\qquad\text{in cut norm}.
\]
We write
\[
p^k:=p^{\alpha^k,\Gr^k},
\qquad
\overline p^k:=\overline p^{\alpha^k,\Gr^k},
\qquad
\overline m^k:=\overline m^{\alpha^k,\Gr^k},
\]
and \(X^k:=X^{\alpha^k,\Gr^k}\). We also introduce the occupation measure
\begin{align*}
    \Gamma^k_t(\mathrm{d}x,\mathrm{d}u,\mathrm{d}e,\mathrm{d}a)\,\mathrm{d}t
    &:=
    \E\Big[
        \delta_{(X^k_t,U,\overline p^k(t,X^k_t,U))}
        (\mathrm{d}x,\mathrm{d}u,\mathrm{d}e)\,
        \Lambda^{\alpha^k}(t,X^k_t,U)(\mathrm{d}a)
    \Big]\,\mathrm{d}t.
\end{align*}
The variable \(e\) records the interaction term \(\overline p^k(t,X^k_t,U)\).

\begin{proposition}\label{prop:cong_graphon}
The sequence \((\overline m^k,\Gamma^k)_{k\ge1}\) is relatively compact. Moreover, for any convergent subsequence
\[
(\overline m^{k_\ell},\Gamma^{k_\ell})_{\ell\ge1},
\]
there exists a control kernel \(\Lambda^\alpha\) such that
\[
\overline m^{k_\ell}\xRightarrow[\ell\to\infty]{}\overline m^{\alpha,\Gr}
\]
and
\[
\Gamma^{k_\ell}
\Longrightarrow
\delta_{\overline p^{\alpha,\Gr}(t,x,u)}(\mathrm{d}e)\,
\Lambda^\alpha(t,x,u)(\mathrm{d}a)\,
p^{\alpha,\Gr}(t,x,u)\,\mathrm{d}x\,\mathrm{d}u\,\mathrm{d}t.
\]
If, in addition,
\[
\Lambda^{\alpha^k}(t,x,u)\Longrightarrow \Lambda(t,x,u)
\qquad\text{for a.e. }(t,x,u),
\]
then \(\Lambda^\alpha=\Lambda\) \(\mathrm{d}t\,\mathrm{d}x\,\mathrm{d}u\)-a.e. In particular, every limit point is again associated with a controlled McKean--Vlasov dynamics driven by the limiting graphon \(\Gr\).
\end{proposition}

\begin{proof}
The relative compactness of \((\overline m^k)_{k\ge1}\) follows from the boundedness of \(F\), the growth assumptions on \(\sigma\), and standard tightness estimates for the corresponding SDEs. Let \(\overline m\) be the limit of a convergent subsequence, not relabeled. By the non--degeneracy of \(\sigma\), for every \(t\in(0,T]\), \(\overline m_t\) admits a density, and we write
\[
\overline m_t(\mathrm{d}x,\mathrm{d}u)
=
p(t,x,u)\,\mathrm{d}x\,\mathrm{d}u.
\]

We first identify the limit of the interaction variables. We claim that, for every integer \(\ell\ge1\), every smooth bounded \(\phi:[0,T] \x \R^d\to\R\), and every smooth bounded \(v:[0,1]\to\R\),
\begin{align}\label{eq:convg_density_better}
    \lim_{k\to\infty}
    \E\Bigg[
        \int_0^T
        \phi(t,X^k_t)\,v(U)\,
        \big|\overline p^k(t,X^k_t,U)\big|^\ell
        \,\mathrm{d}t
    \Bigg]
    =
    \int_0^T\int_{\R^d\times[0,1]}
    \phi(t,x)v(u)\,
    |\overline p(t,x,u)|^\ell\,
    p(t,x,u)\,\mathrm{d}x\,\mathrm{d}u\,\mathrm{d}t,
\end{align}
where $\overline p(t,x,u):=\int_0^1 \Gr(u,v)p(t,x,v)\,\mathrm{d}v.$
We prove this for \(\ell=2\); the general case is identical, replacing the quadratic product by an \(\ell\)--fold product. For each \(k\),
\begin{align*}
    &\E\Bigg[
        \int_0^T
        \phi(t,X^k_t)v(U)\,
        |\overline p^k(t,X^k_t,U)|^2
        \,\mathrm{d}t
    \Bigg]
    =
    \int_0^T\int_{\R^d}
    \phi(t,x)\,P^k(t,x,x,x)\,\mathrm{d}x\,\mathrm{d}t,
\end{align*}
where
\begin{align*}
    P^k(t,x,x_1,x_2)
    :=
    \int_{[0,1]^3}
    v(u)\Gr^k(u,v_1)\Gr^k(u,v_2)
    p^k(t,x_1,v_1)p^k(t,x_2,v_2)p^k(t,x,u)
    \,\mathrm{d}v_1\,\mathrm{d}v_2\,\mathrm{d}u.
\end{align*}
By the uniform Hölder estimates on the densities, as in \Cref{lemma:estimations}, the family \((P^k)_{k\ge1}\) is relatively compact for the locally uniform topology on compact subsets of \((0,T]\times(\R^d)^3\). Let \(P\) be the limit of a further subsequence. Then
\[
\lim_{k\to\infty}
\E\Bigg[
\int_0^T
\phi(t,X^k_t)v(U)|\overline p^k(t,X^k_t,U)|^2\,\mathrm{d}t
\Bigg]
=
\int_0^T\int_{\R^d}
\phi(t,x)P(t,x,x,x)\,\mathrm{d}x\,\mathrm{d}t.
\]

We now identify \(P\). Let \(\varphi:[0,T]\times(\R^d)^3\to\R\) be smooth and compactly supported. Using the cut--norm convergence \(\Gr^k\to\Gr\), the convergence \(\overline m^k\to\overline m\), and the definition of \(P^k\), we obtain
\begin{align*}
&\int_{[0,T]\times(\R^d)^3}
\varphi(t,x,x_1,x_2)P(t,x,x_1,x_2)
\,\mathrm{d}x\,\mathrm{d}x_1\,\mathrm{d}x_2\,\mathrm{d}t
\\
&=
\lim_{k\to\infty}
\int_{[0,T]\times(\R^d\times[0,1])^3}
\varphi(t,x,x_1,x_2)
v(u)\Gr^k(u,v_1)\Gr^k(u,v_2)
\\
&\hspace{4cm}\times
\overline m^k_t(\mathrm{d}x_1,\mathrm{d}v_1)
\overline m^k_t(\mathrm{d}x_2,\mathrm{d}v_2)
\overline m^k_t(\mathrm{d}x,\mathrm{d}u)
\,\mathrm{d}t
\\
&=
\int_{[0,T]\times(\R^d\times[0,1])^3}
\varphi(t,x,x_1,x_2)
v(u)\Gr(u,v_1)\Gr(u,v_2)
\\
&\hspace{4cm}\times
p(t,x_1,v_1)p(t,x_2,v_2)p(t,x,u)
\,\mathrm{d}x_1\,\mathrm{d}v_1\,
\mathrm{d}x_2\,\mathrm{d}v_2\,
\mathrm{d}x\,\mathrm{d}u\,\mathrm{d}t.
\end{align*}
Hence
\[
P(t,x,x_1,x_2)
=
\int_{[0,1]^3}
v(u)\Gr(u,v_1)\Gr(u,v_2)
p(t,x_1,v_1)p(t,x_2,v_2)p(t,x,u)
\,\mathrm{d}v_1\,\mathrm{d}v_2\,\mathrm{d}u.
\]
Taking \(x_1=x_2=x\), we get the desired identity for \(\ell=2\). The proof of \eqref{eq:convg_density_better} for general \(\ell\) is the same.

\medskip

In particular, for each \(\ell\ge1\),
\[
\sup_{k\ge1}
\E\Bigg[
\int_0^T
|\overline p^k(t,X^k_t,U)|^\ell\,\mathrm{d}t
\Bigg]<\infty.
\]
Since \(A\) is compact and \(F\) is bounded, this moment estimate gives tightness of \((\Gamma^k)_{k\ge1}\). Thus, along the same subsequence, we may assume that
\[
\Gamma^k\Longrightarrow \Gamma.
\]
Moreover, the moment bound implies convergence against all continuous test functions with polynomial growth in the \(e\)-variable and bounded dependence in the remaining variables.
By definition of \(\Gamma^k\), for every \(\ell\ge1\), every bounded continuous \(\phi\) and \(v\),
\begin{align*}
\int_{[0,T]\times\R^d\times[0,1]\times\R_+}
\phi(t,x)v(u)|e|^\ell
\Gamma^k_t(\mathrm{d}x,\mathrm{d}u,\mathrm{d}e,A)\,\mathrm{d}t
=
\E\Bigg[
\int_0^T
\phi(t,X^k_t)v(U)|\overline p^k(t,X^k_t,U)|^\ell
\,\mathrm{d}t
\Bigg].
\end{align*}
Passing to the limit and using \eqref{eq:convg_density_better}, we obtain
\begin{align*}
&\int_{[0,T]\times\R^d\times[0,1]\times\R_+}
\phi(t,x)v(u)|e|^\ell
\Gamma_t(\mathrm{d}x,\mathrm{d}u,\mathrm{d}e,A)\,\mathrm{d}t
=
\int_0^T\int_{\R^d\times[0,1]}
\phi(t,x)v(u)|\overline p(t,x,u)|^\ell
p(t,x,u)\,\mathrm{d}x\,\mathrm{d}u\,\mathrm{d}t.
\end{align*}
Since the above identity holds for all moments, \Cref{lemma:proba_equality} yields
\[
\Gamma_t(\mathrm{d}x,\mathrm{d}u,\mathrm{d}e,A)
=
p(t,x,u)\,
\delta_{\overline p(t,x,u)}(\mathrm{d}e)
\,\mathrm{d}x\,\mathrm{d}u
\quad
\text{for a.e. }t.
\]

\medskip

We now identify the limiting flow. For each \(k\), \((\overline m^k,\Gamma^k)\) satisfies the weak Fokker--Planck equation. Passing to the limit gives, for all \(f\in C_c^\infty(\R^d\times[0,1])\) and all \(t\in[0,T]\),
\begin{align*}
\langle f,\overline m_t\rangle
&=
\langle f,\overline m_0\rangle
+
\int_0^t
\int_{\R^d\times[0,1]\times\R_+\times A}
\nabla_x f(x,u)\cdot
F(s,x,u,e,\overline m_s,a)
\,\Gamma_s(\mathrm{d}x,\mathrm{d}u,\mathrm{d}e,\mathrm{d}a)\,\mathrm{d}s
\\
&\quad+
\frac12\int_0^t
\int_{\R^d\times[0,1]}
{\rm Tr}\!\left[
\nabla^2_{xx}f(x,u)\sigma(s,x)\sigma^\top(s,x)
\right]\overline m_s(\mathrm{d}x,\mathrm{d}u)\,\mathrm{d}s.
\end{align*}
Let \(\widehat\Gamma\) be a disintegration kernel such that
\[
\Gamma_s(\mathrm{d}x,\mathrm{d}u,\mathrm{d}e,\mathrm{d}a)
=
\widehat\Gamma(s,x,u,e)(\mathrm{d}a)\,
\Gamma_s(\mathrm{d}x,\mathrm{d}u,\mathrm{d}e,A).
\]
Using the identification of the \(e\)--marginal, we get
\begin{align*}
\langle f,\overline m_t\rangle
&=
\langle f,\overline m_0\rangle
+
\int_0^t
\int_{\R^d\times[0,1]\times A}
\nabla_x f(x,u)\cdot
F(s,x,u,\overline p(s,x,u),\overline m_s,a)
\\
&\hspace{4cm}\times
\widehat\Gamma(s,x,u,\overline p(s,x,u))(\mathrm{d}a)\,
\overline m_s(\mathrm{d}x,\mathrm{d}u)\,\mathrm{d}s
\\
&\quad+
\frac12\int_0^t
\int_{\R^d\times[0,1]}
{\rm Tr}\!\left[
\nabla^2_{xx}f(x,u)\sigma(s,x)\sigma^\top(s,x)
\right]\overline m_s(\mathrm{d}x,\mathrm{d}u)\,\mathrm{d}s.
\end{align*}
Set
\[
\Lambda^\alpha(t,x,u)
:=
\widehat\Gamma(t,x,u,\overline p(t,x,u)).
\]
By the superposition principle, there exists a weak solution \((X,U)\) such that $\overline m_t=\Lc(X_t,U)$
and
\[
\mathrm{d}X_t
=
\int_A
F(t,X_t,U,\overline p(t,X_t,U),\overline m_t,a)
\Lambda^\alpha(t,X_t,U)(\mathrm{d}a)\,\mathrm{d}t
+\sigma(t,X_t)\,\mathrm{d}W_t.
\]
Thus
\[
\overline m=\overline m^{\alpha,\Gr}.
\]
Consequently,
\[
\Gamma^k\Longrightarrow
\delta_{\overline p^{\alpha,\Gr}(t,x,u)}(\mathrm{d}e)
\Lambda^\alpha(t,x,u)(\mathrm{d}a)
p^{\alpha,\Gr}(t,x,u)\,\mathrm{d}x\,\mathrm{d}u\,\mathrm{d}t.
\]

\medskip

It remains to prove the last statement. Assume that
\[
\Lambda^{\alpha^k}(t,x,u)\Longrightarrow \Lambda(t,x,u),
\qquad
\text{for a.e. }(t,x,u).
\]
The conclusion $\Lambda=\Lambda^\alpha$ can be obtain by applying the exact same technique/result of \cite[Proposition 4.6]{djete2025nonexchangeablemeanfieldcontrol}.

\end{proof}

\begin{remark} \label{rm:commonnoise_approx}
    For clarity, {\rm \Cref{prop:cong_graphon}} is stated and proved in the absence of common noise. The same argument, however, extends to a common--noise setting under a uniqueness condition. Indeed, one may freeze a realization of the common noise path \(\sigma_\circ W^\circ(\omega)\) and apply the no--common-noise result to the corresponding dynamics with translated coefficients, obtained by shifting the state variable by \(\sigma_\circ W^\circ(\omega)\).
The only subtle point is that the subsequence extracted in {\rm\Cref{prop:cong_graphon}} may a priori depend on the realization \(\omega\). This difficulty disappears when the limiting graphon equation admits a unique density \(p^{\alpha,\Gr}\), for instance under the conditions of {\rm\Cref{prop:uniqueness_FP}} below. In that case, every subsequence has the same limit, and therefore the whole sequence converges. Thus the pathwise no-common--noise argument yields the desired convergence in the common--noise setting.
\end{remark}

We finish this appendix section by providing a uniqueness result for McKean--Vlsasov equation involving a density dependence. For simplicity, we assume here that $\sigma = \Ir_d$.
Let $\Gr:[0,1]^2\to\Er$ be a graphon, and let
\[
B:[0,T]\times\R^d\times\R\times\Pc(\Er\times\R^d)\to\R^d
\]
be a bounded Borel map, uniformly Lipschitz continuous in its last two variables. More precisely, there exists
$C_B>0$ such that, for every $(t,x)\in[0,T]\times\R^d$, every $p,p'\in\R_+$, and every
$m,m'\in\Pc(\Er\times\R^d)$,
\[
|B(t,x,p,m)-B(t,x,p',m')|
\le C_B\Big(|p-p'|+\|m-m'\|_{\rm TV}\Big)
\]
where $\|\|_{\rm TV}$ is the total variation distance.
We say that a flow of probability measures
\[
\overline{m}=(\overline{m}_t)_{t\in[0,T]}\subset C([0,T];\Pc(\R^d\times[0,1]))
\]
is a \emph{Fokker--Planck weak solution} starting from $\nub_0$ if $\overline{m}_0=\nub_0$ and, for each $t\in[0,T]$,
\[
\overline{m}_t=\Lc(X_t,U),
\]
where $(X,U)$ satisfies: $(X_0,U)\perp W$, $U\sim{\rm Unif}([0,1])$, $\Lc(X_0,U)=\nub_0$,  $W$ is an $\F$--Brownian motion, and
\begin{align*}
    \mathrm{d}X_t
    =
    B\Big(t,X_t,\overline{p}_{\overline{m}_t}(t,X_t,U),\Rr_{\overline{m}_t}(U)\Big)\,\mathrm{d}t
    +\mathrm{d}W_t.
\end{align*}
Here,
\begin{align*}
    \Rr_{\overline{m}_t}(u)(\mathrm{d}e,\mathrm{d}y)
    &=
    \int_{[0,1]} \delta_{\Gr(u,v)}(\mathrm{d}e)\,\overline{m}_t(\mathrm{d}y,\mathrm{d}v),\\
    \overline{p}_{\overline{m}_t}(x,u)
    &:=
    \int_0^1 \Gr(u,v)\,p_{\overline{m}_t}(x,v)\,\mathrm{d}v,
\end{align*}
where $p_{\overline{m}_t}$ denotes the density of $\overline{m}_t$ with respect to $\mathrm{d}x\,\mathrm{d}u$.

\begin{proposition} \label{prop:uniqueness_FP}
If $\overline{m}^1$ and $\overline{m}^2$ are two Fokker--Planck weak solutions starting from the same initial condition $\nub_0$, then
\[
\overline{m}^1=\overline{m}^2.
\]
\end{proposition}

\begin{proof}
Let $(X^1,U^1,W^1)$ and $(X^2,U^2,W^2)$ be two weak realizations associated with $\overline{m}^1$ and $\overline{m}^2$, respectively. Thus, for $i=1,2$, $\overline{m}^i_t=\Lc(X^i_t,U^i)$, $t\in[0,T],$
and
\begin{align*}
    \mathrm{d}X^i_t
    =
    B\Big(t,X^i_t,\overline{p}_{\overline{m}^i_t}(t,X^i_t,U^i),\Rr_{\overline{m}^i_t}(U^i)\Big)\,\mathrm{d}t
    +\mathrm{d}W^i_t.
\end{align*}

We now place ourselves on a common auxiliary probability space carrying random variables $(X_0,U,W)$ such that
\[
\Lc(X_0,U,W)=\Lc(X^i_0,U^i,W^i),\qquad i=1,2.
\]
In particular,
\[
(X_0,U)\perp W,\qquad U\sim{\rm Unif}([0,1]),\qquad \Lc(X_0,U)=\nub_0.
\]

For each $i\in\{1,2\}$, let $(L^i_t)_{t\in[0,T]}$ be the stochastic exponential defined by
\[
L^i_0=1,
\qquad
\mathrm{d}L^i_t
=
L^i_t\,B\Big(t,X_0+W_t,\overline{p}_{\overline{m}^i_t}(t,X_0+W_t,U),\Rr_{\overline{m}^i_t}(U)\Big)\cdot\mathrm{d}W_t.
\]
Since $B$ is bounded, Novikov's criterion is satisfied, so $L^i$ is a true martingale. We may therefore define a new probability measure
$\Pb^i$ by
\[
\frac{\mathrm{d}\Pb^i}{\mathrm{d}\P}=L^i_T.
\]
By Girsanov's theorem, the process
\[
\overline{W}^i_\cdot
:=
W_\cdot-\int_0^\cdot B\Big(s,X_0+W_s,\overline{p}_{\overline{m}^i_s}(s,X_0+W_s,U),\Rr_{\overline{m}^i_s}(U)\Big)\,\mathrm{d}s
\]
is a $\Pb^i$--Brownian motion conditionally on $U$. Hence,
\[
\Lc^{\Pb^i}(X_0,\overline{W}^i,U)=\Lc(X_0,W,U),
\]
and therefore
\[
\Lc^{\Pb^i}(X_0+W\mid U=u)=\Lc(X^i\mid U^i=u).
\]
In particular, for every $t\in[0,T]$ and every $u\in[0,1]$, the conditional law of $X_0+W_t$ under $\Pb^i(\cdot\mid U=u)$ admits density
$p_{\overline{m}^i_t}(\cdot,u)$.

\medskip

Fix now $t_0\in[0,T]$ and let $g:\R^d\to\R$ be bounded Borel. Consider the backward stochastic differential equation
\[
\mathrm{d}Y_t=Z_t\cdot\mathrm{d}W_t,\qquad t\in[0,t_0],
\qquad
Y_{t_0}=g(X_0+W_{t_0}).
\]
Since the terminal condition is bounded and Markovian, there exist measurable functions
\[
\Uc:[0,t_0]\times\R^d\to\R,
\qquad
v:[0,t_0]\times\R^d\to\R^d
\]
such that
\[
Y_t=\Uc(t,X_0+W_t),\qquad Z_t=v(t,X_0+W_t),\qquad t\in[0,t_0].
\]
Moreover, by the standard gradient estimate for the heat semigroup, for a.e. $t\in(0,t_0)$,
\begin{align}\label{eq:grad_bound_uniqueness}
    \sup_{x\in\R^d}|v(t,x)|
    \le \frac{\|g\|_\infty}{\sqrt{t_0-t}}.
\end{align}

We next compute the duality relation satisfied by the densities. For each $i=1,2$, using that
\[
g(X_0+W_{t_0})=Y_{t_0},
\]
we get
\begin{align*}
    \int_{\R^d} g(x)\,p_{\overline{m}^i_{t_0}}(x,u)\,\mathrm{d}x
    &=
    \E^{\Pb^i}\big[g(X_0+W_{t_0})\mid U=u\big]
    =
    \E^{\Pb^i}\big[Y_{t_0}\mid U=u\big].
\end{align*}
Now, under $\Pb^i$, the process $W$ has drift
\[
B\Big(t,X_0+W_t,\overline{p}_{\overline{m}^i_t}(t,X_0+W_t,U),\Rr_{\overline{m}^i_t}(U)\Big),
\]
so the martingale representation under $\P$ yields, under $\Pb^i$,
\[
Y_{t_0}
=
Y_0+\int_0^{t_0} Z_t\cdot
B\Big(t,X_0+W_t,\overline{p}_{\overline{m}^i_t}(t,X_0+W_t,U),\Rr_{\overline{m}^i_t}(U)\Big)\,\mathrm{d}t
+\text{local martingale}.
\]
Taking conditional expectation given $U=u$, the martingale term disappears and we obtain
\begin{align*}
    \int_{\R^d} g(x)\,p_{\overline{m}^i_{t_0}}(x,u)\,\mathrm{d}x
    &=
    \E^{\Pb^i}\big[\Uc(0,X_0)\mid U=u\big]
    \\
    &\quad+
    \int_0^{t_0}
    \E^{\Pb^i}\Big[
        v(t,X_0+W_t)\cdot
        B\Big(t,X_0+W_t,\overline{p}_{\overline{m}^i_t}(t,X_0+W_t,u),\Rr_{\overline{m}^i_t}(u)\Big)
        \,\Big|\,U=u
    \Big]\mathrm{d}t.
\end{align*}
Using again that under $\Pb^i(\cdot\mid U=u)$ the law of $X_0+W_t$ has density $p_{\overline{m}^i_t}(\cdot,u)$, and that
$\overline{m}^i_0=\nub_0$ for both $i=1,2$, this yields
\begin{align*}
    \int_{\R^d} g(x)\,p_{\overline{m}^i_{t_0}}(x,u)\,\mathrm{d}x
    &=
    \int_{\R^d} \Uc(0,x)\,p_{\overline{m}^i_{0}}(x,u)\,\mathrm{d}x
    \\
    &\quad+
    \int_0^{t_0}\int_{\R^d}
    v(t,x)\cdot
    B\Big(t,x,\overline{p}_{\overline{m}^i_t}(t,x,u),\Rr_{\overline{m}^i_t}(u)\Big)\,
    p_{\overline{m}^i_t}(x,u)\,\mathrm{d}x\,\mathrm{d}t.
\end{align*}
Since the initial conditions coincide, the initial terms cancel when subtracting the two identities corresponding to $i=1$ and $i=2$. We therefore obtain
\begin{align*}
    &\left|
    \int_{\R^d} g(x)\,p_{\overline{m}^1_{t_0}}(x,u)\,\mathrm{d}x
    -
    \int_{\R^d} g(x)\,p_{\overline{m}^2_{t_0}}(x,u)\,\mathrm{d}x
    \right|
    \\
    &\le
    \int_0^{t_0}\int_{\R^d}
    |v(t,x)|\,
    \big|
    B\big(t,x,\overline{p}_{\overline{m}^1_t}(t,x,u),\Rr_{\overline{m}^1_t}(u)\big)
    -
    B\big(t,x,\overline{p}_{\overline{m}^2_t}(t,x,u),\Rr_{\overline{m}^2_t}(u)\big)
    \big|
    p_{\overline{m}^1_t}(x,u)\,\mathrm{d}x\,\mathrm{d}t
    \\
    &\quad+
    \|B\|_\infty
    \int_0^{t_0}\int_{\R^d}
    |v(t,x)|\,|p_{\overline{m}^1_t}(x,u)-p_{\overline{m}^2_t}(x,u)|
    \,\mathrm{d}x\,\mathrm{d}t.
\end{align*}
Using the Lipschitz continuity of $B$ in its last two variables, we infer
\begin{align*}
    &\left|
    \int_{\R^d} g(x)\,p_{\overline{m}^1_{t_0}}(x,u)\,\mathrm{d}x
    -
    \int_{\R^d} g(x)\,p_{\overline{m}^2_{t_0}}(x,u)\,\mathrm{d}x
    \right|
    \\
    &\le
    C_B
    \int_0^{t_0}\int_{\R^d}
    |v(t,x)|\,
    \Big(
    \big|\overline{p}_{\overline{m}^1_t}(t,x,u)-\overline{p}_{\overline{m}^2_t}(t,x,u)\big|
    +
    \big\|\Rr_{\overline{m}^1_t}(u)-\Rr_{\overline{m}^2_t}(u)\big\|_{\rm TV}
    \Big)
    p_{\overline{m}^1_t}(x,u)\,\mathrm{d}x\,\mathrm{d}t
    \\
    &\quad+
    \|B\|_\infty
    \int_0^{t_0}\int_{\R^d}
    |v(t,x)|\,|p_{\overline{m}^1_t}(x,u)-p_{\overline{m}^2_t}(x,u)|
    \,\mathrm{d}x\,\mathrm{d}t.
\end{align*}

We now estimate the two quantities appearing in the Lipschitz term. Since $\Gr$ is bounded,
\begin{align*}
\big|\overline{p}_{\overline{m}^1_t}(t,x,u)-\overline{p}_{\overline{m}^2_t}(t,x,u)\big|
&=
\left|
\int_0^1 \Gr(u,v)\big(p_{\overline{m}^1_t}(x,v)-p_{\overline{m}^2_t}(x,v)\big)\,\mathrm{d}v
\right|
\\
&\le
\|\Gr\|_\infty \int_0^1 \big|p_{\overline{m}^1_t}(x,v)-p_{\overline{m}^2_t}(x,v)\big|\,\mathrm{d}v.
\end{align*}
Integrating in $x$, this gives
\[
\int_{\R^d}\big|\overline{p}_{\overline{m}^1_t}(t,x,u)-\overline{p}_{\overline{m}^2_t}(t,x,u)\big|\,\mathrm{d}x
\le
\|\Gr\|_\infty \|p_{\overline{m}^1_t}-p_{\overline{m}^2_t}\|_{L^1(\R^d\times[0,1])}.
\]
Similarly, by definition of $\Rr_{\overline{m}_t}(u)$ and using the boundedness of $\Gr$, one has
\[
\big\|\Rr_{\overline{m}^1_t}(u)-\Rr_{\overline{m}^2_t}(u)\big\|_{\rm TV}
\le C\,\|p_{\overline{m}^1_t}-p_{\overline{m}^2_t}\|_{L^1(\R^d\times[0,1])},
\]
for some constant $C$ depending only on $\Gr$.
Combining these estimates, and using that $p_{\overline{m}^1_t}(\cdot,u)$ is a probability density in $x$, we get
\begin{align*}
    &\left|
    \int_{\R^d} g(x)\,p_{\overline{m}^1_{t_0}}(x,u)\,\mathrm{d}x
    -
    \int_{\R^d} g(x)\,p_{\overline{m}^2_{t_0}}(x,u)\,\mathrm{d}x
    \right|
    \\
    &\le
    C
    \int_0^{t_0}\sup_{y\in\R^d}|v(t,y)|
    \left(
        \|p_{\overline{m}^1_t}(\cdot,u)-p_{\overline{m}^2_t}(\cdot,u)\|_{L^1(\R^d)}
        +
        \|p_{\overline{m}^1_t}-p_{\overline{m}^2_t}\|_{L^1(\R^d\times[0,1])}
    \right)\,\mathrm{d}t,
\end{align*}
for some constant $C>0$ depending only on $B$ and $\Gr$.
Taking the supremum over all bounded Borel functions $g$ with $\|g\|_\infty\le1$, we deduce
\begin{align}\label{eq:L1_u_fixed}
    \|p_{\overline{m}^1_{t_0}}(\cdot,u)-p_{\overline{m}^2_{t_0}}(\cdot,u)\|_{L^1(\R^d)}
    \le
    C
    \int_0^{t_0}\sup_{\|g\|_\infty\le1}\sup_{y\in\R^d}|v(t,y)|
    \left(
        \|p_{\overline{m}^1_t}(\cdot,u)-p_{\overline{m}^2_t}(\cdot,u)\|_{L^1(\R^d)}
        +
        \|p_{\overline{m}^1_t}-p_{\overline{m}^2_t}\|_{L^1(\R^d\times[0,1])}
    \right)\mathrm{d}t.
\end{align}

We now integrate this inequality with respect to $u\in[0,1]$. Since
\[
\|p_{\overline{m}^1_t}-p_{\overline{m}^2_t}\|_{L^1(\R^d\times[0,1])}
=
\int_0^1 \|p_{\overline{m}^1_t}(\cdot,u)-p_{\overline{m}^2_t}(\cdot,u)\|_{L^1(\R^d)}\,\mathrm{d}u,
\]
we obtain
\begin{align*}
    \|p_{\overline{m}^1_{t_0}}-p_{\overline{m}^2_{t_0}}\|_{L^1(\R^d\times[0,1])}
    \le
    C
    \int_0^{t_0}
    \sup_{\|g\|_\infty\le1}\sup_{y\in\R^d}|v(t,y)|
    \,
    \|p_{\overline{m}^1_t}-p_{\overline{m}^2_t}\|_{L^1(\R^d\times[0,1])}
    \,\mathrm{d}t.
\end{align*}

To exploit the singular estimate \eqref{eq:grad_bound_uniqueness}, let us fix $q\in(1,2)$ and denote by
\[
q^\star:=\frac{q}{q-1}>1
\]
its conjugate exponent. By Hölder's inequality,
\begin{align*}
    \|p_{\overline{m}^1_{t_0}}-p_{\overline{m}^2_{t_0}}\|_{L^1}
    &\le
    C
    \left(
        \int_0^{t_0}
        \Big(\sup_{\|g\|_\infty\le1}\sup_{y\in\R^d}|v(t,y)|\Big)^q\,\mathrm{d}t
    \right)^{1/q}
    \left(
        \int_0^{t_0}
        \|p_{\overline{m}^1_t}-p_{\overline{m}^2_t}\|_{L^1}^{q^\star}\,\mathrm{d}t
    \right)^{1/q^\star}.
\end{align*}
Using \eqref{eq:grad_bound_uniqueness}, for every bounded $g$ with $\|g\|_\infty\le1$,
\[
\sup_{y\in\R^d}|v(t,y)|
\le \frac{1}{\sqrt{t_0-t}},
\]
hence
\begin{align*}
    \int_0^{t_0}
    \Big(\sup_{\|g\|_\infty\le1}\sup_{y\in\R^d}|v(t,y)|\Big)^q\,\mathrm{d}t
    &\le
    \int_0^{t_0}(t_0-t)^{-q/2}\,\mathrm{d}t
    =
    \frac{2}{2-q}\,t_0^{(2-q)/2}.
\end{align*}
Therefore,
\begin{align*}
    \|p_{\overline{m}^1_{t_0}}-p_{\overline{m}^2_{t_0}}\|_{L^1}
    \le
    C
    \left(\frac{2}{2-q}\right)^{1/q}
    t_0^{\frac{2-q}{2q}}
    \left(
        \int_0^{t_0}
        \|p_{\overline{m}^1_t}-p_{\overline{m}^2_t}\|_{L^1}^{q^\star}\,\mathrm{d}t
    \right)^{1/q^\star}.
\end{align*}
Raising both sides to the power $q^\star$, we obtain
\begin{align*}
    \|p_{\overline{m}^1_{t_0}}-p_{\overline{m}^2_{t_0}}\|_{L^1}^{q^\star}
    \le
    C
    T^{\frac{(2-q)q^\star}{2q}}
    \left(\frac{2}{2-q}\right)^{q^\star/q}
    \int_0^{t_0}
    \|p_{\overline{m}^1_t}-p_{\overline{m}^2_t}\|_{L^1}^{q^\star}\,\mathrm{d}t.
\end{align*}
Set
\[
\Delta(t):=\|p_{\overline{m}^1_t}-p_{\overline{m}^2_t}\|_{L^1}^{q^\star}.
\]
Then
\[
\Delta(t_0)\le C\int_0^{t_0}\Delta(t)\,\mathrm{d}t,\qquad t_0\in[0,T].
\]
By Gronwall's lemma, it follows that
\[
\Delta(t)=0,\qquad t\in[0,T].
\]
Equivalently,
\[
\|p_{\overline{m}^1_t}-p_{\overline{m}^2_t}\|_{L^1(\R^d\times[0,1])}=0,\qquad t\in[0,T].
\]
Hence
\[
p_{\overline{m}^1_t}=p_{\overline{m}^2_t}
\quad\text{for every }t\in[0,T],
\]
and therefore
\[
\overline{m}^1_t=\overline{m}^2_t,\qquad t\in[0,T].
\]
This proves the uniqueness of the Fokker--Planck weak solution.
\end{proof}

\end{appendix}

\bibliographystyle{plain}


\bibliography{MFG_moderate_Interaction-arxiv_version}

\begin{thebibliography}{42}
\providecommand{\natexlab}[1]{#1}
\providecommand{\url}[1]{\texttt{#1}}
\expandafter\ifx\csname urlstyle\endcsname\relax
  \providecommand{\doi}[1]{doi: #1}\else
  \providecommand{\doi}{doi: \begingroup \urlstyle{rm}\Url}\fi

\bibitem[Aliprantis and Border(2006)]{aliprantis2006infinite}
C.~Aliprantis and K.~Border.
\newblock \emph{Infinite dimensional analysis: a hitchhiker's guide}.
\newblock Springer--Verlag Berlin Heidelberg, 3rd edition, 2006.

\bibitem[Aronson and Serrin(1967)]{AronsonSerrin67}
D.~G. Aronson and J.~Serrin.
\newblock Local behavior of solutions of quasilinear parabolic equations.
\newblock \emph{Archive for Rational Mechanics and Analysis volume}, 25:\penalty0 81–122, 1967.
\newblock \doi{10.1137/S0363012996313549}.
\newblock URL \url{https://doi.org/10.1007/BF00281291}.

\bibitem[Aurell et~al.(2021)Aurell, Carmona, and Lauri{\`e}re]{Aurell2021StochasticGG}
A.~Aurell, R.~A. Carmona, and M.~Lauri{\`e}re.
\newblock Stochastic graphon games: {II}. the linear-quadratic case.
\newblock \emph{Applied Mathematics \& Optimization}, 85, 2021.
\newblock URL \url{https://api.semanticscholar.org/CorpusID:235195871}.

\bibitem[Bayraktar et~al.(2023{\natexlab{a}})Bayraktar, Chakraborty, and Wu]{ErhanGraphon2023}
E.~Bayraktar, S.~Chakraborty, and R.~Wu.
\newblock {Graphon mean field systems}.
\newblock \emph{The Annals of Applied Probability}, 33\penalty0 (5):\penalty0 3587 -- 3619, 2023{\natexlab{a}}.
\newblock \doi{10.1214/22-AAP1901}.
\newblock URL \url{https://doi.org/10.1214/22-AAP1901}.

\bibitem[Bayraktar et~al.(2023{\natexlab{b}})Bayraktar, Wu, and Zhang]{10.1007/s00245-023-09996-y}
E.~Bayraktar, R.~Wu, and X.~Zhang.
\newblock Propagation of chaos of forward–backward stochastic differential equations with graphon interactions.
\newblock \emph{Appl. Math. Optim.}, 88\penalty0 (1), 2023{\natexlab{b}}.
\newblock ISSN 0095-4616.
\newblock \doi{10.1007/s00245-023-09996-y}.
\newblock URL \url{https://doi.org/10.1007/s00245-023-09996-y}.

\bibitem[Bayraktar et~al.(2026)Bayraktar, He, and Kim]{bayraktar2026graphonparticlesystemscommon}
E.~Bayraktar, X.~He, and D.~Kim.
\newblock Graphon particle systems with common noise.
\newblock 2026.
\newblock URL \url{https://arxiv.org/abs/2507.03265}.

\bibitem[Caines and Huang(2020)]{Caines2020GraphonMF}
P.~E. Caines and M.~Huang.
\newblock Graphon mean field games and their equations.
\newblock \emph{SIAM J. Control. Optim.}, 59:\penalty0 4373--4399, 2020.

\bibitem[Cao and Laurière(2025)]{cao2025probabilisticanalysisgraphonmean}
Z.~Cao and M.~Laurière.
\newblock Probabilistic analysis of graphon mean field control.
\newblock 2025.
\newblock URL \url{https://arxiv.org/abs/2505.19664}.

\bibitem[Cardaliaguet(2017)]{Cardaliaguet2017}
P.~Cardaliaguet.
\newblock The convergence problem in mean field games with local coupling.
\newblock \emph{Applied Mathematics \& Optimization}, 76, 2017.
\newblock \doi{10.1007/s00245-017-9434-0}.
\newblock URL \url{https://doi.org/10.1007/s00245-017-9434-0}.

\bibitem[Carmona and Delarue(2018{\natexlab{a}})]{carmona2018probabilisticI}
R.~Carmona and F.~Delarue.
\newblock \emph{Probabilistic theory of mean field games with applications {I}}, volume~83 of \emph{Probability theory and stochastic modelling}.
\newblock Springer International Publishing, 2018{\natexlab{a}}.

\bibitem[Carmona and Delarue(2018{\natexlab{b}})]{carmona2018probabilisticII}
R.~Carmona and F.~Delarue.
\newblock \emph{Probabilistic theory of mean field games with applications {II}}, volume~84 of \emph{Probability theory and stochastic modelling}.
\newblock Springer International Publishing, 2018{\natexlab{b}}.

\bibitem[Carmona et~al.(2016)Carmona, Delarue, and Lacker]{Lacker_carmona_delarue_CN}
R.~Carmona, F.~Delarue, and D.~Lacker.
\newblock {Mean field games with common noise}.
\newblock \emph{The Annals of Probability}, 44\penalty0 (6):\penalty0 3740 -- 3803, 2016.
\newblock \doi{10.1214/15-AOP1060}.
\newblock URL \url{https://doi.org/10.1214/15-AOP1060}.

\bibitem[Coppini et~al.(2024)Coppini, Crescenzo, and Pham]{Coppini2024NonlinearGM}
F.~Coppini, A.~D. Crescenzo, and H.~Pham.
\newblock Nonlinear graphon mean-field systems.
\newblock \emph{Stochastic Processes and their Applications}, 2024.
\newblock URL \url{https://api.semanticscholar.org/CorpusID:267636698}.

\bibitem[Crescenzo et~al.(2024)Crescenzo, Fuhrman, Kharroubi, and Pham]{decrescenzo2024meanfieldcontrolnonexchangeable}
A.~D. Crescenzo, M.~Fuhrman, I.~Kharroubi, and H.~Pham.
\newblock Mean-field control of non exchangeable systems.
\newblock 2024.
\newblock URL \url{https://arxiv.org/abs/2407.18635}.

\bibitem[Crucianelli and Tangpi(2024)]{crucianelli2024interactingparticlesystemssparse}
C.~Crucianelli and L.~Tangpi.
\newblock Interacting particle systems on sparse {W}-random graphs.
\newblock 2024.
\newblock URL \url{https://arxiv.org/abs/2410.11240}.

\bibitem[Djete(2023{\natexlab{a}})]{closed-loop-MFG_MDF}
M.~F. Djete.
\newblock Large population games with interactions through controls and common noise: convergence results and equivalence between open-loop and closed-loop controls.
\newblock \emph{ESAIM: COCV}, 29:\penalty0 39, 2023{\natexlab{a}}.
\newblock \doi{10.1051/cocv/2023005}.
\newblock URL \url{https://doi.org/10.1051/cocv/2023005}.

\bibitem[Djete(2023{\natexlab{b}})]{djete2023stackelbergmeanfieldgames}
M.~F. Djete.
\newblock Stackelberg mean field games: convergence and existence results to the problem of principal with multiple agents in competition.
\newblock 2023{\natexlab{b}}.
\newblock URL \url{https://arxiv.org/abs/2309.00640}.

\bibitem[Djete(2025)]{djete2025nonexchangeablemeanfieldcontrol}
M.~F. Djete.
\newblock A non-exchangeable mean field control problem with controlled interactions.
\newblock 2025.
\newblock URL \url{https://arxiv.org/abs/2511.00288}.

\bibitem[Djete and Touzi(2024)]{10.1214/23-AAP1993}
M.~F. Djete and N.~Touzi.
\newblock {Mean field game of mutual holding}.
\newblock \emph{The Annals of Applied Probability}, 34\penalty0 (6):\penalty0 4999 -- 5031, 2024.
\newblock \doi{10.1214/23-AAP1993}.
\newblock URL \url{https://doi.org/10.1214/23-AAP1993}.

\bibitem[Djete et~al.(0)Djete, Possama{\"\i}, and Tan]{djete2019general}
M.~F. Djete, D.~Possama{\"\i}, and X.~Tan.
\newblock Mckean–vlasov optimal control: Limit theory and equivalence between different formulations.
\newblock \emph{Mathematics of Operations Research}, 0\penalty0 (0):\penalty0 null, 0.
\newblock \doi{10.1287/moor.2021.1232}.
\newblock URL \url{https://doi.org/10.1287/moor.2021.1232}.

\bibitem[Djete et~al.(2022)Djete, Possama{\"\i}, and Tan]{djete2019mckean}
M.~F. Djete, D.~Possama{\"\i}, and X.~Tan.
\newblock {McKean–Vlasov optimal control: The dynamic programming principle}.
\newblock \emph{The Annals of Probability}, 50\penalty0 (2):\penalty0 791 -- 833, 2022.
\newblock \doi{10.1214/21-AOP1548}.
\newblock URL \url{https://doi.org/10.1214/21-AOP1548}.

\bibitem[El~Karoui and Tan(2013)]{karoui2013capacities}
N.~El~Karoui and X.~Tan.
\newblock Capacities, measurable selection and dynamic programming part {I}: abstract framework.
\newblock \emph{arXiv preprint arXiv:1310.3363}, 2013.

\bibitem[El~Karoui et~al.(1997)El~Karoui, Peng, and Quenez]{el1997backward}
N.~El~Karoui, S.~Peng, and M.-C. Quenez.
\newblock Backward stochastic differential equations in finance.
\newblock \emph{Mathematical Finance}, 7\penalty0 (1):\penalty0 1--71, 1997.

\bibitem[Flandoli et~al.(2022)Flandoli, Ghio, and Livieri]{Flandoli2022}
F.~Flandoli, M.~Ghio, and G.~Livieri.
\newblock N-player games and mean field games of moderate interactions.
\newblock \emph{Applied Mathematics \& Optimization}, 85, 2022.
\newblock \doi{10.1007/s00245-022-09834-7}.
\newblock URL \url{https://doi.org/10.1007/s00245-022-09834-7}.

\bibitem[Gao et~al.(2020)Gao, Tchuendom, and Caines]{Gao2020LinearQG}
S.~Gao, R.~F. Tchuendom, and P.~E. Caines.
\newblock Linear quadratic graphon field games.
\newblock \emph{ArXiv}, abs/2006.03964, 2020.
\newblock URL \url{https://api.semanticscholar.org/CorpusID:219530959}.

\bibitem[Huang et~al.(2006)Huang, Malham{\'e}, and Caines]{huang2006large}
M.~Huang, R.~Malham{\'e}, and P.~Caines.
\newblock Large population stochastic dynamic games: closed--loop {M}c{K}ean--{V}lasov systems and the {N}ash certainty equivalence principle.
\newblock \emph{Communications in Information \& Systems}, 6\penalty0 (3):\penalty0 221--252, 2006.

\bibitem[Huang et~al.(2007)Huang, Caines, and Malham{\'e}]{huang2007nash}
M.~Huang, P.~Caines, and R.~Malham{\'e}.
\newblock The {N}ash certainty equivalence principle and {M}c{K}ean--{V}lasov systems: an invariance principle and entry adaptation.
\newblock In D.~Castanon and J.~Spall, editors, \emph{46th IEEE conference on decision and control, 2007}, pages 121--126. IEEE, 2007.

\bibitem[Jabin et~al.(2024)Jabin, Poyato, and Soler]{jabin2024meanfieldlimitnonexchangeablesystems}
P.-E. Jabin, D.~Poyato, and J.~Soler.
\newblock Mean-field limit of non-exchangeable systems.
\newblock 2024.
\newblock URL \url{https://arxiv.org/abs/2112.15406}.

\bibitem[Jourdain and M{\'e}l{\'e}ard(1998)]{jourdain1998propagation}
B.~Jourdain and S.~M{\'e}l{\'e}ard.
\newblock Propagation of chaos and fluctuations for a moderate model with smooth initial data.
\newblock \emph{Annales de l'institut Henri Poincar{\'e}, Probabilit{\'e}s et Statistiques $(${\rm B}$)$}, 34\penalty0 (6):\penalty0 727--766, 1998.

\bibitem[Kharroubi et~al.(2025)Kharroubi, Mekkaoui, and Pham]{kharroubi2025stochasticmaximumprincipleoptimal}
I.~Kharroubi, S.~Mekkaoui, and H.~Pham.
\newblock Stochastic maximum principle for optimal control problem of non exchangeable mean field systems.
\newblock 2025.
\newblock URL \url{https://arxiv.org/abs/2506.05595}.

\bibitem[Krylov(1980)]{krylov1980controlled}
N.~Krylov.
\newblock \emph{Controlled diffusion processes}, volume~14 of \emph{Stochastic modelling and applied probability}.
\newblock Springer--Verlag New York, 1980.

\bibitem[Kurtz(2014)]{kurtz2014weak}
G.~Kurtz, T.
\newblock Weak and strong solutions of general stochastic models.
\newblock \emph{Electronic Communications in Probability}, 19\penalty0 (58):\penalty0 1--16, 2014.

\bibitem[Lacker(2020)]{Lacker-closedloop}
D.~Lacker.
\newblock {On the convergence of closed-loop Nash equilibria to the mean field game limit}.
\newblock \emph{The Annals of Applied Probability}, 30\penalty0 (4):\penalty0 1693 -- 1761, 2020.
\newblock \doi{10.1214/19-AAP1541}.
\newblock URL \url{https://doi.org/10.1214/19-AAP1541}.

\bibitem[Lacker and Flem(2023)]{leflemclosed2023}
D.~Lacker and L.~L. Flem.
\newblock {Closed-loop convergence for mean field games with common noise}.
\newblock \emph{The Annals of Applied Probability}, 33\penalty0 (4):\penalty0 2681 -- 2733, 2023.
\newblock \doi{10.1214/22-AAP1876}.
\newblock URL \url{https://doi.org/10.1214/22-AAP1876}.

\bibitem[Lacker and Soret(2023)]{doi:10.1287/moor.2022.1329}
D.~Lacker and A.~Soret.
\newblock A label-state formulation of stochastic graphon games and approximate equilibria on large networks.
\newblock \emph{Mathematics of Operations Research}, 48\penalty0 (4):\penalty0 1987--2018, 2023.
\newblock \doi{10.1287/moor.2022.1329}.
\newblock URL \url{https://doi.org/10.1287/moor.2022.1329}.

\bibitem[Lacker et~al.(2020)Lacker, Shkolnikov, and Zhang]{Lacker-Shkolnikov-Zhang_2020}
D.~Lacker, M.~Shkolnikov, and J.~Zhang.
\newblock Superposition and mimicking theorems for conditional {M}ckean-{V}lasov equations.
\newblock \emph{arXiv preprint arXiv:2004.00099}, 2020.

\bibitem[Lasry and Lions(2007)]{lasry2007mean}
J.-M. Lasry and P.-L. Lions.
\newblock Mean field games.
\newblock \emph{Japanese Journal of Mathematics}, 2\penalty0 (1):\penalty0 229--260, 2007.

\bibitem[Lov{\'a}sz(2012)]{Lovsz2012LargeNA}
L.~M. Lov{\'a}sz.
\newblock Large networks and graph limits.
\newblock In \emph{Colloquium Publications}, 2012.
\newblock URL \url{https://api.semanticscholar.org/CorpusID:27057044}.

\bibitem[Ma and Zhang(2002)]{JinRepresentation2002}
J.~Ma and J.~Zhang.
\newblock {Representation theorems for backward stochastic differential equations}.
\newblock \emph{The Annals of Applied Probability}, 12\penalty0 (4):\penalty0 1390 -- 1418, 2002.
\newblock \doi{10.1214/aoap/1037125868}.
\newblock URL \url{https://doi.org/10.1214/aoap/1037125868}.

\bibitem[M{\'e}l{\'e}ard and Roelly-Coppoletta(1987)]{meleard1987propagation}
S.~M{\'e}l{\'e}ard and S.~Roelly-Coppoletta.
\newblock A propagation of chaos result for a system of particles with moderate interaction.
\newblock \emph{Stochastic Processes and their Applications}, 26:\penalty0 317--332, 1987.

\bibitem[Oelschl{\"a}ger(1985)]{oelschlager1985law}
K.~Oelschl{\"a}ger.
\newblock A law of large numbers for moderately interacting diffusion processes.
\newblock \emph{Zeitschrift f{\"u}r Wahrscheinlichkeitstheorie und verwandte Gebiete}, 69\penalty0 (2):\penalty0 279--322, 1985.

\bibitem[Tangpi and Zhou(2024)]{repec:spr:finsto:v:28:y:2024:i:2:d:10.1007_s00780-023-00527-9}
L.~Tangpi and X.~Zhou.
\newblock Optimal investment in a large population of competitive and heterogeneous agents.
\newblock \emph{Finance and Stochastics}, 28\penalty0 (2):\penalty0 497--551, April 2024.
\newblock \doi{10.1007/s00780-023-00527-9}.
\newblock URL \url{https://ideas.repec.org/a/spr/finsto/v28y2024i2d10.1007_s00780-023-00527-9.html}.

\end{thebibliography}




\end{document}